\documentclass[10pt]{elsarticle}

\usepackage{hyperref}
\usepackage{amsmath,amssymb,amsthm}
\usepackage{graphicx}
\usepackage{graphics}
\usepackage {graphicx,fancyhdr}
\usepackage{color}
\usepackage{flafter}
\numberwithin{equation}{section}
\usepackage{multirow}
\usepackage{mathrsfs} 
\usepackage{subfigure}
\usepackage{tabularx}
\usepackage{lineno}
\usepackage{xfrac}

\usepackage{stmaryrd} 
\usepackage{bm} 
\usepackage{cases} 

\newtheorem{myTheorem}{Theorem}[section]
\newtheorem{myLemma}[myTheorem]{Lemma}
\newtheorem{myRemark}[myTheorem]{Remark}

\newtheorem{myExam}{Example}{\normalfont}
\newtheorem*{myProof}{Proof} 

\begin{document}
	
 \begin{center}
\large\bf
A mixed discontinuous Galerkin method with symmetric stress
for\\ Brinkman problem based on the velocity-pseudostress formulation
\end{center}
		
\begin{center}
Yanxia Qian\footnote{School of Mathematics and Statistics, Xi'an
Jiaotong University, Xi'an, Shaanxi 710049, P. R. China. Email: {\tt
yxqian0520@163.com}},
\quad
Shuonan Wu\footnote{School of Mathematical Sciences, Peking
University, Beijing 100871, P. R. China. The work of this author is
partially supported by the startup grant from Peking University.
Email: {\tt snwu@math.pku.edu.cn}},
\quad
Fei Wang\footnote{School of Mathematics and
Statistics \& State Key Laboratory of Multiphase Flow in Power
Engineering, Xi'an Jiaotong University, Xi'an, Shaanxi 710049, P.  R.
China. The work of this author was partially supported by the National
Natural Science Foundation of China (Grant No.\ 11771350).  Email:
{\tt feiwang.xjtu@xjtu.edu.cn} (Corresponding author) 
}
\end{center}

	
\begin{quote}
{\bf Abstract.} The Brinkman equations can be regarded as a
combination of the Stokes and Darcy equations which model transitions
between the fast flow in channels (governed by Stokes equations) and
the slow flow in porous media (governed by Darcy's law). The numerical
challenge for this model is the designing of a numerical scheme which
is stable for both the Stokes-dominated (high permeability) and the
Darcy-dominated (low permeability) equations.  In this paper, we solve
the Brinkman model in $n$ dimensions ($n = 2, 3$) by using  the mixed
discontinuous Galerkin (MDG) method, which meets this challenge.  This
MDG method is based on the pseudostress-velocity formulation and uses
a discontinuous piecewise polynomial pair
$\underline{\bm{\mathcal{P}}}_{k+1}^{\mathbb{S}}$-$\bm{\mathcal{P}}_k$
$(k \geq 0)$, where the stress field is symmetric. The main unknowns
are the pseudostress and the velocity, whereas the pressure is easily
recovered through a simple postprocessing.  A key step in the analysis
is to establish the parameter-robust inf-sup stability through
specific parameter-dependent norms at both continuous and discrete
levels. Therefore, the stability results presented here are uniform
with respect to the permeability.  Thanks to the parameter-robust
stability analysis, we obtain optimal error estimates for the stress
in broken $\underline{\bm{H}}(\bm{{\rm div}})$-norm and velocity in
$\bm{L}^2$-norm. Furthermore, the optimal $\underline{\bm{L}}^2$ error
estimate for pseudostress is derived under certain conditions.
Finally, numerical experiments are provided to support the theoretical
results and to show the robustness, accuracy, and flexibility of the
MDG method.


\end{quote}

	

\section{Introduction}\label{Introduction}

The Brinkman equations (cf. \cite{Brinkman1949A}),
\begin{subequations}\label{Brinkman0}
	\begin{align}
	\label{Brinkman_eq0}
  \nu\bm{\underline{\bm{\kappa}}}^{-1}\bm{u} - 2\nu
  \mathbf{div}(\underline{\bm{\varepsilon}}(\bm{u})) + \nabla p &=
  \bm{f}    &\text{in } \Omega,\\
	\label{divergencefree}
	{\rm div}\bm{u}&= 0    &\text{in } \Omega,
	\end{align}
\end{subequations}
which model the flow of a fluid through a complex porous medium
occupying domain $\Omega$ with a high-contrast permeability tensor
$\underline{\bm{\kappa}}$, can be seen as a mixture of Darcy and
Stokes equations. Here, $\nu >0$ is the fluid viscosity, $\bm{u}$
denotes the velocity field,
$\underline{\bm{\varepsilon}}(\bm{u})=(\underline{\bm{\nabla}}\bm{u} +
(\underline{\bm{\nabla}}\bm{u})^t)/2$
is the strain rate, $p$ is the pressure field, and $\bm{f}$ is the
volume force. This model arises from applications in many fields, such
as groundwater hydrology, biomedical engineering, petroleum industry,
and environmental science (cf. \cite{wehrspohn2005ordered,
Ligaarden2010Stokes, vafai2010porous}). From \eqref{Brinkman0}, we can
see that the Brinkman problem becomes Stokes-dominated when the
permeability tensor $\underline{\bm{\kappa}}$ is getting large, and it
becomes Darcy-dominated when $\underline{\bm{\kappa}}$ is quite small.

It is well known that the usual Darcy stable element pairs may diverge
for Stokes flow and vice versa. Therefore, the numerical challenge for
solving the Brinkman model is to construct a stable discretization
method for both the Stokes and the Darcy equations. As shown in
\cite{Mardal2002A}, the Darcy stable finite element pairs, for
example, the Raviart-Thomas element (cf. \cite{Raviart1977mixed})
leads to non-convergent results as the Brinkman model becomes
Stokes-dominated; on the other hand, the usual Stokes stable finite
element pairs, such as Mini-element (cf. \cite{Arnold1984A}),
$\bm{\mathcal{P}_2}$-$\mathcal{P}_0$ element, nonconforming
Crouzeix-Raviart (CR) finite element (cf.
\cite{Crouzeix1973Conforming}), will diverge when the Brinkman
equations turn into Darcy-dominated.  Therefore, many researchers pay
attention to developing stable and accurate numerical methods for
solving Brinkman equations. One way to circumvent this difficulty is
to modify the existing Stokes or Darcy elements to make them work well
for the Brinkman model.  In \cite{Burman2005Stabilized}, inspired by
discontinuous Galerkin (DG) method, a stabilized CR finite element
method is constructed by adding a penalty term. In addition, a
generalization of classical Mini-element is studied in
\cite{Juntunen2010Analysis}; stabilized equal-order finite elements
are proposed and analyzed in \cite{Braack2011Equal};  (hybridized)
interior penalty DG scheme with $\bm{H}(\rm div)$-conforming finite
elements is investigated in \cite{Juho2011H}. Another approach is to
develop new numerical schemes for solving Brinkman equations, for
examples, pseudostress-based mixed finite element methods (cf.
\cite{Barrios2012On,Gatica2014Analysis}), weak Galerkin methods (cf.
\cite{Lin2014A, Zhai2016A}), virtual element method (cf.
\cite{C2017A}), hybrid high-order method (cf. \cite{Botti2018A}) and
hybridizable discontinuous Galerkin method (cf.
\cite{Gatica2018Apriori}).

To study hydrodynamics, different formulations, like
velocity-pressure, stress-velocity-pressure, pseudostress-velocity
formulations, have been introduced and analyzed. The velocity-pressure
formulation has been extensively studied in the computation of
incompressible Newtonian flows (cf.
\cite{Brezzi1991Mixed,Boffi2013Mixed}). However, the study of
numerical methods for the stress-based and pseudostress-based
formulations
(cf. \cite{Farhloul1993A,Marek1993Stabilized,Cai2010Mixed,Gatica2010Analysis})
has become a very active research area because of the arising interest
in non-Newtonian flows. The main advantage of the stress-based and
pseudostress-based formulation is that it provides a unified framework
for both the Newtonian and the non-Newtonian flows.  In addition,
physical quantity like the stress can be computed directly instead of
by taking derivatives of the velocity, which avoids degrading of
accuracy in the process of numerical differentiation. Precise
computation of the stress is of paramount importance for the hydraulic
fracturing problem as the crack propagation is determined by the
stress field. While a formulation comprising the stress as a
fundamental unknown is unavoidable for non-Newtonian flows in which
the constitutive law is nonlinear, the drawback of the
stress-velocity-pressure formulation is the increase in the number of
unknowns. To avoid this disadvantage, we focus on the
pseudostress-velocity formulation. Last but not least, we need to
mention that the pressure field can be easily obtained by a simple
postprocessing without affecting the accuracy of the approximation.

Due to the flexibility in constructing the local shape function spaces
and the ability to capture non-smooth or oscillatory solutions
effectively, DG methods have been applied to solve many problems in
scientific computing and engineering, such as conservation laws
(cf. \cite{Bey1995hp,Chen2017Entropy}), Darcy flow
(cf. \cite{Brezzi2005Mixed,Aizinger2018Analysis}), Navier-Stokes (or
Stokes) equations (cf. \cite{Bassi1997A,Kaya2005A}), variational
inequalities (cf.
\cite{Wang2010Discontinuous,Brenner2012A,Wang2014Discontinuous}) and
much more. Besides, DG methods also enjoy the following advantages:
(i) locally (and globally) conservative; (ii) easy to implement $hp$
adaptivity; (iii) suitable for parallel computing. We refer to
\cite{Cockburn2000Discontinuous,Brezzi2000Discontinuous,Douglas2002Unified,Hong2019A}
for more discussion about DG methods.

In this paper, we construct a mixed discontinuous Galerkin (MDG)
method with
$\underline{\bm{\mathcal{P}}}_{k+1}^{\mathbb{S}}$-$\bm{\mathcal{P}}_k$
element pair for solving the Brinkman equations based on the
pseudostress-velocity formulation. The main results of this article
include that: (i) The MDG scheme with symmetric stress field is
uniformly stable and efficient for both Darcy-dominated and
Stokes-dominated flows; (ii) Under specific parameter-dependent norms,
the parameter-robust stability results of both continuous and discrete
schemes are obtained; (iii) For $k \geq 0$, we get the optimal
convergence order for the stress in broken
$\underline{\bm{H}}(\bm{{\rm div}})$-norm and velocity in
$\bm{L}^2$-norm; (iv) When $k\geq n$ and the Stokes pair
$\bm{\mathcal{P}}_{k+2}$-$\mathcal{P}_{k+1}$ is stable, we obtain the
optimal $\underline{\bm{L}}^2$ error estimate for the pseudostress.

The rest of the paper is organized as follows. In Section
\ref{Preliminaries}, we introduce the pseudostress-velocity
formulation for the Brinkman model and present some preliminary
results. In Section \ref{MDGdiscretizations}, the MDG scheme is
introduced and the well-posedness is obtained.  We show the stability
of the discrete scheme and prove optimal error estimates for both
velocity and pressure in Section \ref{errorestimation}. In Section
\ref{numericalexamples}, numerical examples are provided to confirm
the theoretical findings and to illustrate the performance of the
mixed DG scheme. Finally, we give a short summary in Section \ref{summary}.


\section{Brinkman model in pseudostress-velocity
formulation}\label{Preliminaries}

In this section, we introduce the Brinkman model in the
pseudostress-velocity formulation and provide the parameter-robust
stability analysis of the continuous problem. First, we give the
notation.

\subsection{Notation}\label{Notation}

Given $n$ ($n = 2$ or $3$), we denote the space of real matrices of
order $n\times n$ by $\mathbb{M}$, and define $\mathbb{S} \subset
\mathbb{M}$ as the space of real symmetric matrices.  For matrices
$\underline{\bm{\tau}}=(\tau_{ij}) \in \mathbb{M}$ and
$\underline{\bm{\zeta}}=(\zeta_{ij}) \in \mathbb{M}$, we write as
usual
\begin{equation}\label{sec2_eq}
\underline{\bm{\tau}}^t = (\tau_{ji}), \quad
{\rm tr}(\underline{\bm{\tau}}) = \sum_{i=1}^n\tau_{ii}, \quad
\underline{\bm{\tau}}^d =  \underline{\bm{\tau}}- \frac{1}{n}{\rm
tr}(\underline{\bm{\tau}})\underline{\bm{I}}, \quad
\underline{\bm{\tau}}:\underline{\bm{\zeta}} =
\sum_{i,j=1}^n\tau_{ij}\zeta_{ij},
\end{equation}
where $\underline{\bm{I}}$ is the identity matrix.

For a subdomain $D \subset \mathbb{R}^{n}$ and integer $m\geq0$, we
denote the scalar-valued Sobolev spaces by $H^m(D)=W^{m,2}(D)$ with
the norm $\|\cdot\|_{m,D}$ and seminorm $|\cdot|_{m,D}$. When $m = 0$,
$H^0(D)$ coincides with the Lebesgue spaces $L^2(D)$, which is
equipped with the usual $L^2$-inner product $(\cdot,\cdot)_D$ and
$L^2$-norm $||\cdot||_{0,D}$. The $L^2$-inner product (or duality
pairing) on $\partial D$ is denoted by
$\langle\cdot,\cdot\rangle_{\partial D}$.  We denote the vector-valued
spaces, tensor-valued function spaces and symmetric-tensor-valued
spaces whose entries are in $H^m(D)$ by $\bm{H}^m(D)$,
$\underline{\bm{H}}^m(D)$ and $\underline{\bm{H}}^{m}(D;\mathbb{S})$,
respectively. In particular, $\bm{H}^0(D) = \bm{L}^2(D)$,
$\underline{\bm{H}}^0(D) = \underline{\bm{L}}^2(D)$ and
$\underline{\bm{H}}^0(D;\mathbb{S}) =
\underline{\bm{L}}^2(D;\mathbb{S})$.  Then, we introduce the following
space
\[
\underline{\bm{H}}( \bm{{\rm
div}},D;\mathbb{M}) = \{\underline{\bm{\tau}} \in
\underline{\bm{L}}^{2}(D;\mathbb{M}):\mathbf{div}\underline{\bm{\tau}} \in
\bm{L}^2(D)\},
\]
equipped with the norm $\|\underline{\bm{\tau}}\|_{\bm{{\rm
div}},D} = (\|\underline{\bm{\tau}}\|^2_{0,D} +
\|\mathbf{div}\underline{\bm{\tau}}\|^2_{0,D})^{1/2}$
for all $\underline{\bm{\tau}}\in\,\underline{\bm{H}}( \bm{{\rm
div}};\mathbb{M})$.  Here, differential operators are applied row by
row, i.e., the $i$-th row of $\bm{{\rm div}}\underline{\bm{\sigma}}$
is the divergence of the $i$-th row vector of the matrix
$\underline{\bm{\sigma}}$. Similarly, the $i$-th row of the matrix
$\underline{\bm{\nabla}}\bm{u}$ in the definition of
$\underline{\bm{\varepsilon}}(\bm{u})$ is the gradient (written as a
row) of the $i$-th component of the vector $\bm{u}$. We also define 
$
\underline{\bm{H}}( \bm{{\rm
div}},D;\mathbb{S}) = \{\underline{\bm{\tau}} \in
\underline{\bm{L}}^{2}(D;\mathbb{S}):\mathbf{div}\underline{\bm{\tau}} \in
\bm{L}^2(D)\}
$.

In the present context, Green's formula takes the form
\begin{equation} \label{Green-formula}
(\underline{\bm{\varepsilon}}(\bm{v}),\underline{\bm{\tau}})_D
= -(\mathbf{div}\underline{\bm{\tau}},\bm{v})_D +
\langle\underline{\bm{\tau}}\bm{n}_D,\bm{v}\rangle_{\partial D},
\end{equation}
where $\bm{n}_D$ is the exterior unit normal to $\partial D$. If $D$
is chosen as $\Omega$, we abbreviate it by using $(\cdot,\cdot)$ and
$\langle\cdot,\cdot\rangle$, and similar rule follows for the spaces
and norms mentioned above.

\subsection{Brinkman model}\label{Brinkman_model}

Let $\Omega\subset\mathbb{R}^{n}$ be a bounded and simply connected
polygonal domain with Lipschitz boundary $\Gamma$. In this paper, we
consider the permeability of the form $\underline{\bm{\kappa}} =
\kappa \underline{\bm{I}}$, with the purpose of facilitating
parameter-robust stability analysis. We could also take $\nu = 1$ by a
non-dimensionalization procedure (see Remark
\ref{rmk:non-dimensionalization} below). With these simplifications, we find
that for the unique solution ($\bm{u}$, $p$) of the Brinkman model \eqref{Brinkman0}, ($\underline{\bm{\sigma}}$, $\bm{u}$, $p$) solves the equations
\begin{subequations}\label{Brinkman}
	\begin{align}
	\label{Brinkman_eq1}
	&\underline{\bm{\sigma}}=2\underline{\bm{\varepsilon}}(\bm{u})-p\underline{\bm{I}} \qquad  &\text{in } \Omega,\\
	\label{Brinkman_eq2}
  &\kappa^{-1}\bm{u} - \mathbf{div}\underline{\bm{\sigma}} = \bm{f}
  &\text{in } \Omega,\\
	\label{Brinkman-eq3}
	&{\rm div}\bm{u} = 0    &\text{in } \Omega,\\	
	\label{Brinkman_eq4}
	&\bm{u}= \bm{g}   &\text{on } \Gamma,\\
	\label{Brinkman_eq5}
	&\int_{\Omega}pd\bm{x}= 0.
	\end{align}
\end{subequations}
Additionally, due to the incompressibility condition, we assume that
$\bm{g}$ satisfies the compatibility condition
$\int_{\Gamma}\bm{g}\cdot\bm{n}ds =0$, where $\bm{n}$ stands for the
unit outward normal on $\Gamma$.

\begin{myRemark} \label{rmk:non-dimensionalization}
If $\nu \neq 1$, by taking
$\underline{\widetilde{\bm{\sigma}}} = \underline{\bm{\sigma}}/\nu$,
$\widetilde{\bm{f}} = \bm{u}/{\nu}$ and $\widetilde{p} = p/{\nu}$,
we could eliminate the parameter $\nu$, i.e.
\begin{align*}
&\underline{\widetilde{\bm{\sigma}}}=2\underline{\bm{\varepsilon}}(\bm{u})-\widetilde{p}\underline{\bm{I}}
\qquad  &\text{in } \Omega,\\
  &\kappa^{-1}\bm{u} - \mathbf{div}\underline{\widetilde{\bm{\sigma}}}
  = \widetilde{\bm{f}}    &\text{in } \Omega,\\
	&{\rm div}\bm{u} = 0    &\text{in } \Omega,\\	
	&\bm{u}= \bm{g}   &\text{on } \Gamma,\\
	&\int_{\Omega}\widetilde{p}d\bm{x}= 0.
\end{align*}
As a result, we can get the same conclusions with the problem
\eqref{Brinkman}.
\end{myRemark}


As described in Section \ref{Introduction}, in order to keep the
strengths and improve the weaknesses of the stress-velocity-pressure
formulation, by the incompressible condition, the problem
\eqref{Brinkman} can be rewritten equivalently as the
pseudostress-velocity formulation  (cf. \cite{Gatica2018Apriori}):
\begin{subequations}\label{Brinkman1}
	\begin{align}
	\label{Brinkman1_eq1}
  &\underline{\bm{\sigma}}^d=2\underline{\bm{\varepsilon}}(\bm{u}) &
  \text{in } \Omega,\\
  \label{Brinkman1_eq2} &\kappa^{-1}\bm{u} -
  \mathbf{div}\underline{\bm{\sigma}} = \bm{f}  &\text{in } \Omega,\\
	\label{Brinkman1_eq3} 	
	&\bm{u}=\bm{g} &\text{on } \Gamma,\\
	\label{Brinkman1_eq4}
	&\int_{\Omega}{\rm tr}(\underline{\bm{\sigma}})d\bm{x}= 0, &
	\end{align}
\end{subequations}
where the pressure $p$ can be obtained by the postprocessing formula
\begin{equation}\label{P}
p = -\frac{1}{n}{\rm tr}(\underline{\bm{\sigma}})  \qquad\qquad\qquad
\text{in } \Omega.
\end{equation}
There are two reasons for eliminating the pressure. An obvious one is
to reduce one variable and, hence, many degrees of freedom in the
discrete system. A more important reason is that we can use economic
and accurate stable elements and develop fast solvers for the
resulting discrete system so that computational cost will be greatly
reduced.

Set $\underline{\bm{\Sigma}} = \{\underline{\bm{\tau}} \in
\underline{\bm{H}}( \bm{{\rm div}};\mathbb{S}): \int_{\Omega}{\rm
tr}(\underline{\bm{\tau}})d\bm{x}= 0 \}$ and $\bm{V} =
\bm{L}^2(\Omega)$. Then, the variational formulation of
\eqref{Brinkman1} reads as follows: given $\bm{f} \in
\bm{L}^2(\Omega)$ and $\bm{g} \in \bm{H}^{1/2}(\Gamma)$, find
$(\underline{\bm{\sigma}}, \bm{u}) \in \underline{\bm{\Sigma}}\times
\bm{V}$ such that
\begin{subequations}\label{weak}
	\begin{align}
	\label{weak_eq1}
  a(\underline{\bm{\sigma}},\underline{\bm{\tau}}) +
  b(\underline{\bm{\tau}},\bm{u})
  &=\langle\underline{\bm{\tau}}\bm{n},\bm{g}\rangle_{\Gamma} &\forall
  \underline{\bm{\tau}} \in \underline{\bm{\Sigma}},\\
	\label{weak_eq2}
  b(\underline{\bm{\sigma}},\bm{v}) - s(\bm{u},\bm{v}) &=
  -(\bm{f},\bm{v})  &\forall \bm{v} \in \bm{V}.
	\end{align}
\end{subequations}
Here, the bilinear forms are defined by
\[ 
a(\underline{\bm{\sigma}},\underline{\bm{\tau}})=
\frac{1}{2}(\underline{\bm{\sigma}}^d,\underline{\bm{\tau}}^d),\quad
  b(\underline{\bm{\tau}},\bm{v}) =
  (\mathbf{div}\underline{\bm{\tau}},\bm{v})\quad {\rm and} \quad
  s(\bm{u},\bm{v})= (\kappa^{-1}\bm{u},\bm{v}). 
\]	

Notice that, by Green's formula, the equation \eqref{weak_eq1}
contains both the equation
$2\underline{\bm{\varepsilon}}(\bm{u})=\underline{\bm{\sigma}}^d$ in
$\Omega$ and the boundary condition $\bm{u}=\bm{g}$, with in
particular the incompressibility condition ${\rm div}\,\bm{u}=
\frac{1}{2}{\rm tr}(\underline{\bm{\sigma}}^d)=0$.

Furthermore, by the definition of $\underline{\bm{\tau}}^d$, it is easy to check that
\begin{align}
\label{sec2_eq2_1}
\|\underline{\bm{\tau}}\|_0^2 = \|\underline{\bm{\tau}}^d\|_0^2+\frac{1}{n}\|{\rm tr}(\underline{\bm{\tau}})\|_0^2, \\
\label{sec2_eq2_2}
\|{\rm tr}(\underline{\bm{\tau}})\|_0 \leq\sqrt{n}\|\bm{\tau}\|_0.
\end{align}

Throughout the paper, we use the abbreviation $ x \lesssim y $ ($
x\gtrsim y $) for the inequality $ x \leq Cy$ ($ x \geq Cy $), where
the letter $C$ denotes a positive constant independent of the
parameters $\kappa$, $\nu$, and the mesh size $h$, and may stand for
different values at its different occurrences.

\subsection{Well-posedness of the continuous
problem} \label{Well_posedness_continuous_problem}

Due to the large variation of the permeability tensor, in order to
show that our analysis is independent of the parameters $\kappa$,
$\underline{\bm{\Sigma}}$ and $ \bm{V}$ are endowed with the norm 
\begin{subequations}\label{Cnorms}
\begin{align}
\label{Cnorm-Sigma}
\|\underline{\bm{\tau}}\|_{\underline{\bm{\Sigma}}}^2 &=
(\underline{\bm{\tau}}^d,\underline{\bm{\tau}}^d) +
\widetilde{\kappa}(\mathbf{div}\underline{\bm{\tau}},\mathbf{div}\underline{\bm{\tau}})
& \forall \underline{\bm{\tau}} \in \underline{\bm{\Sigma}},\\
\label{Cnorm-V}
\|\bm{v}\|_{\bm{V}}^2 &= \widetilde{\kappa}^{-1}(\bm{v},\bm{v}) &
\forall \bm{v} \in \bm{V},
\end{align}
\end{subequations}
where $\widetilde{\kappa} = \min\{\kappa,1\}$. We note that
$\|\cdot\|_{\underline{\bm{\Sigma}}}$ is indeed a norm due to the fact
that $\|\underline{\bm{\tau}}\|_0^2 \lesssim
\|\underline{\bm{\tau}}^d\|_0^2 +
\|\mathbf{div}\underline{\bm{\tau}}\|_0^2$ for all
$\underline{\bm{\tau}} \in \underline{\bm{\Sigma}}$ (cf.
\cite{Brezzi1991Mixed,Boffi2013Mixed}).

We introduce a new bilinear form
$A((\underline{\bm{\sigma}},\bm{u}),(\underline{\bm{\tau}},\bm{v}))$
on $(\underline{\bm{\Sigma}}\times\bm{V},\underline{\bm{\Sigma}}\times\bm{V})$,
i.e.
\begin{equation}\label{sec2_eq3}	
A((\underline{\bm{\sigma}},\bm{u}),(\underline{\bm{\tau}},\bm{v})) 
= a(\underline{\bm{\sigma}},\underline{\bm{\tau}}) +
b(\underline{\bm{\tau}},\bm{u})-b(\underline{\bm{\sigma}},\bm{v})+s(\bm{u},\bm{v}).
\end{equation}
Then, the problem \eqref{weak} can be transformed into the following
problem: Find $(\underline{\bm{\sigma}},\bm{u}) \in
\underline{\bm{\Sigma}}\times\bm{V}$, such that
\begin{equation}\label{weakBrinkman}	
A((\underline{\bm{\sigma}},\bm{u}),(\underline{\bm{\tau}},\bm{v}))=
F((\underline{\bm{\tau}},\bm{v})),
\end{equation}
where $F((\underline{\bm{\tau}},\bm{v}))=(\bm{f},\bm{v})+
\langle\underline{\bm{\tau}}\bm{n},\bm{g}\rangle_{\Gamma}$.  By
Cauchy-Schwarz inequality, we have the boundedness of
$A(\cdot,\cdot)$.
\begin{myLemma}\label{sec2_Lemma1} 
The bilinear form $A(\cdot,\cdot)$ satisfies
\begin{equation}\label{sec2_eq4}
A((\underline{\bm{\sigma}},\bm{u}),(\underline{\bm{\tau}},\bm{v}))
\lesssim
(\|\underline{\bm{\sigma}}\|_{\underline{\bm{\Sigma}}}+\|\bm{u}\|_{\bm{V}})(\|\underline{\bm{\tau}}\|_{\underline{\bm{\Sigma}}}+\|\bm{v}\|_{\bm{V}})
\quad \ \forall
(\underline{\bm{\sigma}},\bm{u}),(\underline{\bm{\tau}},\bm{v}) \in
\underline{\bm{\Sigma}}\times\bm{V}.
\end{equation}
\end{myLemma}

Next, we show the inf-sup condition of $A(\cdot, \cdot)$ at continuous
level. 
\begin{myLemma}\label{sec2_Lemma2} 
For any
$(\underline{\bm{\sigma}},\bm{u}),(\underline{\bm{\tau}},\bm{v}) \in
\underline{\bm{\Sigma}}\times\bm{V}$, we have
\begin{equation}\label{sec2_eq5}
\inf_{(\underline{\bm{\sigma}},\bm{u})\in \underline{\bm{\Sigma}}\times\bm{V}}\sup_{(\underline{\bm{\tau}},\bm{v})\in \underline{\bm{\Sigma}}\times\bm{V}}\frac{A((\underline{\bm{\sigma}},\bm{u}),(\underline{\bm{\tau}},\bm{v}))}{
(\|\underline{\bm{\sigma}}\|_{\underline{\bm{\Sigma}}}+\|\bm{u}\|_{\bm{V}})
(\|\underline{\bm{\tau}}\|_{\underline{\bm{\Sigma}}}+\|\bm{v}\|_{\bm{V}})}\gtrsim 1.
\end{equation}
\end{myLemma}
\begin{proof} For any $\bm{u} \in \bm{V}$, there exists
${\underline{\bm{\sigma}}}^* \in \underline{\bm{H}}( \bm{{\rm
div}};\mathbb{S})$ and a positive constant $C_0 > 0$ such that (cf.
\cite{Arnold1984A2,Brezzi1991Mixed})
\[
\mathbf{div}{\underline{\bm{\sigma}}}^* = \bm{u} ~~{\rm in} \ \Omega,
\quad {\rm and}\quad
\|{\underline{\bm{\sigma}}}^*\|_{\mathbf{div}}\leq C_0\|\bm{u}\|_0.
\]
Set $\gamma = \frac{1}{|\Omega|}
\int_{\Omega}\mathrm{tr}({\underline{\bm{\sigma}}}^*)d\bm{x}$ and
$\widetilde{\underline{\bm{\sigma}}} =
{\underline{\bm{\sigma}}}^*-\frac{\gamma}{n}\underline{\bm{I}}$.
Then, it is straightforward to show that
\begin{equation}\label{sec2_eq6}
\widetilde{\underline{\bm{\sigma}}} \in \underline{\bm{\Sigma}}, \quad
\mathbf{div}\widetilde{\underline{\bm{\sigma}}} = \bm{u} ~~{\rm in}
\ \Omega, \quad {\rm and}\quad
\|\widetilde{\underline{\bm{\sigma}}}\|_{\mathbf{div}} \leq
C_0\|\bm{u}\|_0.
\end{equation}
We take $\underline{\bm{\tau}} = \underline{\bm{\sigma}} +
\alpha\widetilde{\underline{\bm{\sigma}}}$ and $\bm{v} =
\delta_1\bm{u}-\delta_2\mathbf{div}\underline{\bm{\sigma}}$, where the
non-negative coefficients $\alpha$, $\delta_1$ and $\delta_2$ will be
specified later. Then, according to Cauchy-Schwarz inequality and
\eqref{sec2_eq6}, one gets
\begin{align*}
A((\underline{\bm{\sigma}},\bm{u}),(\underline{\bm{\tau}},\bm{v}))&=\frac{1}{2}(\underline{\bm{\sigma}}^d,\underline{\bm{\sigma}}^d
    + \alpha\widetilde{\underline{\bm{\sigma}}}^d) +
(\mathbf{div}(\underline{\bm{\sigma}} +
              \alpha\widetilde{\underline{\bm{\sigma}}}),\bm{u}) -
(\mathbf{div}\underline{\bm{\sigma}},\delta_1\bm{u}-
 \delta_2\mathbf{div}\underline{\bm{\sigma}}) +
(\kappa^{-1}\bm{u},\delta_1\bm{u}-\delta_2\mathbf{div}\underline{\bm{\sigma}})
\\
&=\frac{1}{2}\|\underline{\bm{\sigma}}^d\|^2+\frac{\alpha}{2}(\underline{\bm{\sigma}}^d,\widetilde{\underline{\bm{\sigma}}})+
(1-\delta_1-\kappa^{-1}\delta_2)(\mathbf{div}\underline{\bm{\sigma}},\bm{u})+(\alpha+\kappa^{-1}\delta_1)\|\bm{u}\|_0^2+\delta_2\|\mathbf{div}\underline{\bm{\sigma}}\|_0^2
\\
&\geq(\frac{1}{2}-\frac{\alpha\varepsilon}{4})\|\underline{\bm{\sigma}}^d\|^2+
(1-\delta_1-\kappa^{-1}\delta_2)(\mathbf{div}\underline{\bm{\sigma}},\bm{u})+(\alpha+\kappa^{-1}\delta_1-\frac{\alpha
C_0^2}{4\varepsilon})\|\bm{u}\|_0^2+\delta_2\|\mathbf{div}\underline{\bm{\sigma}}\|_0^2.	
\end{align*}
From above inequality, let $\delta_1 = 1-
\frac{\widetilde{\kappa}}{2\kappa}$, $\delta_2 =
\frac{\widetilde{\kappa}}{2}$, $\varepsilon = \frac{C_0^2}{2}$ and
$\alpha=\frac{2}{C_0^2}$. Then, it holds that 
\begin{align*}
&\frac{1}{2}-\frac{\alpha\varepsilon}{4}=\frac{1}{4},\quad
1-\delta_1-\kappa^{-1}\delta_2=0,\\
&\delta_2 =
\frac{\widetilde{\kappa}}{2}\gtrsim\widetilde{\kappa},\quad
\alpha+\kappa^{-1}\delta_1-\frac{\alpha C_0^2}{4\varepsilon} =
\frac{1}{C_0^2} + \kappa^{-1} (1-\frac{\widetilde{\kappa}}{2\kappa})
\gtrsim \widetilde{\kappa}^{-1}.
\end{align*}
Here, we use the fact that $\frac{1}{2} \leq \delta_1 =
1-\frac{\widetilde{\kappa}}{2\kappa} < 1$ due to the definition of
$\widetilde{\kappa}$.  Then, we obtain
\[
A((\underline{\bm{\sigma}},\bm{u}),(\underline{\bm{\tau}},\bm{v}))
\gtrsim(\|\underline{\bm{\sigma}}\|_{\underline{\bm{\Sigma}}}+\|\bm{u}\|_{\bm{V}})^2.
\]
Next, taking $\alpha$, $\delta_1$ and $\delta_2$ in
$\underline{\bm{\tau}}$ and $\bm{v}$, by \eqref{sec2_eq6} and the fact that $\widetilde{\kappa} \leq 1 \leq \widetilde{\kappa}^{-1}$, one finds
\begin{align*}
\|\underline{\bm{\tau}}\|_{\underline{\bm{\Sigma}}}^2&=
\|\underline{\bm{\sigma}}^d +
\alpha\widetilde{\underline{\bm{\sigma}}}^d\|_0^2 + \widetilde{\kappa}
(\mathbf{div}\underline{\bm{\sigma}} + \alpha
 \mathbf{div}\widetilde{\underline{\bm{\sigma}}},
 \mathbf{div}\underline{\bm{\sigma}} + \alpha
 \mathbf{div}\widetilde{\underline{\bm{\sigma}}})\\
&\lesssim \|\underline{\bm{\sigma}}\|_{\underline{\bm{\Sigma}}}^2 +
\|\widetilde{\underline{\bm{\sigma}}}^d\|_0^2+
\widetilde{\kappa}(\bm{u}, \bm{u}) \\
&\lesssim(\|\underline{\bm{\sigma}}\|_{\underline{\bm{\Sigma}}}+\|\bm{u}\|_{\bm{V}})^2,\\
\|\bm{v}\|_{\bm{V}}^2&=\widetilde{\kappa}^{-1}\|\delta_1\bm{u}-\delta_2\mathbf{div}\underline{\bm{\sigma}}\|_0^2\\
&\lesssim \widetilde{\kappa}^{-1}(\bm{u},\bm{u})
+\widetilde{\kappa}^{-1}(\widetilde{\kappa}\mathbf{div}\underline{\bm{\sigma}},\widetilde{\kappa}\mathbf{div}\underline{\bm{\sigma}})\\
&\lesssim(\|\bm{u}\|_{\bm{V}}+\|\underline{\bm{\sigma}}\|_{\underline{\bm{\Sigma}}})^2.
\end{align*}
Then, we finish this proof.	
\end{proof}

From Lemma \ref{sec2_Lemma1} and Lemma \ref{sec2_Lemma2}, we get the
well-posedness of the problem \eqref{weakBrinkman}.
\begin{myTheorem}\label{sec2_Lemma3} Given $\bm{f} \in
\bm{L}^2(\Omega)$ and $\bm{g} \in \bm{H}^{1/2}(\Gamma)$, the problem
\eqref{weakBrinkman} has a unique solution
$(\underline{\bm{\sigma}},\bm{u}) \in
\underline{\bm{\Sigma}}\times\bm{V}$.
\end{myTheorem}


\section{MDG method}\label{MDGdiscretizations} 

In this section, we formulate the MDG method for the Brinkman problem
in the pseudostress-velocity formulation and show that it has a unique
solution.

\subsection{Derivation of the MDG scheme}\label{Derivation_MDG}

Let $\{\mathcal {T}_h\}_h$ be a family of quasi-regular decomposition
of the domain $\overline{\Omega}$ into triangles (tetrahedrons), $h_K$
be the diameter of the element $K\in \mathcal {T}_h$ and $h=\max\{h_K:
K\in \mathcal {T}_h\}$. We denote the union of the boundaries of all
the $K \in \mathcal {T}_h$ by $\mathcal{E}_h$, $\mathcal{E}^i_h$ is
the set of all the interior edges and
$\mathcal{E}_h^\partial=\mathcal{E}_h/\mathcal{E}^i_h$ is the set of
boundary edges. Let $\underline{\bm{\nabla}}_h$ and $\mathbf{div}_h$
be the broken gradient and divergence operators whose restrictions on
each element $K \in \mathcal{T}_h$ are equal to
$\underline{\bm{\nabla}}$ and $\mathbf{div}$, respectively. In
addition, given an integer $k\geq0$, we denote by $\mathcal{P}_k(D)$
the space of polynomials defined in $D$ of total degree at most $k$.
Recall the notation for vector-valued, tensor-valued and
symmetric-tensor-valued function spaces, we have
$\bm{\mathcal{P}}_k(D)=[\mathcal{P}_k(D)]^n$,
$\underline{\bm{\mathcal{P}}}^{\mathbb{M}}_{k}(D)=[\mathcal{P}_k(D)]^{n\times
n}$, and
$\underline{\bm{\mathcal{P}}}^{\mathbb{S}}_{k}(D)=\{\underline{\bm{\tau}}
\in [\mathcal{P}_k(D)]^{n\times n}: \underline{\bm{\tau}}^t =
\underline{\bm{\tau}} \}$. Construct the discontinuous finite element
spaces $\underline{\bm{\Sigma}}_h$ and $\bm{V}_h$ by
\begin{subequations} \label{Dspaces}
\begin{align}
\label{Dspace-Sigma}
&\underline{\bm{\Sigma}}_h=\{\underline{\bm{\tau}}_h \in
\underline{\bm{L}}^2(\Omega;\mathbb{S}): \underline{\bm{\tau}}_h \in
\underline{\bm{\mathcal{P}}}^{\mathbb{S}}_{k+1}(K)\quad\forall\,K \in
\mathcal {T}_h,\ \int_{\Omega}{\rm
  tr}(\underline{\bm{\tau}}_h)d\bm{x}= 0\},\\
\label{Dspace-V}
&\bm{V}_h=\{\bm{v}_h \in \bm{L}^2(\Omega):  \bm{v}_h \in
\bm{\mathcal{P}}_{k}(K)\quad\forall\,K \in \mathcal {T}_h\}.
\end{align}
\end{subequations}
The norm of $\underline{\bm{\Sigma}}_h$ is defined by
\begin{equation}\label{norm}
\|\underline{\bm{\tau}}_h\|_{\underline{\bm{\Sigma}}_h}^2=\|\underline{\bm{\tau}}_h^d\|_0^2+\widetilde{\kappa}\|\mathbf{div}_h\underline{\bm{\tau}}_h\|^2_{0}+|\underline{\bm{\tau}}_h|_*^2,
\end{equation}
where $|\underline{\bm{\tau}}_h|_*^2=\sum_{e
\in\mathcal{E}^i_h}h_e^{-1}\|[\underline{\bm{\tau}}_h]\|_{e}^2$, $h_e$
is the length of edge $e$, and $\|\cdot\|_{e}$ denotes the $L^2$-norm
on edge $e$.

For an interior edge $e \in \mathcal{E}^i_h$ shared by elements $K^+$
and $K^-$, we define the unit normal vectors $\bm{n}^+$ and $\bm{n}^-$
on $e$ pointing exterior to $K^+$ and $K^-$, respectively. Similarly,
we define vector-valued functions $\bm{v}^{\pm}=\bm{v}|_{\partial
K^{\pm}}$ and tensor-valued functions
$\underline{\bm{\tau}}^{\pm}=\underline{\bm{\tau}}|_{\partial
K^{\pm}}$. Then define the averages $\{\cdot\}$ and the jumps
$\llbracket \cdot \rrbracket $, $[\cdot]$ on $e\in \mathcal{E}^i_h$ by
\begin{align*}
&\{\bm{v}\} = \frac{1}{2}(\bm{v}^+ + \bm{v}^-), \qquad \llbracket
\bm{v} \rrbracket = \frac{1}{2}(\bm{v}^+\otimes\bm{n}^+ +
\bm{v}^-\otimes\bm{n}^- + \bm{n}^+\otimes\bm{v}^+ +
\bm{n}^-\otimes\bm{v}^-),\\
&\{\underline{\bm{\tau}}\} = \frac{1}{2}(\underline{\bm{\tau}}^+ +
\underline{\bm{\tau}}^-), \qquad [\underline{\bm{\tau}}] =
\frac{1}{2}(\underline{\bm{\tau}}^+\bm{n}^+ +
\underline{\bm{\tau}}^-\bm{n}^-),
\end{align*}
where $\bm{v}\otimes\bm{w}$ is a matrix with $v_iw_j$ as its
$(i,j)$-th element. On boundary edge $e \in \mathcal{E}_h^\partial$,
we set
\begin{align*}
&\{\bm{v}\} = \bm{v}, \qquad\qquad \llbracket \bm{v} \rrbracket =
\frac{1}{2}(\bm{v}\otimes\bm{n} + \bm{n}\otimes\bm{v}),\\
&\{\underline{\bm{\tau}}\} = \underline{\bm{\tau}}, \qquad\qquad
[\underline{\bm{\tau}}] = \underline{\bm{\tau}}\bm{n}.
\end{align*}

For any tensor-valued function $\underline{\bm{\tau}}$ and
vector-valued function $\bm{v}$, a straightforward computation shows
that
\begin{equation}\label{sec3_eq0}
\sum_{T\in\mathcal{T}_h}\int_{\partial
  K}\underline{\bm{\tau}}\bm{n}_K\cdot\bm{v}ds =
  \sum_{e\in\mathcal{E}_h^i}\int_e[\underline{\bm{\tau}}]\cdot\{\bm{v}\}ds
  + \sum_{e\in\mathcal{E}_h}\int_e\{\underline{\bm{\tau}}\}:\llbracket
  \bm{v} \rrbracket ds.
\end{equation}

Let us derive the MDG scheme for problem \eqref{Brinkman1}.
Multiplying \eqref{Brinkman1_eq1} by a test function
$\underline{\bm{\tau}}_h$ and \eqref{Brinkman1_eq2} by a test function
$\bm{v}_h$, respectively, integrating on any element $K \in \mathcal
{T}_h$ and applying the Green's formula, we obtain
\begin{subequations}\label{weakBrinkman0}
\begin{align}
\label{weakBrinkman0_eq1}
  &\frac{1}{2}(\underline{\bm{\sigma}}^d,\underline{\bm{\tau}}_h^d)_K
  +
  (\mathbf{div}\underline{\bm{\tau}}_h,\bm{u})_K-\langle\underline{\bm{\tau}}_h\bm{n}_K,\bm{u}\rangle_{\partial
    K} =0 &\forall \underline{\bm{\tau}}_h \in
    \underline{\bm{\Sigma}}_h,\\
	\label{weakBrinkman0_eq2}
  &(\bm{\kappa}^{-1}\bm{u},\bm{v}_h)_K +
  (\underline{\bm{\sigma}},\underline{\bm{\varepsilon}}_h(\bm{v}_h))_K
  -\langle\underline{\bm{\sigma}}\bm{n}_K,\bm{v}_h\rangle_{\partial
    K}= (\bm{f},\bm{v}_h)_K  &\forall \bm{v}_h \in \bm{V}_h.
	\end{align}
\end{subequations}
Here, $\underline{\bm{\varepsilon}}_h(\bm{v}_h) =
(\underline{\bm{\nabla}}_h\bm{v}_h+(\underline{\bm{\nabla}}_h\bm{v}_h)^t)/2$.

Then, we approximate $\underline{\bm{\sigma}}$ and $\bm{u}$ by
$\underline{\bm{\sigma}}_h \in \underline{\bm{\Sigma}}_h$ and
$\bm{u}_h \in \bm{V}_h$, respectively, and the trace of
$\underline{\bm{\sigma}}$ and $\bm{u}$ on element edge by the
numerical fluxes $\widehat{\underline{\bm{\sigma}}}_h$ and
$\widehat{\bm{u}}_h$. Summing on all $K \in \mathcal {T}_h$, we get
\begin{subequations}\label{weakBrinkman1}
\begin{align}
\label{weakBrinkman1_eq1} 
&\frac{1}{2}(\underline{\bm{\sigma}}_h^d,\underline{\bm{\tau}}_h^d)+(\mathbf{div}_h\underline{\bm{\tau}}_h,\bm{u}_h)-\langle\underline{\bm{\tau}}_h\bm{n}_K,\widehat{\bm{u}}_h\rangle_{\partial
\mathcal {T}_h}=0 &\forall \underline{\bm{\tau}}_h \in
\underline{\bm{\Sigma}}_h,\\
\label{weakBrinkman1_eq2}
&(\bm{\kappa}^{-1}\bm{u}_h,\bm{v}_h)+
(\underline{\bm{\sigma}}_h,\underline{\bm{\varepsilon}}_h(\bm{v}_h))
-\langle\widehat{\underline{\bm{\sigma}}}_h\bm{n}_K,\bm{v}_h\rangle_{\partial
\mathcal {T}_h}= (\bm{f},\bm{v}_h)  &\forall \bm{v}_h \in \bm{V}_h.
\end{align}
\end{subequations}
By Green's formula and \eqref{sec3_eq0}, we have
\begin{subequations}\label{weakBrinkman2}
\begin{align}
\label{weakBrinkman2_eq1}
&\frac{1}{2}(\underline{\bm{\sigma}}_h^d,\underline{\bm{\tau}}_h^d)+(\mathbf{div}_h\underline{\bm{\tau}}_h,\bm{u}_h)-\int_{\mathcal{E}_h^i}[\underline{\bm{\tau}}_h]\cdot\{\widehat{\bm{u}}_h\}ds-\int_{\mathcal{E}_h}\{\underline{\bm{\tau}}_h\}:\llbracket
\widehat{\bm{u}}_h \rrbracket ds=0 &\forall \underline{\bm{\tau}}_h
\in \underline{\bm{\Sigma}}_h,\\
  \label{weakBrinkman2_eq2}
&(\bm{\kappa}^{-1}\bm{u}_h,\bm{v}_h)-(\mathbf{div}_h\underline{\bm{\sigma}}_h,\bm{v}_h)+\int_{\mathcal{E}_h^i}[\underline{\bm{\sigma}}_h-\widehat{\underline{\bm{\sigma}}}_h]\cdot\{\bm{v}_h\}ds+\int_{\mathcal{E}_h}\{\underline{\bm{\sigma}}_h-\widehat{\underline{\bm{\sigma}}}_h\}:\llbracket
\bm{v}_h \rrbracket ds= (\bm{f},\bm{v}_h)  &\forall \bm{v}_h \in
\bm{V}_h.
\end{align}
\end{subequations}
We define the numerical fluxes $\widehat{\underline{\bm{\sigma}}}_h$
and $\widehat{\bm{u}}_h$ by
\begin{align}
\label{sec3_eq1}
&\widehat{\underline{\bm{\sigma}}}_h = \{\underline{\bm{\sigma}}_h\} \quad \mbox{and}\quad \widehat{\bm{u}}_h=\{\bm{u}_h\}-\frac{\eta_e}{h_e}[\underline{\bm{\sigma}}_h] &\qquad {\rm on} \ e \in \mathcal{E}^i_h,\\
\label{sec3_eq2}
&\widehat{\underline{\bm{\sigma}}}_h = \underline{\bm{\sigma}}_h \quad
\mbox{and} \quad \widehat{\bm{u}}_h= \bm{g} &\qquad {\rm on} \ e \in
\mathcal{E}^{\partial}_h,
\end{align}
where the penalty parameter $\eta_e=\mathcal{O}(1)$. For simplicity,
we choose $\eta_e=1$ in the analysis.

With such choices and symmetry properties of
$\underline{\bm{\tau}}_h$, the MDG method of the problem
\eqref{Brinkman1} is to find $(\underline{\bm{\sigma}}_h,\bm{u}_h) \in
\underline{\bm{\Sigma}}_h\times\bm{V}_h$ such that
\begin{subequations}\label{weakBrinkman3}
	\begin{align}
	\label{sec3_eq3}
  &a_h(\underline{\bm{\sigma}}_h,\underline{\bm{\tau}}_h) +
  b_h(\underline{\bm{\tau}}_h,\bm{u}_h) =
  \langle\underline{\bm{\tau}}_h\bm{n},\bm{g}\rangle_{\mathcal{E}_h^\partial}
  &\forall \underline{\bm{\tau}}_h \in \underline{\bm{\Sigma}}_h,\\
	\label{sec3_eq4}
  &b_h(\underline{\bm{\sigma}}_h,\bm{v}_h)-s(\bm{u}_h,\bm{v}_h) =
  -(\bm{f},\bm{v}_h) &\forall \bm{v}_h \in \bm{V}_h,
  \end{align}
\end{subequations}
where
\begin{align*}
&a_h(\underline{\bm{\sigma}}_h,\underline{\bm{\tau}}_h)=
\frac{1}{2}(\underline{\bm{\sigma}}_h^d,\underline{\bm{\tau}}_h^d)+\int_{\mathcal{E}_h^i}\frac{1}{h_e}[\underline{\bm{\sigma}}_h]\cdot[\underline{\bm{\tau}}_h]ds
&\forall \underline{\bm{\sigma}}_h, \underline{\bm{\tau}}_h \in
\underline{\bm{\Sigma}}_h,\\
&b_h(\underline{\bm{\tau}}_h,\bm{v}_h) =
(\mathbf{div}_h\underline{\bm{\tau}}_h,\bm{v}_h)-\int_{\mathcal{E}_h^i}[\underline{\bm{\tau}}_h]\cdot\{\bm{v}_h\}ds
&\forall \underline{\bm{\tau}}_h \in \underline{\bm{\Sigma}}_h,
\forall \bm{v}_h \in \bm{V}_h.
\end{align*}

\subsection{Well-posedness of the MDG method}\label{Well_posedness}

In this subsection, we show the well-posedness of the MDG scheme
\eqref{weakBrinkman3}. First, we give some inequalities by lemmas.

The first lemma is a discrete analogy of \cite[Proposition
9.1.1]{Boffi2013Mixed}, which indicates the well-posedness of the
discrete norm \eqref{norm}.
\begin{myLemma}[Lemma 3.3 in \cite{Gatica2015Analysis}]\label{sec3_Lemma0}
For every $\underline{\bm{\tau}}_h \in \underline{\bm{\Sigma}}_h$, it
holds
\[
\|\underline{\bm{\tau}}_h\|_0^2
\lesssim\|\underline{\bm{\tau}}_h^d\|_0^2
+\|\mathbf{div}_h\underline{\bm{\tau}}_h\|_0^2+|[\underline{\bm{\tau}}_h]|_*^2.
\]
\end{myLemma}

\begin{myLemma}[cf. \cite{Douglas2002Unified}]\label{sec3_Lemma1}
There exists positive constants $C_1$ and $C_2$ such that
\begin{align}
\label{sec3_eq5}
  & \|\varphi\|_{e}^2\leq
  C_1(h_K^{-1}\|\varphi\|_{K}^2+h_K|\varphi|_{1,K}^2)\quad\forall\,\varphi\in
  H^1(K),\\
	\label{sec3_eq6}
  & |\varphi|_{1,K}^2\leq C_2h_K^{-2}\|\varphi\|_{K}^2
  ~\qquad\qquad\qquad\forall\,\varphi\in\mathcal{P}_{k+1}(K).
	\end{align}
\end{myLemma}

Then, we define
\begin{equation}\label{sec3_eq7}	
A_h((\underline{\bm{\sigma}}_h,\bm{u}_h),(\underline{\bm{\tau}}_h,\bm{v}_h))
= a_h(\underline{\bm{\sigma}}_h,\underline{\bm{\tau}}_h) +
b_h(\underline{\bm{\tau}}_h,\bm{u}_h)-b_h(\underline{\bm{\sigma}}_h,\bm{v}_h)+s(\bm{u}_h,\bm{v}_h),
\end{equation}
and
\[
F_h((\underline{\bm{\tau}}_h,\bm{v}_h))=(\bm{f},\bm{v}_h)+
\langle\underline{\bm{\tau}}_h\bm{n},\bm{g}\rangle_{\mathcal{E}_h^\partial}.
\]
Equivalently, the problem \eqref{weakBrinkman3} can be rewritten as
the following problem: Find $(\underline{\bm{\sigma}}_h,\bm{u}_h) \in
\underline{\bm{\Sigma}}_h\times\bm{V}_h$, such that
\begin{equation}\label{weakBrinkman4}	
A_h((\underline{\bm{\sigma}}_h,\bm{u}_h),(\underline{\bm{\tau}}_h,\bm{v}_h))=
F_h((\underline{\bm{\tau}}_h,\bm{v}_h)).
\end{equation}
In what follows, we prove the well-posedness of problem
\eqref{weakBrinkman4}.

\begin{myLemma}\label{sec3_Lemma2} The bilinear form $A_h(\cdot,\cdot)$ satisfies
\begin{equation}\label{sec3_eq8}
A_h((\underline{\bm{\sigma}}_h,\bm{u}_h),(\underline{\bm{\tau}}_h,\bm{v}_h))
\lesssim
(\|\underline{\bm{\sigma}}_h\|_{\underline{\bm{\Sigma}}_h}+\|\bm{u}_h\|_{\bm{V}})(\|\underline{\bm{\tau}}_h\|_{\underline{\bm{\Sigma}}_h}+\|\bm{v}_h\|_{\bm{V}})
\qquad \forall (\underline{\bm{\sigma}}_h,\bm{u}_h),
(\underline{\bm{\tau}}_h,\bm{v}_h) \in
\underline{\bm{\Sigma}}_h\times\bm{V}_h.
\end{equation}
\begin{proof} By the definition of norm \eqref{norm} and
Cauchy-Schwarz inequality, we have
\begin{align*}
A_h((\underline{\bm{\sigma}}_h,\bm{u}_h),(\underline{\bm{\tau}}_h,\bm{v}_h))
  &\leq
  \frac{1}{2}\|\underline{\bm{\sigma}}_h^d\|_0\|\underline{\bm{\tau}}_h^d\|_0+|\underline{\bm{\sigma}}_h|_*|\underline{\bm{\tau}}_h|_*
  +\|{\widetilde{\kappa}}^{\frac{1}{2}}\mathbf{div}_h\underline{\bm{\tau}}_h\|_0\|{\widetilde{\kappa}}^{-\frac{1}{2}}\bm{u}_h\|_0
  +\sum_{e \in
\mathcal{E}^i_h}\|h_e^{\frac{1}{2}}{\widetilde{\kappa}}^{-\frac{1}{2}}\{\bm{u}_h\}\|_e\|h_e^{-\frac{1}{2}}[\underline{\bm{\tau}}_h]\|_e\\
&\quad
+\|{\widetilde{\kappa}}^{\frac{1}{2}}\mathbf{div}_h\underline{\bm{\sigma}}_h\|_0\|{\widetilde{\kappa}}^{-\frac{1}{2}}\bm{v}_h\|_0+\sum_{e
  \in
\mathcal{E}^i_h}\|h_e^{\frac{1}{2}}{\widetilde{\kappa}}^{-\frac{1}{2}}\{\bm{v}_h\}\|_e\|h_e^{-\frac{1}{2}}[\underline{
\bm{\sigma}}_h]\|_e+\|\bm{u}_h\|_{\bm{V}}\|\bm{v}_h\|_{\bm{V}}\\
&\lesssim(\|\underline{\bm{\sigma}}_h\|_{\underline{\bm{\Sigma}}_h}+\|\bm{u}_h\|_{\bm{V}})
(\|\underline{\bm{\tau}}_h\|_{\underline{\bm{\Sigma}}_h}+\|\bm{v}_h\|_{\bm{V}}),
\end{align*}
where we use the trace and inverse inequalities and the fact that $\widetilde{\kappa}\leq 1$.
\end{proof}
\end{myLemma}

In order to prove the discrete inf-sup condition, we introduce the
space (cf. \cite{Wu2018Interior})
\begin{eqnarray*}
&&\underline{\bm{\Sigma}}_h^{\rm NC}=\{\underline{\bm{\tau}}_h \in
\underline{\bm{L}}^2(\Omega;\mathbb{S}): \underline{\bm{\tau}}_h \in
\underline{\bm{\mathcal{P}}}^{\mathbb{S}}_{k+1}(K)\ \ \forall\,K \in
\mathcal {T}_h,\ {\rm and }\ {\rm the}\ {\rm moments}\ {\rm of}\\
&&\qquad\qquad\quad \underline{\bm{\tau}}_h\bm{n} \ {\rm up}\ {\rm
  to}\ {\rm degree} \ k\ {\rm are} \ {\rm continuous}\ {\rm across}\
{\rm the}\ {\rm interior}\ {\rm edges}\},\\
&&\underline{\bm{\mathring{\Sigma}}}_h^{\rm
  NC}=\{\underline{\bm{\tau}}_h: \underline{\bm{\tau}}_h \in
  \underline{\bm{\Sigma}}_h^{\rm NC},
  \int_{\Omega}\mathrm{tr}(\underline{\bm{\tau}}_h)d\bm{x}=0\}.
\end{eqnarray*}
For any $\bm{u}_h \in \bm{V}_h$, there exists a constant $C_3>0$ and
$\underline{\bm{\sigma}}_h^* \in \underline{\bm{\Sigma}}_h^{\rm NC}$
(cf. \cite{Wu2018Interior}) such that
\begin{eqnarray}\label{sec3_eq9}
\mathbf{div}_h\underline{\bm{\sigma}}_h^*= \bm{u}_h \quad {\rm
and}\quad
\|\underline{\bm{\sigma}}_h^*\|_{0}^2+\|\mathbf{div}_h\underline{\bm{\sigma}}_h^*\|_{0}^2+|\underline{\bm{\sigma}}_h^*|_*^2\leq C_3\|\bm{u}_h\|_{0}^2.
\end{eqnarray}
We are now in the position to show the inf-sup condition of
$A_h(\cdot, \cdot)$.

\begin{myLemma}\label{sec3_Lemma3} 
For any
$(\underline{\bm{\sigma}}_h,\bm{u}_h),(\underline{\bm{\tau}}_h,\bm{v}_h)
\in \underline{\bm{\Sigma}}_h\times\bm{V}_h$, it holds
\begin{equation}\label{sec3_eq10}
\inf_{(\underline{\bm{\sigma}}_h,\bm{u}_h)\in \underline{\bm{\Sigma}}_h\times\bm{V}_h}\sup_{(\underline{\bm{\tau}}_h,\bm{v}_h)\in \underline{\bm{\Sigma}}_h\times\bm{V}_h}\frac{A_h((\underline{\bm{\sigma}}_h,\bm{u}_h),(\underline{\bm{\tau}}_h,\bm{v}_h))}{
(\|\underline{\bm{\sigma}}_h\|_{\underline{\bm{\Sigma}}_h}+\|\bm{u}_h\|_{\bm{V}})(\|\underline{\bm{\tau}}_h\|_{\underline{\bm{\Sigma}}_h}+\|\bm{v}_h\|_{\bm{V}}) }\gtrsim 1.
\end{equation}
\end{myLemma}
\begin{proof} Set $\gamma = \frac{1}{|\Omega|}\int_{\Omega}
\mathrm{tr}(\underline{\bm{\sigma}}_h^*)d\bm{x}$ and
$\widetilde{\underline{\bm{\sigma}}}_h=\underline{\bm{\sigma}}_h^*-\frac{\gamma}{n}\underline{\bm{I}}$.
Then, it is straightforward to show that 
	\begin{equation}\label{sec3_eq11}
\widetilde{\underline{\bm{\sigma}}}_h \in
\underline{\bm{\mathring{\Sigma}}}_h^{\rm NC}, \quad
\mathbf{div}_h\widetilde{\underline{\bm{\sigma}}}_h = \bm{u}_h \quad
{\rm in} \,\Omega, \quad {\rm and}\quad
\|\widetilde{\underline{\bm{\sigma}}}_h\|_{0}^2+\|\mathbf{div}_h\widetilde{\underline{\bm{\sigma}}}_h\|_{0}^2+|\widetilde{\underline{\bm{\sigma}}}_h|_*^2\leq
C_3\|\bm{u}_h\|_{0}^2.
\end{equation}
We take $\underline{\bm{\tau}}_h = \underline{\bm{\sigma}}_h +
\alpha\widetilde{\underline{\bm{\sigma}}}_h$ and $\bm{v}_h =
\delta_1\bm{u}_h-\delta_2\mathbf{div}_h\underline{\bm{\sigma}}_h$,
where the non-negative undetermined coefficients $\alpha$, $\delta_1$
and $\delta_2$ will be specified in the following analysis. Thanks to
\eqref{sec3_eq11} and
$\int_{\mathcal{E}_h^i}[\widetilde{\underline{\bm{\sigma}}}_h]
\cdot\{\bm{u}_h\}ds = 0$ obtained from the property of
$\underline{\bm{\Sigma}}_h^{\rm NC}$, one gets
\begin{align*}
A_h((\underline{\bm{\sigma}}_h,\bm{u}_h),(\underline{\bm{\tau}}_h,\bm{v}_h))
&=\frac{1}{2}(\underline{\bm{\sigma}}_h^d,\underline{\bm{\sigma}}_h^d
+ \alpha\widetilde{\underline{\bm{\sigma}}}_h^d) +
\int_{\mathcal{E}_h^i}\frac{1}{h_e}[\underline{\bm{\sigma}}_h]\cdot[\underline{\bm{\sigma}}_h
+ \alpha\widetilde{\underline{\bm{\sigma}}}_h]ds +
(\mathbf{div}_h(\underline{\bm{\sigma}}_h +
\alpha\widetilde{\underline{\bm{\sigma}}}_h),\bm{u}_h) \\
& \quad - \int_{\mathcal{E}_h^i}[\underline{\bm{\sigma}}_h +
\alpha\widetilde{\underline{\bm{\sigma}}}_h]\cdot\{\bm{u}_h\}ds -
(\mathbf{div}_h\underline{\bm{\sigma}}_h,\delta_1\bm{u}_h -
 \delta_2\mathbf{div}_h\underline{\bm{\sigma}}_h)  \\
& \quad +
\int_{\mathcal{E}_h^i}[\underline{\bm{\sigma}}_h]\cdot\{\delta_1\bm{u}_h
- \delta_2\mathbf{div}_h\underline{\bm{\sigma}}_h\}ds +
(\kappa^{-1}\bm{u}_h,\delta_1\bm{u}_h-
 \delta_2\mathbf{div}_h\underline{\bm{\sigma}}_h)
\\ 
&\geq\frac{1}{2}\|\underline{\bm{\sigma}}_h^d\|^2+\frac{\alpha}{2}(\underline{\bm{\sigma}}_h^d,\widetilde{\underline{\bm{\sigma}}}_h)
  +|\underline{\bm{\sigma}}_h|_*^2+\alpha\int_{\mathcal{E}_h^i}\frac{1}{h_e}[\underline{\bm{\sigma}}_h]\cdot[\widetilde{\underline{\bm{\sigma}}}_h]ds
  +(1-\delta_1-\delta_2\kappa^{-1})(\mathbf{div}\underline{\bm{\sigma}}_h,\bm{u}_h)\\
& \quad
+\delta_2\|\mathbf{div}_h\underline{\bm{\sigma}}_h\|_0^2+(\delta_1-1)\int_{\mathcal{E}_h^i}[\underline{\bm{\sigma}}_h]\cdot\{\bm{u}_h\}ds
-\delta_2\int_{\mathcal{E}_h^i}[\underline{\bm{\sigma}}_h]\cdot\{\mathbf{div}_h\underline{\bm{\sigma}}_h\}ds
+(\alpha+\kappa^{-1}\delta_1)\|\bm{u}_h\|_0^2.
\end{align*}
According to Cauchy-Schwarz inequality, \eqref{sec3_eq11}, the trace
and inverse inequalities, we obtain
\begin{align*}
\frac{\alpha}{2}(\underline{\bm{\sigma}}_h^d,\widetilde{\underline{\bm{\sigma}}}_h)&\leq\frac{\varepsilon}{4}\|\underline{\bm{\sigma}}_h^d\|_0^2
+\frac{\alpha^2C_3}{4\varepsilon}\|\bm{u}_h\|_0^2,\\
\alpha\int_{\mathcal{E}_h^i}\frac{1}{h_e}[\underline{\bm{\sigma}}_h]\cdot[\widetilde{\underline{\bm{\sigma}}}_h]ds
&\leq\frac{1}{3}|\underline{\bm{\sigma}}_h|_*^2+\frac{3\alpha^2C_3}{4}\|\bm{u}_h\|_0^2,\\
-\delta_2\int_{\mathcal{E}_h^i}[\underline{\bm{\sigma}}_h]\cdot\{\mathbf{div}_h\underline{\bm{\sigma}}_h\}ds
&\leq\frac{\delta_2}{2}\|\mathbf{div}_h\underline{\bm{\sigma}}_h\|_0^2+\frac{\delta_2C_1(1+C_2)\widetilde{\kappa}^{-1}}{2}|\underline{\bm{\sigma}}_h|_*^2,\\
(\delta_1-1)\int_{\mathcal{E}_h^i}[\underline{\bm{\sigma}}_h]\cdot\{\bm{u}_h\}ds
&\leq\frac{1}{3}|\underline{\bm{\sigma}}_h|_*^2+\frac{3(\delta_1-1)^2C_1(1+C_2)\widetilde{\kappa}^{-1}}{4}\|\bm{u}_h\|_0^2.
\end{align*}
Combining with above inequalities, we have
\begin{eqnarray}\label{sec3_eq12}
A_h((\underline{\bm{\sigma}}_h,\bm{u}_h),(\underline{\bm{\tau}}_h,\bm{v}_h))\geq(\frac{1}{2}-\frac{\varepsilon}{4})\|\underline{\bm{\sigma}}_h^d\|^2
+(\frac{1}{3}-\frac{\delta_2C_1(1+C_2)\widetilde{\kappa}^{-1}}{2})|\underline{\bm{\sigma}}_h|_*^2
+(1-\delta_1-\delta_2\kappa^{-1})(\mathbf{div}\underline{\bm{\sigma}}_h,\bm{u}_h)\nonumber\\
\qquad\qquad\qquad
+\frac{\delta_2}{2}\|\mathbf{div}_h\underline{\bm{\sigma}}_h\|_0^2
+\left(\alpha+\kappa^{-1}\delta_1-\frac{\alpha^2C_3}{4\varepsilon}-\frac{3\alpha^2C_3}{4}-\frac{3(\delta_1-1)^2C_1(1+C_2)\widetilde{\kappa}^{-1}}{4}\right)\|\bm{u}_h\|_0^2.
\end{eqnarray}

Now, we need to choose appropriate parameters $\delta_1$, $\delta_2$, $\varepsilon$ and $\alpha$, such that
\begin{align*}
&\frac{1}{2}-\frac{\varepsilon}{4}\gtrsim 1,\qquad 1-
\delta_1-\kappa^{-1}\delta_2=0, \qquad \delta_2\gtrsim\widetilde{\kappa},\\
&\frac{1}{3}-\frac{\delta_2C_1(1+C_2)\widetilde{\kappa}^{-1}}{2}\gtrsim 1,\\
&\alpha+\kappa^{-1}\delta_1-\frac{\alpha^2C_3}{4\varepsilon} 
-\frac{3\alpha^2C_3}{4} 
-\frac{3(\delta_1-1)^2C_1(1+C_2)\widetilde{\kappa}^{-1}}{4}\gtrsim
\widetilde{\kappa}^{-1}.
\end{align*}
We could take $\delta_2 =
\frac{\widetilde{\kappa}}{\max\{2,3C_1(1+C_2),2C_3\}}$, $\delta_1 = 1-
\frac{\delta_2}{\kappa}$, $\varepsilon = 1$ and
$\alpha=\frac{1}{2C_3}$, by which the first three requirements
above meet easily. Furthermore, we have 
$$ 
\frac{1}{3}-\frac{\delta_2C_1(1+C_2)\widetilde{\kappa}^{-1}}{2} =
\frac{1}{3}-\frac{C_1(1+C_2)}{2\max\{2,3C_1(1+C_2),2C_3\}}\geq
\frac{1}{6} \gtrsim 1.
$$ 
Using $\frac{1}{2} \leq \delta_1 < 1$ and $\widetilde{\kappa}^{-1} =
\max\{1, \kappa^{-1}\}$, we show the last inequality into
two cases: 
\begin{align*}
\text{If }\kappa > 1: 
&\quad 
\alpha-\frac{\alpha^2C_3}{4\varepsilon}-\frac{3\alpha^2C_3}{4} 
-\frac{3(\delta_1-1)^2C_1(1+C_2)\widetilde{\kappa}^{-1}}{4}
= \frac{1}{4C_3} -
\frac{3C_1(1+C_2)}{4\kappa^2\max\{2,3C_1(1+C_2),2C_3\}^2} \gtrsim 1; \\
\text{If }\kappa\leq1:
& \quad 
\kappa^{-1}\delta_1 - 
\frac{3(\delta_1-1)^2C_1(1+C_2)\widetilde{\kappa}^{-1}}{4} =
\kappa^{-1}\left( 
\delta_1 - \frac{3C_1(1+C_2)}{4\max\{2,3C_1(1+C_2),2C_3\}^2} 
\right) \geq \frac{3}{8}\kappa^{-1}\gtrsim
\kappa^{-1}.
\end{align*}
From the above, we obtain
\[
A_h((\underline{\bm{\sigma}}_h,\bm{u}_h),(\underline{\bm{\tau}}_h,\bm{v}_h))
  \gtrsim(\|\underline{\bm{\sigma}}_h\|_{\underline{\bm{\Sigma}}_h}+\|\bm{u}_h\|_{\bm{V}})^2.
\]

Next, taking $\alpha$, $\delta_1$ and $\delta_2$ in
$\underline{\bm{\tau}}_h$ and $\bm{v}_h$, due to \eqref{sec3_eq11},
the fact that $\delta_1<1$ and
$\widetilde{\kappa}\leq1\leq\widetilde{\kappa}^{-1}$ , it holds that 
\begin{align*}
\|\underline{\bm{\tau}}_h\|_{\underline{\bm{\Sigma}}_h}^2
&=\|\underline{\bm{\sigma}}_h^d + \alpha\widetilde{\underline{\bm{\sigma}}}_h^d\|_0^2+\widetilde{\kappa}\|\mathbf{div}_h(\underline{\bm{\sigma}}_h + \alpha\widetilde{\underline{\bm{\sigma}}}_h)\|^2_{0}
+|\underline{\bm{\sigma}}_h + \alpha\widetilde{\underline{\bm{\sigma}}}_h|_*^2\\
&\lesssim\|\underline{\bm{\sigma}}_h\|_{\underline{\bm{\Sigma}}_h}^2+\|\widetilde{\underline{\bm{\sigma}}}_h\|_0^2
+\widetilde{\kappa}\|\mathbf{div}_h\widetilde{\underline{\bm{\sigma}}}_h\|^2_{0}
+|\widetilde{\underline{\bm{\sigma}}}_h|_*^2\\
&\lesssim\|\underline{\bm{\sigma}}_h\|_{\underline{\bm{\Sigma}}_h}^2+\|\bm{u}_h\|_{\bm{V}}^2,\\
\|\bm{v}_h\|_{\bm{V}}^2
&=\widetilde{\kappa}^{-1}\|\delta_1\bm{u}_h- \delta_2\mathbf{div}_h\underline{\bm{\sigma}}_h\|_0^2\lesssim \|\bm{u}_h\|_{\bm{V}}^2+\widetilde{\kappa}^{-1}\|\widetilde{\kappa}\mathbf{div}_h\underline{\bm{\sigma}}_h\|_0^2\\
&\lesssim\|\underline{\bm{\sigma}}_h\|_{\underline{\bm{\Sigma}}_h}^2+\|\bm{u}_h\|_{\bm{V}}^2.
\end{align*}
Then, we finish this proof.	
\end{proof}

From Lemma \ref{sec3_Lemma2} and Lemma \ref{sec3_Lemma3}, the
well-posedness of the problem \eqref{weakBrinkman4} can be obtained.

\begin{myTheorem}\label{sec3_Lemma4} The mixed DG scheme
\eqref{weakBrinkman4} has a unique solution
$(\underline{\bm{\sigma}}_h,\bm{u}_h) \in
\underline{\bm{\Sigma}}_h\times\bm{V}_h$.
\end{myTheorem}

\begin{myRemark}\label{sec3_Remark2}
For the bilinear form $A_h(\bm{u}_h,\bm{v}_h)$ of
\eqref{weakBrinkman4}, when the penalty term
$\int_{\mathcal{E}_h^i}\frac{1}{h_e}[\underline{\bm{\sigma}}_h] \cdot
[\underline{\bm{\tau}}_h]ds$ is replaced by $\sum_{e \in
\mathcal{E}_h^i}\int_{\Omega}r_e([\underline{\bm{\sigma}}_h])\cdot
r_e([\underline{\bm{\tau}}_h])d\bm{x}$, we can obtain another mixed DG
scheme, which is the dual form of the method of Brezzi et al. 
\cite{Bassi1997A}. Here, $r_e: (L^2({\cal E}_h))^d \rightarrow
\bm{V}_h$ is the lifting operator (cf. \cite{Bassi1997A,
Douglas2002Unified, Wang201sMixed}).  The well-posedness of the
corresponding scheme can be proved similarly.
\end{myRemark}
		

\section{Error estimates}\label{errorestimation}

In this section, we aim to derive the error estimates for the MDG
scheme \eqref{weakBrinkman3}. First, we show the consistency of the
MDG scheme, which naturally leads to an error estimate by the inf-sup
condition.

\subsection{Error estimate in energy norm for the pseudostress and
velocity}\label{energy_norm}

\begin{myLemma}\label{sec4_Lemma1}
Let the solution $(\underline{\bm{\sigma}},\bm{u}) \in
\underline{\bm{\Sigma}}\times\bm{H}^1(\Omega)$, then
\begin{align}\label{sec4_eq3}
A_h((\underline{\bm{\sigma}}-\underline{\bm{\sigma}}_h,\bm{u}-\bm{u}_h),(\underline{\bm{\tau}}_h,\bm{v}_h))=0
\qquad &\forall (\underline{\bm{\tau}}_h,\bm{v}_h) \in
\underline{\bm{\Sigma}}_h\times\bm{V}_h.
\end{align}
\end{myLemma}
\begin{proof} 
Since $(\underline{\bm{\sigma}},\bm{u})\in
\underline{\bm{\Sigma}}\times\bm{H}^1(\Omega)$, we have
$[\underline{\bm{\sigma}}]=0$ and $\llbracket \bm{u} \rrbracket=0$ on
each edge $e \in \mathcal{E}_h^i$, then
\begin{align*}	
  &\ \ \ \ a_h(\underline{\bm{\sigma}},\underline{\bm{\tau}}_h) +
  b_h(\underline{\bm{\tau}}_h,\bm{u})-\langle\underline{\bm{\tau}}_h\bm{n},\bm{g}\rangle_{\mathcal{E}_h^\partial}\\
  &=
  \frac{1}{2}(\underline{\bm{\sigma}}^d,\underline{\bm{\tau}}_h^d)+(\mathbf{div}_h{\underline{\bm{\tau}}_h},\bm{u})-\int_{\mathcal{E}_h^i}\left[\underline{\bm{\tau}}_h\right]\cdot\{\bm{u}\}ds-\langle\underline{\bm{\tau}}_h\bm{n},\bm{g}\rangle_{\mathcal{E}_h^\partial}\\
   &=\frac{1}{2}(\underline{\bm{\sigma}}^d,\underline{\bm{\tau}}_h^d)-(\underline{\bm{\varepsilon}}(\bm{u}),\underline{\bm{\tau}}_h)+\int_{\mathcal{E}_h^i}\left[\underline{\bm{\tau}}_h\right]\cdot\{\bm{u}\}ds+\int_{\mathcal{E}_h}\{\underline{\bm{\tau}}_h\}:\llbracket
   \bm{u} \rrbracket ds
   -\int_{\mathcal{E}_h^i}\left[\underline{\bm{\tau}}_h\right]\cdot\{\bm{u}\}ds-\langle\underline{\bm{\tau}}_h\bm{n},\bm{g}\rangle_{\mathcal{E}_h^\partial}\\
	&=(\frac{1}{2}\underline{\bm{\sigma}}^d-\underline{\bm{\varepsilon}}(\bm{u}),\underline{\bm{\tau}}_h)=0,
	\end{align*}
	and
	\begin{eqnarray}\label{sec4_eq4}
  b_h(\underline{\bm{\sigma}},\bm{v}_h)-s(\bm{u},\bm{v}_h)= (
      \mathbf{div}\underline{\bm{\sigma}},\bm{v}_h)-(\kappa^{-1}\bm{u},\bm{v}_h)=
  -(\bm{f},\bm{v}_h).
	\end{eqnarray}
Then, from \eqref{weakBrinkman3}, we have
\begin{subequations} \label{equ:error-equ}
\begin{align}
\label{sec4_eq3.1}
&a_h(\underline{\bm{\sigma}}-\underline{\bm{\sigma}}_h,\underline{\bm{\tau}}_h)
  + b_h(\underline{\bm{\tau}}_h,\bm{u}-\bm{u}_h) =0   &\forall
  \underline{\bm{\tau}}_h \in \underline{\bm{\Sigma}}_h,\\
\label{sec4_eq3.2}
&b_h(\underline{\bm{\sigma}}-\underline{\bm{\sigma}}_h,\bm{v}_h) -
s(\bm{u}-\bm{u}_h,\bm{v}_h) =0   &\forall \bm{v}_h \in \bm{V}_h.
\end{align}
\end{subequations}
By \eqref{sec3_eq7}, we complete the proof. 
\end{proof}

\begin{myTheorem}\label{sec4_Lemma2}
Let $(\underline{\bm{\sigma}},\bm{u}) \in
\underline{\bm{\Sigma}}\times\bm{H}^1(\Omega)$ be the solution of
\eqref{weakBrinkman}, and $(\underline{\bm{\sigma}}_h,\bm{u}_h) \in
\underline{\bm{\Sigma}}_h\times\bm{V}_h$ be the solution of
\eqref{weakBrinkman3}. Then, we have
\begin{align}\label{sec4_eq5}
\|\underline{\bm{\sigma}}-\underline{\bm{\sigma}}_h\|_{\underline{\bm{\Sigma}}_h}+\|\bm{u}-\bm{u}_h\|_{\bm{V}}
\lesssim \inf_{\underline{\bm{\tau}}_h\in
\underline{\bm{\Sigma}}_h}\|\underline{\bm{\sigma}}-\underline{\bm{\tau}}_h\|_{\underline{\bm{\Sigma}}_h}+\inf_{\bm{v}_h\in
\bm{V}_h}(\|\bm{u}-\bm{v}_h\|_{\bm{V}}+\widetilde{\kappa}^{-\frac{1}{2}}\sum_{K
\in \mathcal {T}_h}h|\bm{v}_h-\bm{u}|_{1,K}).
		\end{align}
\end{myTheorem}
\begin{proof} From Lemma \ref{sec2_Lemma2}, Lemma \ref{sec3_Lemma3}
and Lemma \ref{sec4_Lemma1}, we obtain that for any
$\underline{\bm{\tau}}_h \in \underline{\bm{\Sigma}}_h$ and $\bm{v}_h
\in \bm{V}_h$,
	\begin{align}\label{sec4_eq8} \|\underline{\bm{\tau}}_h-\underline{\bm{\sigma}}_h\|_{\underline{\bm{\Sigma}}_h}+\|\bm{v}_h-\bm{u}_h\|_{\bm{V}}&\lesssim\sup_{(\underline{\bm{\theta}}_h,\bm{w}_h) \in \underline{\bm{\Sigma}}_h\times\bm{V}_h}\frac{A_h((\underline{\bm{\tau}}_h-\underline{\bm{\sigma}}_h,\bm{v}_h-\bm{u}_h),(\underline{\bm{\theta}}_h,\bm{w}_h))}{\|\underline{\bm{\theta}}_h\|_{\underline{\bm{\Sigma}}_h}+ \|\bm{w}_h\|_{\bm{V}}}\nonumber\\
	&=\sup_{(\underline{\bm{\theta}}_h,\bm{w}_h) \in \underline{\bm{\Sigma}}_h\times\bm{V}_h}\frac{A_h((\underline{\bm{\tau}}_h-\underline{\bm{\sigma}},\bm{v}_h-\bm{u}),(\underline{\bm{\theta}}_h,\bm{w}_h))}{\|\underline{\bm{\theta}}_h\|_{\underline{\bm{\Sigma}}_h}+ \|\bm{w}_h\|_{\bm{V}}}\nonumber\\
	&=\sup_{(\underline{\bm{\theta}}_h,\bm{w}_h) \in \underline{\bm{\Sigma}}_h\times\bm{V}_h}\frac{a_h(\underline{\bm{\tau}}_h-\underline{\bm{\sigma}},\underline{\bm{\theta}}_h)+b_h(\underline{\bm{\theta}}_h,\bm{v}_h-\bm{u})-b_h(\underline{\bm{\tau}}_h-\underline{\bm{\sigma}},\bm{w}_h)+s(\bm{v}_h-\bm{u},\bm{w}_h)}{\|\underline{\bm{\theta}}_h\|_{\underline{\bm{\Sigma}}_h}+\|\bm{w}_h\|_{\bm{V}}}\nonumber\\
(\text{by } \kappa^{-1} \lesssim \widetilde{\kappa}^{-1})\quad\qquad\qquad	 &\lesssim\|\underline{\bm{\tau}}_h-\underline{\bm{\sigma}}\|_{\underline{\bm{\Sigma}}_h}+\|\bm{v}_h-\bm{u}\|_{\bm{V}}+\sup_{\underline{\bm{\theta}}_h \in \underline{\bm{\Sigma}}_h}\frac{b_h(\underline{\bm{\theta}}_h,\bm{v}_h-\bm{u})}{\|\underline{\bm{\theta}}_h\|_{\underline{\bm{\Sigma}}_h}}\nonumber\\
	&\lesssim \|\underline{\bm{\tau}}_h-\underline{\bm{\sigma}}\|_{\underline{\bm{\Sigma}}_h}+\|\bm{v}_h-\bm{u}\|_{\bm{V}}+\widetilde{\kappa}^{-\frac{1}{2}}\sum_{K \in \mathcal {T}_h}h|\bm{v}_h-\bm{u}|_{1,K}.
\end{align}
Then, the triangle inequality indicates the estimate \eqref{sec4_eq5}.
\end{proof}

Recall that $\widetilde{\kappa}^{-1} \geq 1$ and the definition of
$\|\cdot\|_{\bm{V}}$ in \eqref{Cnorm-V}, the above theorem shows the
parameter-robust error estimate of $\bm{u}_h$ in $\bm{L}^2$ norm by
the standard interpolation theory (cf. \cite{Scott1990Finite}). 

\begin{myTheorem} \label{thm:error1}
Assume that the solution of \eqref{weakBrinkman} satisfies
$(\underline{\bm{\sigma}}, \bm{u})\in \underline{\bm{H}}^{k+2}(\Omega) \times
\bm{H}^{k+1}(\Omega)$.  Then, the solution of the mixed DG problem
\eqref{weakBrinkman3} satisfies for any $\kappa > 0$,
\begin{equation} \label{equ:error-u}
\|\bm{u} - \bm{u}_h\|_{0} \lesssim h^{k+1}(|\underline{\bm \sigma}|_{k+2} +
|\bm{u}|_{k+1}).
\end{equation}
Further, if $\kappa \gtrsim 1$ (high permeability case), we have  
\begin{equation} \label{equ:error-sigma1}
\|\underline{\bm \sigma}^d - \underline{\bm{\sigma}}_h^d\|_0 +
\|\mathbf{div}_h ( \underline{\bm \sigma} -
\underline{\bm{\sigma}}_h)\|_{0} + |\underline{\bm{\sigma}}_h|_*
\lesssim h^{k+1}(|\underline{\bm \sigma}|_{k+2} + |\bm{u}|_{k+1}).
\end{equation}
Here, the hidden constants in \eqref{equ:error-u} and
\eqref{equ:error-sigma1} are both independent of $\kappa$.
\end{myTheorem}

\begin{myRemark}
By Lemma \ref{sec3_Lemma0}, we have that if $\kappa \gtrsim 1$, then 
$\|\underline{\bm \sigma} - \underline{\bm{\sigma}}_h\|_0 \lesssim h^{k+1}(|\underline{\bm \sigma}|_{k+2} + |\bm{u}|_{k+1})$.
\end{myRemark}

\subsection{Parameter-robust error estimate of pseudostress}
In this subsection, we show a parameter-robust error estimate of
pseudostress for arbitrary permeability, which fills the gap of
\eqref{equ:error-sigma1} in Theorem \ref{thm:error1}. The result
hinges on the parameter-robust estimate of velocity given in
\eqref{equ:error-u}.  

Let $\bm{P}_h: \bm{V} \rightarrow \bm{V}_h$ denote the
$\bm{L}^2$-orthogonal projection defined by
\begin{equation}\label{l2-operator}
\int_{K}(\bm{P}_h\bm{v} -
\bm{v})\cdot\bm{w}_h d\bm{x}=0\quad\forall\,\bm{w}_h \in
\bm{V}_h,\ K\in \mathcal {T}_h.
\end{equation} 
For $\underline{\bm \sigma} \in \underline{\bm H}^{k+2}(\Omega)$, let
$\underline{\bm{\sigma}}_I^{\rm SZ} \in \underline{\bm{H}}^1(\Omega)$
be the Scott-Zhang interpolation (cf.  \cite{Scott1990Finite}) that
satisfies 
\begin{equation} \label{equ:Scott-Zhang}
\|\underline{\bm \sigma} - \underline{\bm \sigma}_I^{\rm SZ}\|_0 + h
|\underline{\bm \sigma} - \underline{\bm \sigma}_I^{\rm SZ}|_1 \lesssim
h^{k+2}|\underline{\bm \sigma}|_{k+2}.
\end{equation} 
Next, we modify the Scott-Zhang interpolation by
${\underline{\bm \sigma}}_I = \underline{\bm \sigma}_I^{\rm SZ} -
\frac{1}{n|\Omega|} \int_\Omega \mathrm{tr}(\underline{\bm
\sigma}_I^{\rm SZ})\underline{\bm I} d\bm{x} \in \underline{\bm{H}}^1(\Omega) \cap \underline{\bm{\Sigma}}_h$.
Using the property of Scott-Zhang interpolation in
\eqref{equ:Scott-Zhang} and the fact that $\int_\Omega \mathrm{tr}(\underline{\bm \sigma}) d\bm{x} = 0$, we have 
$$ 
\left|\frac{1}{n|\Omega|} \int_\Omega \mathrm{tr}(\underline{\bm
\sigma}_I^{\rm SZ})\underline{\bm I} d\bm{x}\right| = 
\left|
\frac{1}{n|\Omega|} \int_\Omega \mathrm{tr}(\underline{\bm
\sigma}_I^{\rm SZ} - \underline{\bm \sigma})\underline{\bm I} d\bm{x}
\right| \lesssim \|\mathrm{tr}(\underline{\bm
\sigma}_I^{\rm SZ} - \underline{\bm \sigma})\|_0 \lesssim
h^{k+2}|\underline{\bm \sigma}|_{k+2}. 
$$ 
Hence, the modified Scott-Zhang interpolation has the same
approximation as the standard one, i.e., 
\begin{equation} \label{equ:Scott-Zhang2}
\|\underline{\bm \sigma} - {\underline{\bm \sigma}}_I\|_0 + h
|\underline{\bm \sigma} - {\underline{\bm \sigma}}_I|_1 \lesssim h^{k+2}|\underline{\bm \sigma}|_{k+2}.
\end{equation} 
We also denote $\underline{\bm e}_{\underline{\bm \sigma}} =
\underline{\bm \sigma}_I - \underline{\bm \sigma}_h$, ${\bm e}_{\bm u}
= \bm{P}_h \bm{u} - \bm{u}_h$.

\begin{myTheorem} \label{thm:error2}
Assume that the solution of \eqref{weakBrinkman} satisfies
$(\underline{\bm{\sigma}}, \bm{u})\in \underline{\bm{H}}^{k+2}(\Omega) \times
\bm{H}^{k+1}(\Omega)$.  Then, the solution of the mixed DG problem
\eqref{weakBrinkman3} satisfies for any $\kappa > 0$,
\begin{equation} \label{equ:error-sigma2}
\|\underline{\bm \sigma}^d - \underline{\bm{\sigma}}_h^d\|_0^2 +
|\underline{\bm{\sigma}}_h|_*^2 \lesssim h^{k+1}(|\underline{\bm
\sigma}|_{k+2} + |\bm{u}|_{k+1}),
\end{equation}
where the hidden constant is independent of $\kappa$.
\end{myTheorem}
\begin{proof}
Taking $\underline{\bm \tau}_h = \underline{\bm e}_{\underline{\bm
\sigma}} $ and $\bm{v}_h = \bm{e}_{\bm u}$ in the error equation
\eqref{equ:error-equ}, we have 
$$ 
\begin{aligned}
&\frac{1}{2}(\underline{\bm \sigma}^d - \underline{\bm \sigma}_h^d,
\underline{\bm \sigma}_I^d - \underline{\bm \sigma}_h^d) +
\int_{\mathcal{E}_h^i} \frac{1}{h_e} [\underline{\bm \sigma}_h] \cdot 
[\underline{\bm \sigma}_h] ds + b_h(\underline{\bm e}_{\underline{\bm
\sigma}}, {\bm e}_{\bm u}) 
+ b_h(\underline{\bm e}_{\underline{\bm \sigma}}, \bm{u} - {\bm
P}_h{\bm u}) = 0, \\
& b_h(\underline{\bm e}_{\underline{\bm \sigma}}, {\bm e}_{\bm u}) + 
b_h(\underline{\bm \sigma} - \underline{\bm \sigma}_I, {\bm
e}_{\bm u}) - \kappa^{-1}\|{\bm e}_{\bm u}\|_0^2 = 0.
\end{aligned}
$$ 
Subtracting the above equations, we get  
\begin{equation} \label{equ:l2-estimate1}
\frac{1}{2}\|\underline{\bm e}_{\underline{\bm \sigma}}^d\|_0^2 +
|\underline{\bm \sigma}_h|_*^2 + \kappa^{-1}\|{\bm
e}_{\bm u}\|_0^2 = -\frac{1}{2}(\underline{\bm \sigma}^d -
 \underline{\bm \sigma}_I^d, \underline{\bm e}_{\underline{\bm
\sigma}}^d) - b_h(\underline{\bm e}_{\underline{\bm
\sigma}}, \bm{u} - \bm{P}_h{\bm u}) + b_h(\underline{\bm \sigma} -
\underline{\bm \sigma}_I, \bm{e}_{\bm u}).
\end{equation}
Using the property of Scott-Zhang interpolation
\eqref{equ:Scott-Zhang2}, the estimate \eqref{equ:error-u}, Cauchy-Schwarz
inequality, and trace inequality, we have 
$$ 
\begin{aligned}
-\frac{1}{2}(\underline{\bm \sigma}^d - \underline{\bm \sigma}_I^d,
\underline{\bm e}_{\underline{\bm \sigma}}^d) 
& \leq 
\frac{1}{4} \|\underline{\bm e}_{\underline{\bm \sigma}}^d\|_0^2 +
Ch^{2k+4}|\underline{\bm \sigma}|_{k+2}^2,\\ 
- b_h(\underline{\bm e}_{\underline{\bm \sigma}}, \bm{u} - \bm{P}_h{\bm u}) &=
-(\mathbf{div}_h\underline{\bm e}_{\underline{\bm \sigma}},\bm{u} -
    \bm{P}_h{\bm u}) - \int_{\mathcal E_h^i} [\underline{\bm
  \sigma}_h] \cdot \{\bm{u} - \bm{P}_h{\bm u}\}
  ds \\
& = - \int_{\mathcal
  E_h^i} h_e^{-1/2}[\underline{\bm \sigma}_h] \cdot h_e^{1/2}\{\bm{u}
  - \bm{P}_h{\bm u}\} ds 
\leq \frac{1}{2}|\underline{\bm \sigma}_h|_*^2 + C
h^{2k+2}|u|_{k+1}^2, \\
b_h(\underline{\bm \sigma} - \underline{\bm \sigma}_I, \bm{e}_{\bm
u}) & \leq \|\bm{e}_{\bm u}\|_0^2 + C h^{2k+2}|\underline{\bm
  \sigma}|_{k+2}^2 \leq Ch^{2k+2}(|\underline{\bm \sigma}|_{k+2} +
|\bm{u}|_{k+1})^2. 
\end{aligned}
$$ 
Taking above inequalities into the right hand side of
\eqref{equ:l2-estimate1}, we have 
$$ 
\frac{1}{4}\|\underline{\bm e}_{\underline{\bm \sigma}}^d\|_0^2 + \frac{1}{2}|\underline{\bm \sigma}_h|_*^2 +
\kappa^{-1}\|{\bm e}_{\bm u}\|_0^2 \leq 
Ch^{2k+2}(|\underline{\bm \sigma}|_{k+2} +
|\bm{u}|_{k+1})^2,
$$ 
which leads to the desired estimate \eqref{equ:error-sigma2} by
triangle inequality.
\end{proof}

\begin{myRemark}
As a byproduct in the proof of above theorem, it can be seen that,
under the condition of Theorem \ref{thm:error2},
\begin{equation} \label{equ:kappa-eu}
\kappa^{-1/2}\|{\bm e}_{\bm u}\|_0 \lesssim h^{k+1} 
(|\underline{\bm \sigma}|_{k+2} + |\bm{u}|_{k+1}),
\end{equation}
which implies that $\bm{u}_h \to \bm{P}_h\bm{u}$ as $\kappa \to 0$. 
\end{myRemark}

We then have the error estimate of pressure in the following theorem. 
\begin{myTheorem} \label{thm:error3}
Assume that the solution of \eqref{weakBrinkman} satisfies
$(\underline{\bm{\sigma}}, \bm{u})\in \underline{\bm{H}}^{k+2}(\Omega) \times
\bm{H}^{k+1}(\Omega)$.  Then, the solution of the mixed DG problem
\eqref{weakBrinkman3} satisfies 
\begin{equation} \label{equ:error-p} 
\|\mathrm{tr}(\underline{\bm \sigma} -
\underline{\bm{\sigma}}_h)\|_0^2 \lesssim \widetilde{\kappa}^{-1/2}
h^{k+1}(|\underline{\bm \sigma}|_{k+2} + |\bm{u}|_{k+1}),
\end{equation}
\end{myTheorem}
\begin{proof}
Taking $\bm{v}_h = \mathbf{div}_h {\underline{\bm
e}}_{\underline{\bm \sigma}}$ in \eqref{sec4_eq3.2}, we obtain 
$$ 
\begin{aligned}
\|\mathbf{div}_h {\underline{\bm
e}}_{\underline{\bm \sigma}}\|_0^2 &=- (\mathbf{div}_h(\underline{\bm
\sigma} - {\underline{\bm \sigma}}_I), \mathbf{div}_h
    {\underline{\bm e}}_{\underline{\bm \sigma}}) 
- \int_{\mathcal E_h^i} [\underline{\bm \sigma}_h] \cdot \{ 
\mathbf{div}_h
{\underline{\bm e}}_{\underline{\bm \sigma}}
\} ds + \kappa^{-1}(\bm{e}_{\bm u}, 
\mathbf{div}_h {\underline{\bm e}}_{\underline{\bm \sigma}})
\\
&\leq \frac{1}{2} \|\mathbf{div}_h {\underline{\bm
e}}_{\underline{\bm \sigma}}\|_0^2 + Ch^{2k+2}|\underline{\bm
  \sigma}|_{k+2}^2 + C|\underline{\bm \sigma}_h|_*^2 +
  C\kappa^{-2}\|\bm{e}_{\bm u}\|_0^2.
\end{aligned}
$$ 
Using Theorem \ref{thm:error2} and \eqref{equ:kappa-eu}, one gets 
$$ 
\|\mathbf{div}_h {\underline{\bm
e}}_{\underline{\bm \sigma}}\|_0 \lesssim \widetilde{\kappa}^{-1/2}
h^{k+1}(|\underline{\bm \sigma}|_{k+2} + |\bm{u}|_{k+1}).
$$ 
In light of \eqref{equ:Scott-Zhang2}, we apply Lemma \ref{sec3_Lemma0}
to have 
$$ 
\begin{aligned}
\|\mathrm{tr}({\underline{\bm
e}}_{\underline{\bm \sigma}})\|_0 &\lesssim \| {\underline{\bm
e}}_{\underline{\bm \sigma}}^d\|_0 + \| \mathbf{div}_h
{\underline{\bm e}}_{\underline{\bm \sigma}}\|_0 +  
|{\underline{\bm e}}_{\underline{\bm \sigma}}|_{*} \\ 
&\leq
\| \underline{\bm \sigma}^d - \underline{\bm \sigma}_h^d\|_0 +
\| \underline{\bm \sigma}^d - {\underline{\bm \sigma}}_I^d\|_0
+ \| \mathbf{div}_h
{\underline{\bm e}}_{\underline{\bm \sigma}}\|_0 +  
|\underline{\bm \sigma}_h|_{*} \\
&\lesssim \widetilde{\kappa}^{-1/2}
h^{k+1}(|\underline{\bm \sigma}|_{k+2} + |\bm{u}|_{k+1}),
\end{aligned}
$$ 
which leads to the desired estimate. 
\end{proof}

\subsection{Improved error estimates in $L^2$ norm for the
pseudostress and pressure}\label{L2_norm} 

In this subsestion, following a similar argument in
\cite{Wang201sMixed}, we show that the $L^2$ error
estimates for pesudostress and pressure are both optimal when the Stokes finite element
pair $\bm{\mathcal{P}}^c_{k+2}$-$\mathcal{P}_{k+1}$ ($k\geq n$) is
stable. Here, $\mathcal{P}^c$ represents the conforming polynomial.

Now, we recall the classical BDM projection $\pi_h^{c}$ (cf.
\cite{Brezzi1985Two} for two-dimension case and \cite{Brezzi1987Mixed}
for three-dimensional case).  The projection $\pi_h^{c}: \bm{H}({\rm
  div}) \to \bm{W}_{h}=\{\bm{v} \in \bm{H}({\rm div}): \bm{v}|_{K}\in
\bm{\mathcal{P}}_{k+1}(K) \ \forall  K \in \mathcal {T}_h\}$ is
defined by
\begin{subequations}\label{BDM_operator1}
\begin{align}
&\int_e(\pi_h^{c}\bm{v}-\bm{v})\cdot\bm{n}q_h = 0 & \forall q_h \in
\mathcal{P}_{k+1}(e),\\
&\int_K(\pi_h^{c}\bm{v}-\bm{v})\cdot\nabla q_h = 0 & \forall q_h \in
\mathcal{P}_{k}(K),\\
&\int_K(\pi_h^{c}\bm{v}-\bm{v})\cdot\bm{w}_h =0 & \forall \bm{w}_h \in
\bm{W}_{h,*}(K).
\end{align}
\end{subequations}
Here, $\bm{W}_{h,*}(K) = \{ \bm{z} \in \bm{\mathcal{P}}_{k+1}(K):
\bm{z}\cdot\bm{n}=0 \ {\rm on}\ e \in \partial K \ {\rm and}\
(\bm{z},\nabla q_h)_K=0\ \forall\, q_h \in \mathcal{P}_{k}(K)\}$.

Then, based on the projection \eqref{BDM_operator1}, on each element
$K \in \mathcal {T}_h$, we define a function
$\widetilde{\underline{\bm{\sigma}}}_h$ as the only element of
$\underline{\bm{\mathcal{P}}}^{\mathbb{M}}_{k+1}(K)$ by
$\widehat{\underline{\bm{\sigma}}}_h$ and $\underline{\bm{\sigma}}_h$
in \eqref{weakBrinkman1}.
\begin{subequations}\label{BDM_operator2}
\begin{align}
	\label{BDM_operator2_eq1}
   &\int_e(\widetilde{\underline{\bm{\sigma}}}_h-\widehat{\underline{\bm{\sigma}}}_h)\bm{n}\cdot\bm{v}_hds
   = 0\qquad  & \forall \bm{v}_h \in \bm{\mathcal{P}}_{k+1}(e),\\
	\label{BDM_operator2_eq2}
   &\int_K(\widetilde{\underline{\bm{\sigma}}}_h-\underline{\bm{\sigma}}_h):\underline{\bm{\nabla}}\bm{v}_hd\bm{x}
   = 0 \qquad  & \forall \bm{v}_h \in \bm{\mathcal{P}}_{k}(K),\\
	\label{BDM_operator2-eq3}
  &\int_K(\widetilde{\underline{\bm{\sigma}}}_h-\underline{\bm{\sigma}}_h):\underline{\bm{\tau}}_hd\bm{x}
  =0 \qquad  & \forall \underline{\bm{\tau}}_h \in
  \underline{\bm{\Sigma}}_{h,*}^{c}(K),
	\end{align}
\end{subequations}
where $\underline{\bm{\Sigma}}_{h,*}^{c}(K)=\{\underline{\bm{\theta}}
\in \underline{\bm{\mathcal{P}}}^{\mathbb{M}}_{k+1}(K):
\underline{\bm{\theta}}\bm{n}=\mathbf{0}\ {\rm on}\ e \in \partial K \
{\rm and}\
  (\underline{\bm{\theta}},\underline{\bm{\nabla}}\bm{v}_h)_K=0 \
  \forall\, \bm{v}_h \in \bm{\mathcal{P}}_{k}(K)\}$.

The system \eqref{BDM_operator2} can be regarded as the row-wise BDM
projection. According to the definition of $\pi_h^{c}$ and the fact
that the normal component of the numerical trace for the flux is
single-valued, we have the following lemma.
\begin{myLemma}\label{sec4_Lemma3}
The function $\widetilde{\underline{\bm{\sigma}}}_h$ in
\eqref{BDM_operator2} is well-defined,
	\begin{align}\label{sec4_eq9.1}
  &\widetilde{\underline{\bm{\sigma}}}_h \in
  \underline{\bm{\Sigma}}_h^{c}=\{\underline{\bm{\tau}} \in
  \underline{\bm{H}}( \bm{{\rm div}};\mathbb{M}):
    \underline{\bm{\tau}}|_{K}\in
    \underline{\bm{\mathcal{P}}}^{\mathbb{M}}_{k+1}(K) \ \forall  K
    \in \mathcal {T}_h, \ \int_{\Omega}{\rm
      tr}(\underline{\bm{\tau}})d\bm{x}=0\},\\
\label{sec4_eq9.2}
&\|\underline{\bm{\sigma}}_h-\widetilde{\underline{\bm{\sigma}}}_h\|_{0,K}
\lesssim
h_{K}^{1/2}\|(\underline{\bm{\sigma}}_h-\widehat{\underline{\bm{\sigma}}}_h)\bm{n}\|_{0,\partial
  K}.
	\end{align}
\end{myLemma}
\begin{proof}
According to definition of BDM projection, the fact that the normal
component of the numerical trace for the flux is single-valued and
$\int_{\Omega}{\rm tr}(\underline{\bm{\sigma}}_h)d\bm{x}=0$, the
well-posedness and \eqref{sec4_eq9.1} are directly available. Setting
$\underline{\bm{\sigma}}=\widetilde{\underline{\bm{\sigma}}}_h-\underline{\bm{\sigma}}_h$, from \eqref{BDM_operator2}, we know
   \begin{align*}
   &\int_e\underline{\bm{\sigma}}\bm{n}\cdot\bm{v}_hds=\int_e(\widehat{\underline{\bm{\sigma}}}_h-\underline{\bm{\sigma}}_h)\bm{n}\cdot\bm{v}_hds
   \qquad  & \forall \bm{v}_h \in \bm{\mathcal{P}}_{k+1}(e),\\
   &\int_K\underline{\bm{\sigma}}:\underline{\bm{\nabla}}\bm{v}_hd\bm{x}=0
   \qquad  & \forall \bm{v}_h \in \bm{\mathcal{P}}_{k}(K),\\
  &\int_K\underline{\bm{\sigma}}:\underline{\bm{\tau}}_hd\bm{x} =0
  \qquad  & \forall \underline{\bm{\tau}}_h \in
  \underline{\bm{\Sigma}}_{h,*}^{c}(K).
	\end{align*}
By the standard scaling argument, we have \eqref{sec4_eq9.2}.
\end{proof}

Then, by the similar argument in \cite{Gong2018New, Wang201sMixed}, we
symmetrize $\widetilde{\underline{\bm{\sigma}}}_h$ to establish the
$\underline{\bm{L}}^2$ error estimate for pseudostress variable. With
the help of the stable Stokes pair
$\bm{\mathcal{P}}^c_{k+2}$-$\mathcal{P}_{k+1}$ ($k\geq n$), one finds
the following result. We refer the reader to \cite{Wang201sMixed} for
detailed discussion.
\begin{myLemma}[cf. \cite{Wang201sMixed}]\label{sec4_Lemma4}
Assume that the Stokes pair
$\bm{\mathcal{P}}^c_{k+2}$-$\mathcal{P}_{k+1}$ ($k\geq n$) is stable
on the decomposition $\mathcal {T}_h$. For
$\widetilde{\underline{\bm{\sigma}}}_h$ given in
\eqref{BDM_operator2}, there exists
  $\widetilde{\underline{\bm{\tau}}}_h \in
  \underline{\bm{\Sigma}}_h^{c}$  such that
  $\underline{\bm{\sigma}}_{h,*}=\widetilde{\underline{\bm{\sigma}}}_h+\widetilde{\underline{\bm{\tau}}}_h
  \in \underline{\bm{H}}(\mathbf{div};\mathbb{S})$,
    \begin{equation}\label{sec4_eq10}
	\mathbf{div}\widetilde{\underline{\bm{\tau}}}_h =0 \ {\rm and} \ \|\widetilde{\underline{\bm{\tau}}}_h\|_{0}
	\lesssim \|\underline{\bm{\sigma}}_h-\widetilde{\underline{\bm{\sigma}}}_h\|_{0}.
	\end{equation}
\end{myLemma}

Next, we begin to show the optimal $\underline{\bm{L}}^2$ error
estimate. In \cite{Hu2015A}, the conforming mixed element
$\underline{\bm{\mathcal{P}}}^c_{k+1}$-$\bm{\mathcal{P}}_k$ ($k\geq
n$) is constructed on simplicial grids. Moreover, when $k\geq n$,
there exists a projection $\Pi_h^c$ satisfying (\cite{Hu2015Finite})
\begin{subequations}\label{sec4_eq13}
	\begin{align}
	\label{sec4_eq13_1}
&(\bm{{\rm
    div}}(\underline{\bm{\tau}}-\Pi_h^c\underline{\bm{\tau}}),\bm{v}_h)
  = 0    &\forall \underline{\bm{\tau}} \in
  \underline{\bm{H}}^1(\Omega;\mathbb{S}), \forall \bm{v}_h \in
  \bm{V}_h,\\
	\label{sec4_eq13_2}
  &\|\underline{\bm{\tau}}-\Pi_h^c\underline{\bm{\tau}}\|_0 \lesssim
  h^{k+2}|\underline{\bm{\tau}}|_{k+2} &\forall \underline{\bm{\tau}}
  \in \underline{\bm{H}}^{k+2}(\Omega;\mathbb{S}).
	\end{align}
\end{subequations}

\begin{myTheorem}\label{sec4_Theorem2}
Let the solutions $(\underline{\bm{\sigma}},\bm{u}) \in
\underline{\bm{\Sigma}}\times\bm{H}^{1}(\Omega)$ and
$(\underline{\bm{\sigma}}_h,\bm{u}_h)$ be the solutions of MDG
problems \eqref{weakBrinkman3}. Under the condition of Lemma
\ref{sec4_Lemma4}, we have
\begin{equation}\label{sec4_eq11}
  \|\underline{\bm{\sigma}}^d-\underline{\bm{\sigma}}_h^d\|_0\lesssim
  h^{k+2} (|\underline{\bm \sigma}|_{k+2} + |\bm{u}|_{k+1}).
\end{equation}
Further, the following superconvergence holds
\begin{equation} \label{equ:superconvergence}
  \kappa^{-1/2} \| \bm{P}_h\bm{u} - \bm{u}_h\|_0 \lesssim h^{k+2}
  (|\underline{\bm \sigma}|_{k+2} + |\bm{u}|_{k+1}).
\end{equation}
\end{myTheorem}
\begin{proof} 
According to \eqref{weakBrinkman1_eq2}, \eqref{BDM_operator2} and
Lemma \ref{sec4_Lemma3}, for any $\bm{v}_h \in \bm{V}_h$, we have
\begin{align*}
(\bm{f},\bm{v}_h)&=(\underline{\bm{\sigma}}_h,\underline{\bm{\varepsilon}}_h(\bm{v}_h))-
\langle\bm{\widehat{\underline\sigma}}_h\bm{n},\bm{v}_h\rangle_{\partial
  \mathcal {T}_h}+(\kappa^{-1}\bm{u}_h,\bm{v}_h)
=(\underline{\bm{\sigma}}_h,\underline{\bm{\nabla}}\bm{u}_h)-\langle\bm{\widehat{\underline{\sigma}}}_h\bm{n},\bm{v}_h\rangle_{\partial \mathcal {T}_h}+(\kappa^{-1}\bm{u}_h,\bm{v}_h)\\
&=(\widetilde{\underline{\bm{\sigma}}}_h,\underline{\bm{\nabla}}\bm{v}_h)
-\langle\widetilde{\underline{\bm{\sigma}}}_h\bm{n},\bm{v}_h\rangle_{\partial \mathcal {T}_h}+(\kappa^{-1}\bm{u}_h,\bm{v}_h)=-(\mathbf{div}\widetilde{\underline{\bm\sigma}}_h,\bm{v}_h)+(\kappa^{-1}\bm{u}_h,\bm{v}_h).
\end{align*}
Applying Lemma \ref{sec4_Lemma4}, there exist
$\widetilde{\underline{\bm{\tau}}}_h \in
\underline{\bm{\Sigma}}_h^{c}$ such that the symmetrized variable
$\underline{\bm{\sigma}}_{h,*}=\widetilde{\underline{\bm{\sigma}}}_h+\widetilde{\underline{\bm{\tau}}}_h$
is piecewise $\underline{\bm{\mathcal{P}}}^{\mathbb{M}}_{k+1}(K)$ and
$\underline{\bm{\sigma}}_{h,*} \in
\underline{\bm{H}}(\mathbf{div};\mathbb{S})$. Then,
\begin{equation}\label{sec4_eq12}
(\bm{{\rm div}}\underline{\bm{\sigma}}_{h,*},\bm{v}_h)
-(\kappa^{-1}\bm{u}_h,\bm{v}_h) = - (\bm{f},\bm{v}_h).
\end{equation}

 By \eqref{sec4_eq4}, \eqref{l2-operator}, \eqref{sec4_eq13_1} and
 \eqref{sec4_eq12}, it holds that
\begin{equation}\label{sec4_eq14}
(\bm{{\rm
 div}}(\Pi_h^c\underline{\bm{\sigma}}-\underline{\bm{\sigma}}_{h,*}),\bm{v}_h)
  - \kappa^{-1}(\bm{e}_{\bm u},\bm{v}_h)=0   \qquad \forall \bm{v}_h
  \in \bm{V}_h.
\end{equation}
Taking $\bm{v}_h=\bm{e}_{\bm u}$ in \eqref{sec4_eq14} and
$\underline{\bm{\tau}}_h=\Pi_h^c\underline{\bm{\sigma}}-\underline{\bm{\sigma}}_{h,*}$
in \eqref{sec4_eq3.1}, due to the $\underline{H}(\mathbf{div})$
conformity of
$\Pi_h^c\underline{\bm{\sigma}}-\underline{\bm{\sigma}}_{h,*}$, we
obtain
\begin{equation}\label{sec4_eq15}
\frac{1}{2}(\underline{\bm{\sigma}}^d-\underline{\bm{\sigma}}_h^d,\Pi_h^c\underline{\bm{\sigma}}^d-\underline{\bm{\sigma}}_{h,*}^d)
  + \kappa^{-1}\|\bm{e}_{\bm u}\|_0^2= 0,
\end{equation}
which implies that 
\begin{equation*}
\|\underline{\bm{\sigma}}^d-\underline{\bm{\sigma}}_h^d\|_0^2 +
2\kappa^{-1}\|\bm{e}_{\bm u}\|_0^2 \leq
(\underline{\bm{\sigma}}^d-\underline{\bm{\sigma}}_h^d,\underline{\bm{\sigma}}^d-\Pi_h^c\underline{\bm{\sigma}}^d)
  +(\underline{\bm{\sigma}}^d-\underline{\bm{\sigma}}_h^d,\underline{\bm{\sigma}}_{h,*}^d-\underline{\bm{\sigma}}_h^d).
\end{equation*}
Then, it holds
$$
\|\underline{\bm{\sigma}}^d-\underline{\bm{\sigma}}_h^d\|_0 +
\kappa^{-1/2}\|\bm{e}_{\bm u}\|_0 \lesssim
\|\underline{\bm{\sigma}}-\Pi_h^c\underline{\bm{\sigma}}\|_0+
\|\underline{\bm{\sigma}}_h-\underline{\bm{\sigma}}_{h,*}\|_0.
$$
Combining the above inequality with Lemma \ref{sec4_Lemma3}, Lemma
\ref{sec4_Lemma4} and \eqref{equ:error-sigma2}, we get
\begin{align*}
\|\underline{\bm{\sigma}}^d-\underline{\bm{\sigma}}_h^d\|_0 +
\kappa^{-1/2}\|\bm{e}_{\bm u}\|_0 &\lesssim
\|\underline{\bm{\sigma}}-\Pi_h^c\underline{\bm{\sigma}}\|_0+
\|\underline{\bm{\sigma}}_h-\underline{\bm{\sigma}}_{h,*}\|_0\\
&\lesssim \|\underline{\bm{\sigma}}-\Pi_h^c\underline{\bm{\sigma}}\|_0+\|\widetilde{\underline{\bm{\tau}}}_h\|_0+\|\widetilde{\underline{\bm{\sigma}}}_h-\underline{\bm{\sigma}}_h\|_0\\
&\lesssim
\|\underline{\bm{\sigma}}-\Pi_h^c\underline{\bm{\sigma}}\|_0+\sum_{e
  \in
    \mathcal{E}_h^i}h^{1/2}\|(\widehat{\underline{\bm{\sigma}}}_h-\underline{\bm{\sigma}}_h)\bm{n}\|_{e}\\
&\lesssim \|\underline{\bm{\sigma}}-\Pi_h^c\underline{\bm{\sigma}}\|_0 + h|\underline{\bm{\sigma}}_h|_* 
\lesssim h^{k+2}(|\underline{\bm \sigma}|_{k+2} + |\bm{u}|_{k+1}),
\end{align*}
which finishes the proof.
\end{proof}

\begin{myRemark}\label{sec4_Remark1}
For two dimensional case, when $k \geq 2$, the conforming mixed
element pair
$\underline{\bm{\mathcal{P}}}^c_{k+1}$-$\bm{\mathcal{P}}_k$ is
constructed in \cite{Hu2014A} and the Scott-Vogelius elements
$\bm{\mathcal{P}}^c_{k+2}$-$\mathcal{P}_{k+1}$ are stable
(cf. \cite{Scott1985Norm,Guzman2019The}). Therefore, the optimal
$\underline{\bm{L}}^2$ estimate for pseudostress holds when $k\geq 2$
in two dimensional case. 
\end{myRemark} 

\begin{myTheorem} \label{sec4_Theorem3}
Assume that the solution of \eqref{weakBrinkman} satisfies
$(\underline{\bm{\sigma}}, \bm{u})\in \underline{\bm{H}}^{k+2}(\Omega)
\times \bm{H}^{k+1}(\Omega)$.  Under the condition of Theorem
\ref{sec4_Theorem2}, the solution of the mixed DG problem
\eqref{weakBrinkman3} satisfies 
\begin{equation} \label{sec4_eq17} 
\|\mathrm{tr}(\underline{\bm \sigma} -
\underline{\bm{\sigma}}_h)\|_0 \lesssim \widetilde{\kappa}^{-1/2}
h^{k+2}(|\underline{\bm \sigma}|_{k+2} + |\bm{u}|_{k+1}).
\end{equation}
Further, we have the optimal error estimate for pressure $p$ in
$L^2$-norm
\begin{equation} \label{equ:L2-p} 
  \|p-p_h\|_0 \lesssim \widetilde{\kappa}^{-1/2}
  h^{k+2}(|\underline{\bm \sigma}|_{k+2} + |\bm{u}|_{k+1}).
\end{equation}
\end{myTheorem}
\begin{proof} Taking $\bm{v}_h=\bm{{\rm
div}}_h(\Pi_h^c\underline{\bm{\sigma}}-\underline{\bm{\sigma}}_{h,*})$
in \eqref{sec4_eq14}, by the superconvergence result in
\eqref{equ:superconvergence}, we obtain
	\begin{equation}\label{sec4_eq18}
\|\bm{{\rm
div}}_h(\Pi_h^c\underline{\bm{\sigma}}-\underline{\bm{\sigma}}_{h,*})\|_0 \lesssim {\kappa}^{-1}\|\bm{e}_{\bm u}\|_0 \lesssim\widetilde{\kappa}^{-1/2}
h^{k+2}(|\underline{\bm \sigma}|_{k+2} + |\bm{u}|_{k+1}).
	\end{equation}
Here, we use the fact that $\kappa^{-1} \leq \widetilde{\kappa}^{-1}$.
Again, by Lemma \ref{sec4_Lemma4}, Lemma \ref{sec4_Lemma3} and Theorem
\ref{thm:error2}, we have 
\begin{equation*}
  \|\bm{\sigma}_h-\underline{\bm{\sigma}}_{h,*}\|_0 \lesssim\sum_{e
    \in\mathcal{E}_h^i}h_e^{1/2}\|(\widehat{\underline{\bm{\sigma}}}_h-\underline{\bm{\sigma}}_h)\bm{n}\|_{e}
    \leq h|\underline{\bm{\sigma}}_h|_* \lesssim
    h^{k+2}(|\underline{\bm \sigma}|_{k+2} + |\bm{u}|_{k+1}),
	\end{equation*}
whence, by \eqref{sec4_eq13_2}, \eqref{sec4_eq11} and \eqref{sec4_eq18},
$$
\begin{aligned}
  \|\Pi_h^c\underline{\bm{\sigma}}^d-\underline{\bm{\sigma}}_{h,*}^d\|_0
  &\leq \|\Pi_h^c\underline{\bm\sigma}^d - \underline{\bm\sigma}^d\|_0
  + \| \underline{\bm\sigma}^d - \underline{\bm\sigma}_h^d\|_0 +
  \|\underline{\bm\sigma}_h^d - \underline{\bm\sigma}_{h,*}^d\|_0 \\
  &\leq \|\Pi_h^c\underline{\bm\sigma} - \underline{\bm\sigma}\|_0
  + \| \underline{\bm\sigma}^d - \underline{\bm\sigma}_h^d\|_0 +
  \|\underline{\bm\sigma}_h - \underline{\bm\sigma}_{h,*}\|_0 
  \lesssim h^{k+2}(|\underline{\bm \sigma}|_{k+2} + |\bm{u}|_{k+1}).
\end{aligned}
$$ 
Note here that
$\Pi_h^c\underline{\bm{\sigma}}-\underline{\bm{\sigma}}_{h,*} \in
\underline{\bm\Sigma}_h \cap \underline{\bm H}(\mathbf{div},\Omega;
\mathbb{S})$, using Lemma \ref{sec3_Lemma0}, \eqref{sec4_eq18}, one
gets
\begin{equation*}
\|\Pi_h^c\underline{\bm{\sigma}}-\underline{\bm{\sigma}}_{h,*}\|_0 
\lesssim
\|\Pi_h^c\underline{\bm{\sigma}}^d-\underline{\bm{\sigma}}_{h,*}^d\|_0 
+\|\bm{{\rm
div}}_h(\Pi_h^c\underline{\bm{\sigma}}-\underline{\bm{\sigma}}_{h,*})\|_0
\lesssim \widetilde{\kappa}^{-1/2}h^{k+2}(|\underline{\bm
\sigma}|_{k+2} + |\bm{u}|_{k+1}),
	\end{equation*}
which yields
$$
\|\underline{\bm\sigma}
-\underline{\bm\sigma}_{h}\|_0\leq\|\underline{\bm{\sigma}}-\Pi_h^c\underline{\bm{\sigma}}\|_0+\|\Pi_h^c\underline{\bm{\sigma}}-\underline{\bm{\sigma}}_{h,*}\|_0+
\|\underline{\bm{\sigma}}_{h,*}-\bm{\sigma}_h\|_0 \lesssim
\widetilde{\kappa}^{-1/2}h^{k+2}(|\underline{\bm \sigma}|_{k+2} +
|\bm{u}|_{k+1}).
$$
Note that \eqref{P}, we can define the numerical solution of pressure
by the postprocessed approximation $p_h = -\frac{1}{n}{\rm
tr}(\underline{\bm{\sigma}}_h)$.  The optimal $L^2$ estimate for
pressure \eqref{equ:L2-p} then follows from the fact that $\|p -
p_h\|_0 \lesssim \|\mathrm{tr}(\underline{\bm\sigma} -
\underline{\bm\sigma}_h)\|_0$. 
\end{proof}
	

\section{Numerical examples}\label{numericalexamples}
In this section, we present some numerical results to illustrate the
reliability, accuracy, and flexibility of the MDG method
\eqref{weakBrinkman4}. The numerical results presented below are
obtained by using Fenics software (cf. \cite{Langtangen2016Solving}).
For simplicity, we consider the triangular meshes in the two-dimensional
case and the discontinuous Galerkin pair
$\underline{\bm{\mathcal{P}}}^{\mathbb{S}}_{k+1}$-$\bm{\mathcal{P}}_{k}$
for all the numerical examples.

Example \ref{example_1} is employed to illustrate the performance of
the MDG scheme \eqref{weakBrinkman4} for different permeability with
the polynomial degrees $k = 0,1,2$.  For the variable permeability,
Example \ref{example_2} is used to test the accuracy of the MDG scheme
\eqref{weakBrinkman4} with different viscosity. Example
\ref{example_3} and Example \ref{example_4} are utilized to show the
behavior of MDG scheme \eqref{weakBrinkman4} for the Brinkman problem
in a region with different contrast permeability.

\begin{myExam}\label{example_1} Consider the steady Brinkman problem
\eqref{Brinkman} in a square domain $(0,1)\times(0,1)$ with a
homogeneous boundary condition that $\bm{u}=\bm{0}$ on $\Gamma$. The
right hand side function $\bm{f}$  and the exact stress function
$\underline{\bm{\sigma}}$ are selected such that the exact solution is
given by
\begin{align*}
\left\{\begin {array}{lll}
u_1(x,y,t) = x^2(x-1)^2y(y-1)(2y-1),\\
u_2(x,y,t) = -x(x-1)(2x-1)y^2(y-1)^2,\\
p(x,y,t) = (2x-1)(2y-1).
\end{array}\right.
\end{align*}
\end{myExam}

This example aims at testing the accuracy and reliability of the MDG
method for fixed viscosity and different permeability. Set $1/h=4$,
$8$, $16$, $32$ and $\nu=1$. We compute the numerical solutions
$(\underline{\bm{\sigma}}_h,\bm{u}_h)$ on uniform meshes with
$\kappa^{-1}=10^{-3},10^0$, and $10^3$. The numerical results of
$\|\bm{u}-\bm{u}_h\|_0$,
$\|\underline{\bm{\sigma}}-\underline{\bm{\sigma}}_h\|_{\underline{\bm{\Sigma}}_h}$,
$\|\underline{\bm{\sigma}}-\underline{\bm{\sigma}}_h\|_0$ and
$\|p-p_h\|_0$ for finite element pairs
$\underline{\bm{\mathcal{P}}}^{\mathbb{S}}_{k+1}$-$\bm{\mathcal{P}}_{k}$
$(k=0,1 {\rm \ and\ } 2)$ are given in Table \ref{tab1}--Table
\ref{tab3}, respectively. The numerical results confirm the optimal
convergence orders, which are consistent with the theoretical results
developed in Section \ref{errorestimation}.  We can see that the MDG
scheme is very stable with respected to different permeability.

\begin{table}[h!]
	\begin{center}
		\def\temptablewidth{1.0\textwidth}
		{\rule{\temptablewidth}{0.8pt}}
 \begin{tabular*}{\temptablewidth}{@{\extracolsep{\fill}}cccccccccc}
			$\kappa^{-1}$	& $1/h$
			& $\|\bm{u}-\bm{u}_h\|_0$                   &  Order
			& $\|\underline{\bm{\sigma}}-\underline{\bm{\sigma}}_h\|_{\underline{\bm{\Sigma}}_h}$                   &  Order
			& $\|\underline{\bm{\sigma}}-\underline{\bm{\sigma}}_h\|_0$    &  Order
			& $\|p-p_h\|_0$                   &  Order\\
            \hline
			\multirow{4}{*}{$10^{-3}$}
			&4&3.20004e-03&---& 4.86047e-01&---& 1.08555e-01&---& 4.02386e-02&--- \\   
			&8&1.66698e-03&0.94&2.49013e-01&0.96&4.91672e-02&1.14&1.76784e-02&1.19\\
			&16&8.40052e-04&0.99&1.25359e-01&0.99&2.36517e-02&1.06&8.40333e-03&1.07\\
			&32&4.20839e-04&1.00&6.27953e-02&1.00&1.16891e-02&1.02&4.13811e-03&1.02\\
		   \hline
           \multirow{4}{*}{$10^{0}$}
           &4&3.18199e-03&---& 4.83635e-01&---& 1.08124e-01&---& 4.01034e-02&--- \\   
           &8&1.65770e-03&0.94&2.47759e-01&0.96&4.88790e-02&1.15&1.75800e-02&1.19\\
           &16&8.35424e-04&0.99&1.24732e-01&0.99&2.34947e-02&1.06&8.34834e-03&1.07\\
           &32&4.18528e-04&1.00&6.24823e-02&1.00&1.16088e-02&1.02&4.10978e-03&1.02\\
		   \hline
           \multirow{4}{*}{$10^{3}$}
           &4&1.54374e-03&---& 7.99766e-02&---& 7.95933e-02&---& 3.23179e-02&--- \\   
           &8&8.38782e-04&0.88&2.27374e-02&1.81&2.24375e-02&1.83&9.01438e-03&1.84\\
           &16&4.27995e-04&0.97&6.78529e-03&1.74&6.51937e-03&1.78&2.53637e-03&1.83\\
           &32&2.15076e-04&0.99&2.51625e-03&1.43&2.32453e-03&1.49&8.61808e-04&1.56\\
		\end{tabular*}
		{\rule{\temptablewidth}{1pt}}
	\end{center}
\vspace*{-20pt}
	\caption{ Numerical errors and orders in Example \ref{example_1} for $\underline{\bm{\mathcal{P}}}^{\mathbb{S}}_{1}$-$\bm{\mathcal{P}}_{0}$ element}\label{tab1}
\end{table}

\begin{table}[h!]
	\begin{center}
		\def\temptablewidth{1.0\textwidth}
		{\rule{\temptablewidth}{0.8pt}}
		 \begin{tabular*}{\temptablewidth}{@{\extracolsep{\fill}}cccccccccc}
			$\kappa^{-1}$	& $1/h$
			& $\|\bm{u}-\bm{u}_h\|_0$                   &  Order
			& $\|\underline{\bm{\sigma}}-\underline{\bm{\sigma}}_h\|_{\underline{\bm{\Sigma}}_h}$                   &  Order
			& $\|\underline{\bm{\sigma}}-\underline{\bm{\sigma}}_h\|_0$    &  Order
			& $\|p-p_h\|_0$                   &  Order\\
			\hline
			\multirow{4}{*}{$10^{-3}$}
			&4&4.36672e-04&---& 3.90029e-02&---& 3.16881e-03&---& 1.13778e-03&--- \\   
			&8&1.15747e-04&1.92&1.12384e-02&1.80&5.28805e-04&2.58&1.79899e-04&2.66\\
			&16&2.93956e-05&1.98&2.99769e-03&1.91&9.40039e-05&2.49&3.16202e-05&2.51\\
			&32&7.37888e-06&1.99&7.68566e-04&1.96&2.02008e-05&2.22&6.96908e-06&2.18\\
			\hline
			\multirow{4}{*}{$10^{0}$}
			&4&4.36637e-04&---& 3.89813e-02&---& 3.16559e-03&---& 1.13689e-03&--- \\   
			&8&1.15745e-04&1.92&1.12359e-02&1.79&5.28345e-04&2.58&1.79751e-04&2.66\\
			&16&2.93954e-05&1.98&2.99735e-03&1.91&9.38787e-05&2.49&3.15755e-05&2.51\\
			&32&7.37884e-06&1.99&7.68503e-04&1.96&2.01643e-05&2.22&6.95601e-06&2.18\\
			\hline
			\multirow{4}{*}{$10^{3}$}
			&4&4.34313e-04&---& 2.87962e-03&---& 2.40618e-03&---& 9.05464e-04&--- \\   
			&8&1.15593e-04&1.91&6.22572e-04&2.21&3.97226e-04&2.60&1.36738e-04&2.73\\
			&16&2.93699e-05&1.98&1.45246e-04&2.10&6.23008e-05&2.67&2.03260e-05&2.75\\
			&32&7.37266e-06&1.99&3.55679e-05&2.03&1.05275e-05&2.57&3.47418e-06&2.55\\
		\end{tabular*}
	{\rule{\temptablewidth}{1pt}}
\end{center}
\vspace*{-20pt}
\caption{ Numerical errors and orders in  Example \ref{example_1} for $\underline{\bm{\mathcal{P}}}^{\mathbb{S}}_{2}$-$\bm{\mathcal{P}}_{1}$ element}\label{tab2}
\end{table}

\begin{table}[h!]
	\begin{center}
		\def\temptablewidth{1.0\textwidth}
		{\rule{\temptablewidth}{0.8pt}}
		 \begin{tabular*}{\temptablewidth}{@{\extracolsep{\fill}}cccccccccc}
			$\kappa^{-1}$	& $1/h$
			& $\|\bm{u}-\bm{u}_h\|_0$                   &  Order
			& $\|\underline{\bm{\sigma}}-\underline{\bm{\sigma}}_h\|_{\underline{\bm{\Sigma}}_h}$                   &  Order
			& $\|\underline{\bm{\sigma}}-\underline{\bm{\sigma}}_h\|_0$    &  Order
			& $\|p-p_h\|_0$                   &  Order\\
			\hline
			\multirow{4}{*}{$10^{-3}$}
			&4&7.56819e-05&---& 8.64697e-03&---& 3.81029e-04&---& 1.31551e-04&---\\
			&8&1.00692e-05&2.91&1.15393e-03&2.91&3.11168e-05&3.61&1.07342e-05&3.62\\
			&16&1.28159e-06&2.97&1.44753e-04&2.99&2.23659e-06&3.80&7.69790e-07&3.80\\
			&32&1.60948e-07&2.99&1.79936e-05&3.01&1.49415e-07&3.90&5.13626e-08&3.91\\
			\hline
			\multirow{4}{*}{$10^{0}$}
			&4&7.56816e-05&---& 8.64553e-03&---& 3.80925e-04&---& 1.31524e-04&--- \\   
			&8&1.00692e-05&2.91&1.15385e-03&2.91&3.11128e-05&3.61&1.07330e-05&3.62\\
			&16&1.28159e-06&2.97&1.44749e-04&2.99&2.23650e-06&3.80&7.69761e-07&3.80\\
			&32&1.60948e-07&2.99&1.79934e-05&3.01&1.49413e-07&3.90&5.13620e-08&3.91\\
			\hline
			\multirow{4}{*}{$10^{3}$}
			&4&7.56400e-05&---& 4.76787e-04&---& 3.31906e-04&---& 1.18158e-04&--- \\   
			&8&1.00688e-05&2.91&5.54499e-05&3.10&2.81247e-05&3.56&9.81496e-06&3.59\\
			&16&1.28159e-06&2.97&6.58331e-06&3.07&2.15065e-06&3.71&7.42738e-07&3.72\\
			&32&1.60948e-07&2.99& 8.02772e-07&3.04& 1.47697e-07&3.86& 5.08168e-08&3.87\\
		\end{tabular*}
		{\rule{\temptablewidth}{1pt}}
	\end{center}
	\vspace*{-20pt}
	\caption{ Numerical errors and orders in  Example \ref{example_1} for $\underline{\bm{\mathcal{P}}}^{\mathbb{S}}_{3}$-$\bm{\mathcal{P}}_{2}$ element}\label{tab3}
\end{table}

\begin{myExam}\label{example_2} In this example, choose
$\Omega=(0,1)\times(0,1)$ for the steady Brinkman problem
\eqref{Brinkman}. The right hand side function $\bm{f}$, the exact
stress function $\underline{\bm{\sigma}}$ and boundary condition
$\bm{g}$ are selected such that the exact solution is given by
\begin{align*}
\left\{\begin {array}{lll}
u_1(x,y,t) = 2\sin(\pi x)^2\sin(\pi y)\cos(\pi y),\\
u_2(x,y,t) = -2\sin(\pi y)^2\sin(\pi x)\cos(\pi x),\\
p(x,y,t) = \cos(\pi x)\cos(\pi y).
\end{array}\right.
\end{align*}
\end{myExam}

This example aims at testing the accuracy and reliability of the MDG
method for different viscosity and variable permeability. The MDG
finite element pair
$\underline{\bm{\mathcal{P}}}^{\mathbb{S}}_{2}$-$\bm{\mathcal{P}}_{1}$
is employed in the numerical discretization on uniform meshes. The
parameter $\kappa^{-1} = 1000(\sin(\pi x)+1.1)$ and set $1/h=4$, $8$,
$16$, $32$. Then, for $\nu = 10^{-2}$, $10^{-1}$ and $10^{0}$, we
present the numerical results of $\|\bm{u}-\bm{u}_h\|_0$,
$\|\underline{\bm{\sigma}}-\underline{\bm{\sigma}}_h\|_{\underline{\bm{\Sigma}}_h}$,
$\|\underline{\bm{\sigma}}-\underline{\bm{\sigma}}_h\|_0$ and
$\|p-p_h\|_0$ in Table \ref{tab4}.

\begin{table}[h!]
	\begin{center}
		\def\temptablewidth{1.0\textwidth}
		{\rule{\temptablewidth}{0.8pt}}
		 \begin{tabular*}{\temptablewidth}{@{\extracolsep{\fill}}cccccccccc}
			$\nu$	& $1/h$
			& $\|\bm{u}-\bm{u}_h\|_0$                   &  Order
			& $\|\underline{\bm{\sigma}}-\underline{\bm{\sigma}}_h\|_{\underline{\bm{\Sigma}}_h}$                   &  Order
			& $\|\underline{\bm{\sigma}}-\underline{\bm{\sigma}}_h\|_0$    &  Order
			& $\|p-p_h\|_0$                   &  Order\\
			\hline
			\multirow{4}{*}{$10^{-2}$}
			&4& 6.31075e-02&---&  2.64131e-01&---&  2.82987e-02&---&  1.09014e-02&---\\
			&8& 1.65010e-02&1.94& 8.51447e-02&1.63& 5.24654e-03&2.43& 1.91774e-03&2.51\\
			&16&4.15519e-03&1.99& 2.20791e-02&1.95& 7.67132e-04&2.77& 2.65942e-04&2.85\\
			&32&1.04098e-03&2.00& 5.57152e-03&1.99& 1.02641e-04&2.90& 3.45945e-05&2.94\\
			\hline
			\multirow{4}{*}{$10^{-1}$}
			&4& 6.22748e-02&---&  6.62081e-01&---&  7.67147e-02&---&  3.00923e-02&---\\
			&8& 1.63746e-02&1.93& 1.57909e-01&2.07& 1.01413e-02&2.92& 3.77929e-03&2.99\\
			&16&4.14629e-03&1.98& 3.98776e-02&1.99& 1.25819e-03&3.01& 4.36425e-04&3.11\\
			&32&1.03994e-03&2.00& 1.01836e-02&1.97& 1.58246e-04&2.99& 5.24614e-05&3.06\\
			\hline
			\multirow{4}{*}{$10^{0}$}
			&4& 6.22280e-02&---&  5.31556e+00&---&  5.76297e-01&---&  2.29330e-01&---\\
			&8& 1.63708e-02&1.93& 1.27720e+00&2.06& 6.42025e-02&3.17& 2.40393e-02&3.25\\
			&16&4.14614e-03&1.98& 3.22914e-01&1.98& 7.49026e-03&3.10& 2.53180e-03&3.25\\
			&32&1.03993e-03&2.00& 8.12870e-02&1.99& 9.69215e-04&2.95& 3.08775e-04&3.04\\
		\end{tabular*}
		{\rule{\temptablewidth}{1pt}}
	\end{center}
	\vspace*{-20pt}
	\caption{ Numerical errors and orders in  Example \ref{example_2} for $\underline{\bm{\mathcal{P}}}^{\mathbb{S}}_{2}$-$\bm{\mathcal{P}}_{1}$ element}\label{tab4}
\end{table}

\begin{figure}[!ht]
	\begin{center}
    	\subfigure[]{
		    \label{bar_domain}
		    \centering
		    \includegraphics[width=2.0in]{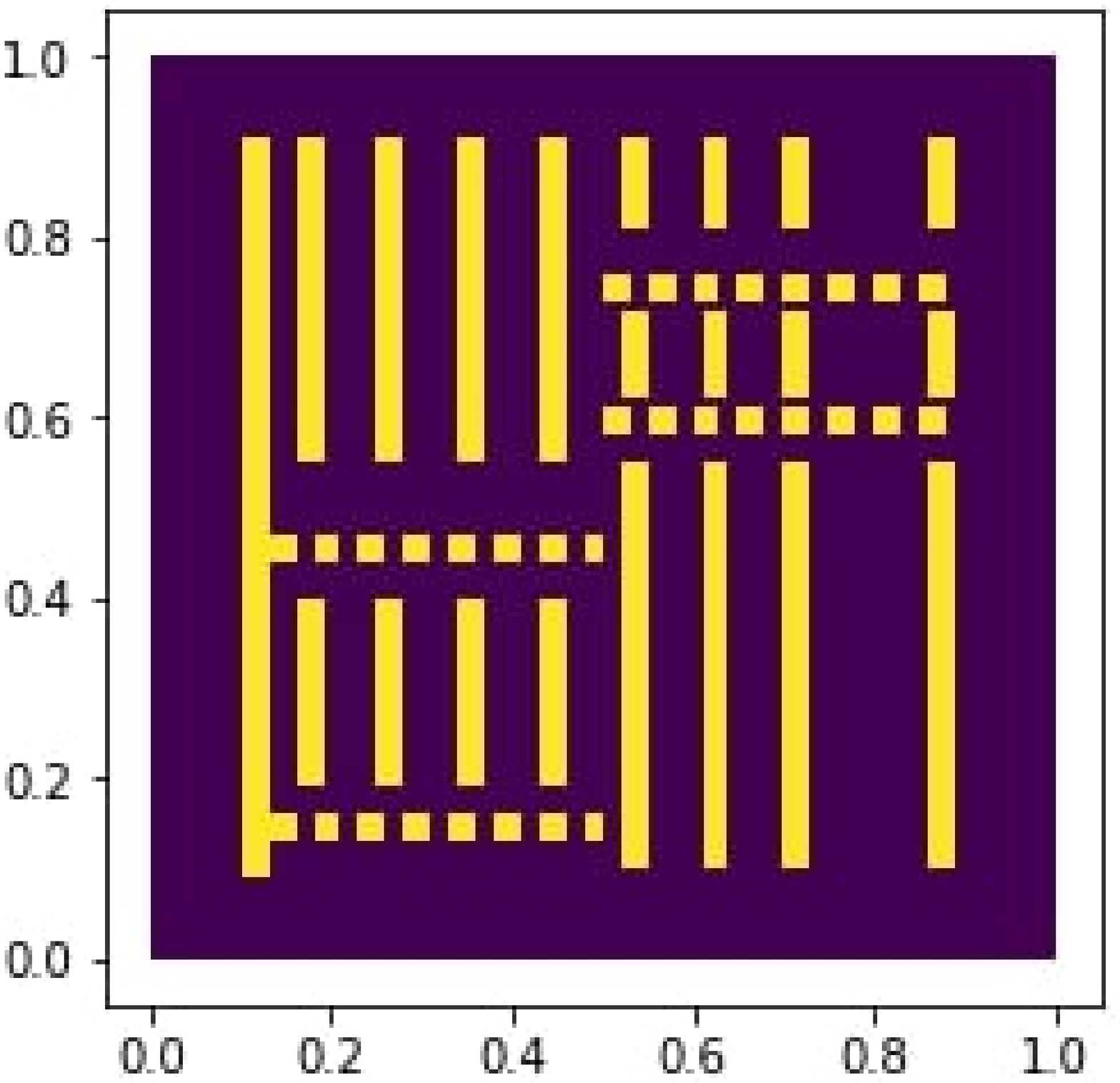}
     	}
    	\subfigure[]{
	    	\label{vuggy_domain}
	    	\centering
	    	\includegraphics[width=2.0in]{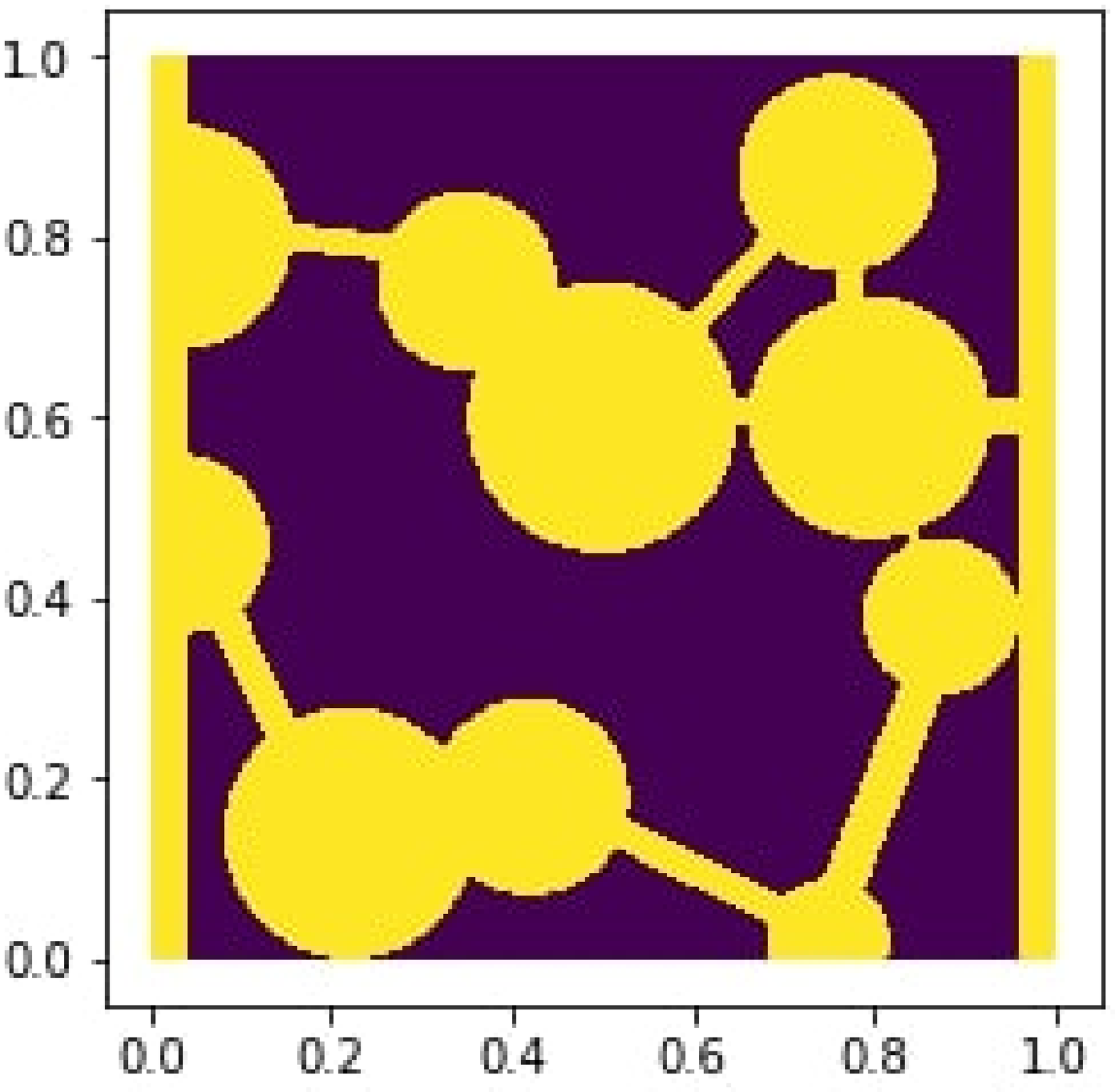}
    	}
	\end{center}
	\vspace*{-15pt}
	\caption{(a)The profile of $\kappa^{-1}$ in Example \ref{example_3}; (b) The profile of $\kappa^{-1}$ in Example \ref{example_4}.}
\end{figure}

Examples \ref{example_3} and \ref{example_4} do not have analytical
solutions, so we do not list the convergence order as shown in the
first two examples. We mention that the similar test of the profile of
$\kappa^{-1}$ can be found in other literature
\cite{Iliev_2011_Variational,Lin2014A}. In the following two examples,
a mesh $100\times 100$ is used and the data setting is designed as
follows:  $\Omega=(0,1)\times(0,1)$, $\bm{f} = \bm{0}$, $\nu =10^{-2}$
and $\bm{g}=(1,0)^t$.  The Brinkman problems \eqref{weakBrinkman4} are
solved in a region with different contrast permeability.

\begin{myExam}\label{example_3} For this test case, the profile of
	$\kappa^{-1}$ is plotted in Figure \ref{bar_domain} with
	$\kappa^{-1}=10,10^3, 10^5$ in yellow region and $\kappa^{-1}=1$ in
	the purple region. (\cite{Lin2014A}).
\end{myExam}

For $\kappa^{-1}=10, 10^3, 10^5$ in the yellow region and
$\kappa^{-1}=1$ in purple region of \ref{bar_domain}, the first and
the second components of the velocity obtained by MDG method with
$\underline{\bm{\mathcal{P}}}^{\mathbb{S}}_{2}$-$\bm{\mathcal{P}}_{1}$
element are presented in Figure \ref{bar_u1} and Figure \ref{bar_u2},
respectively.  The stress intensity and pressure profiles are showed
in Figure \ref{bar_s} and Figure \ref{bar_p}.

\begin{figure}[!ht]
	\begin{center}
		\subfigure[$\kappa^{-1}=10$ in the yellow region]{
			\label{bar1_velocity1}
			\centering
			\includegraphics[width=2.08in]{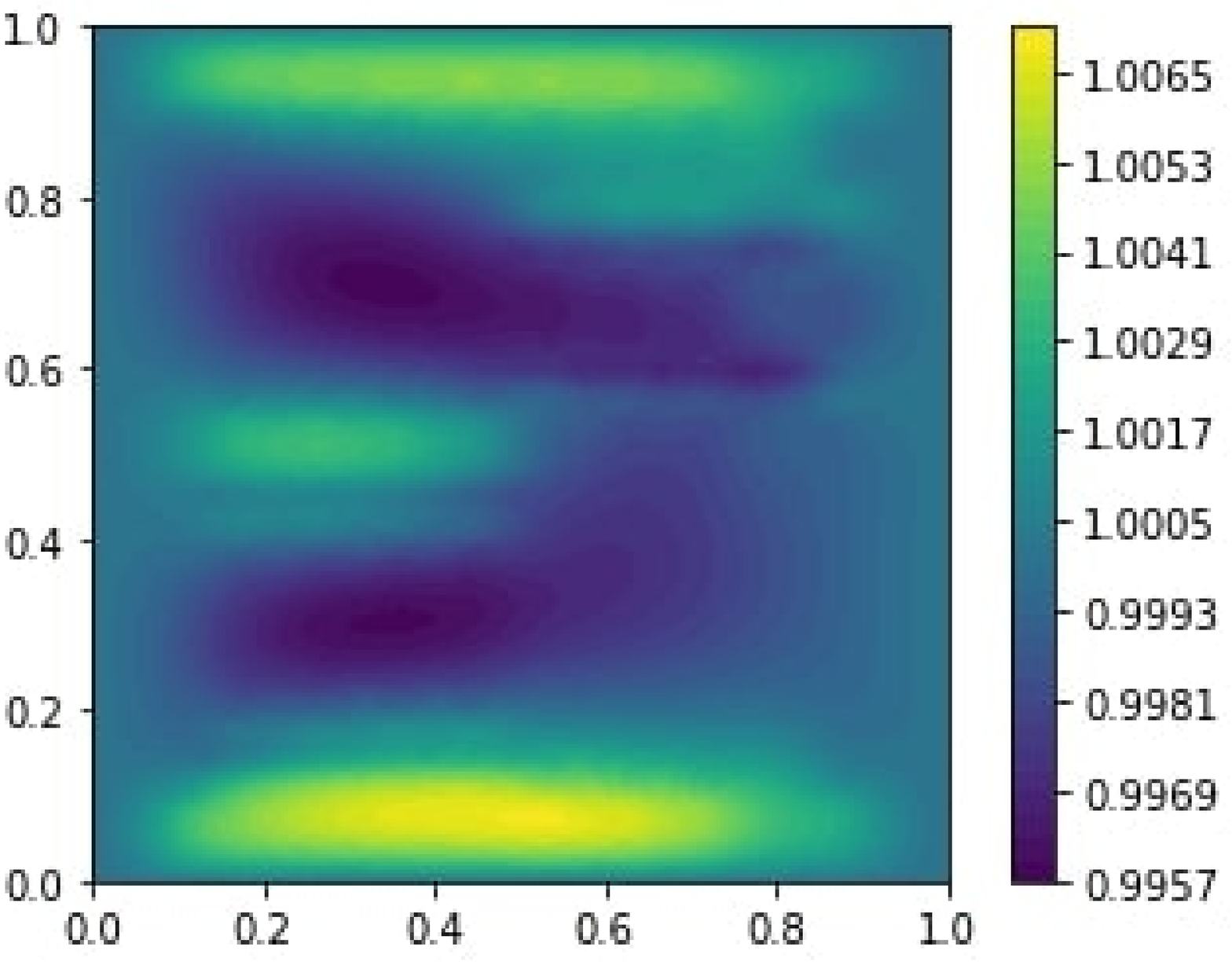}
		}
		\subfigure[$\kappa^{-1}=10^3$ in the yellow region]{
			\label{bar3_velocity1}
			\centering
			\includegraphics[width=2.0in]{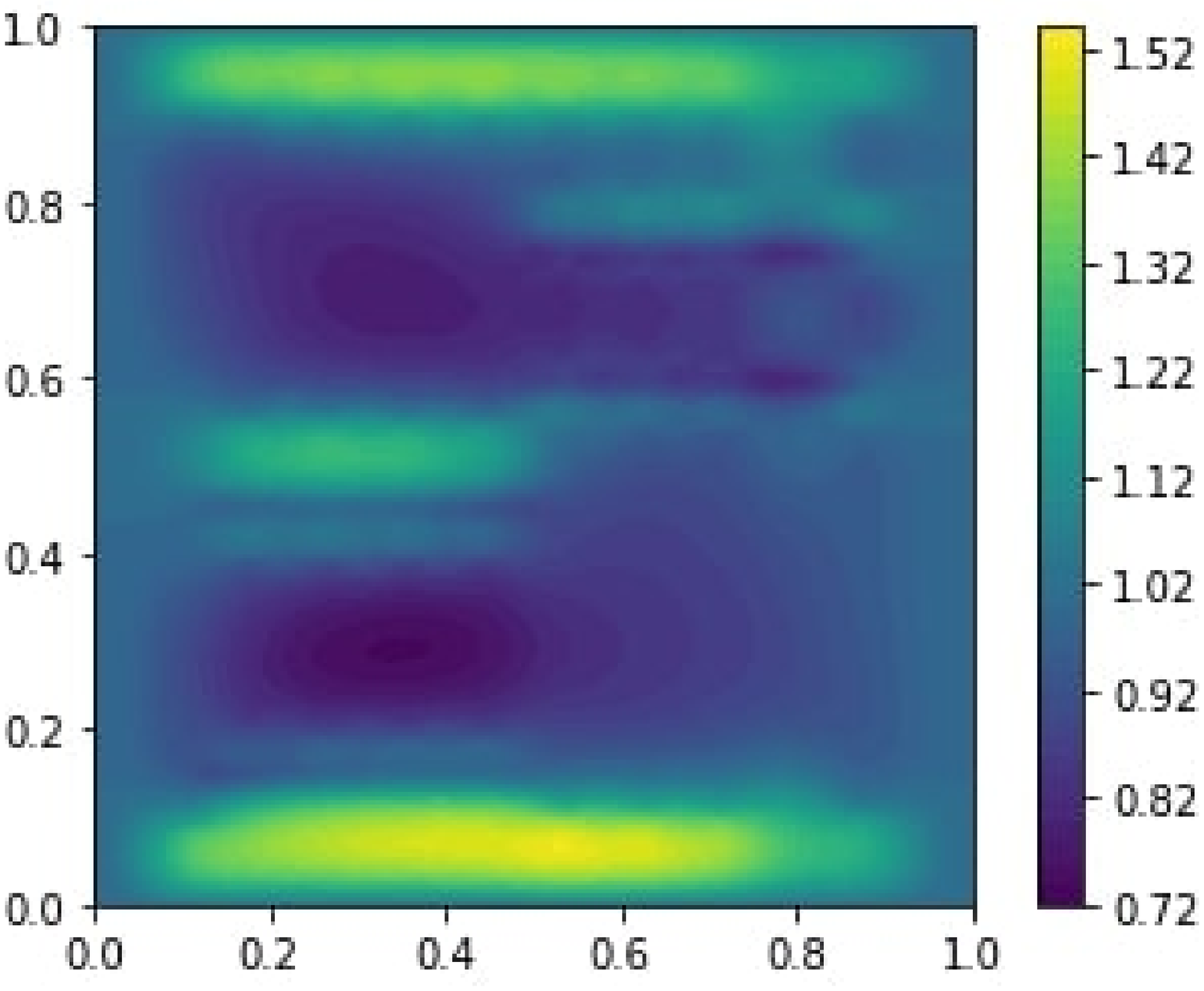}
		}
		\subfigure[$\kappa^{-1}=10^5$ in the yellow region]{
			\label{bar5_velocity1}
			\centering
			\includegraphics[width=1.98in]{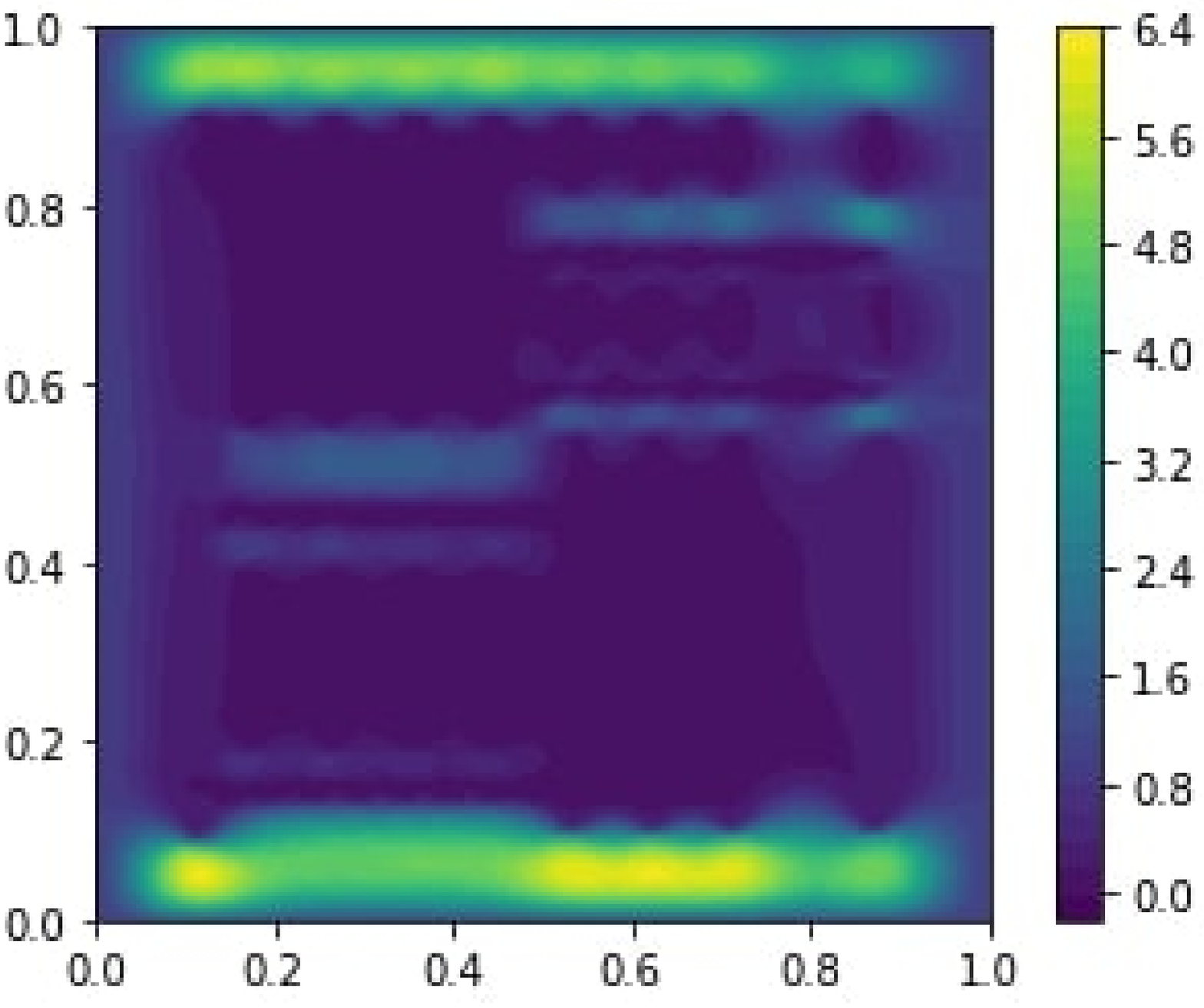}
		}
	\end{center}
	\vspace*{-15pt}
	\caption{Distributions of $u_1$ with $\kappa^{-1}=1$ in purple region in Example \ref{example_3}.} \label{bar_u1}
\end{figure}

\begin{figure}[!ht]
	\begin{center}
		\subfigure[$\kappa^{-1}=10$ in the yellow region]{
			\label{bar1_velocity2}
			\centering
			\includegraphics[width=2.16in]{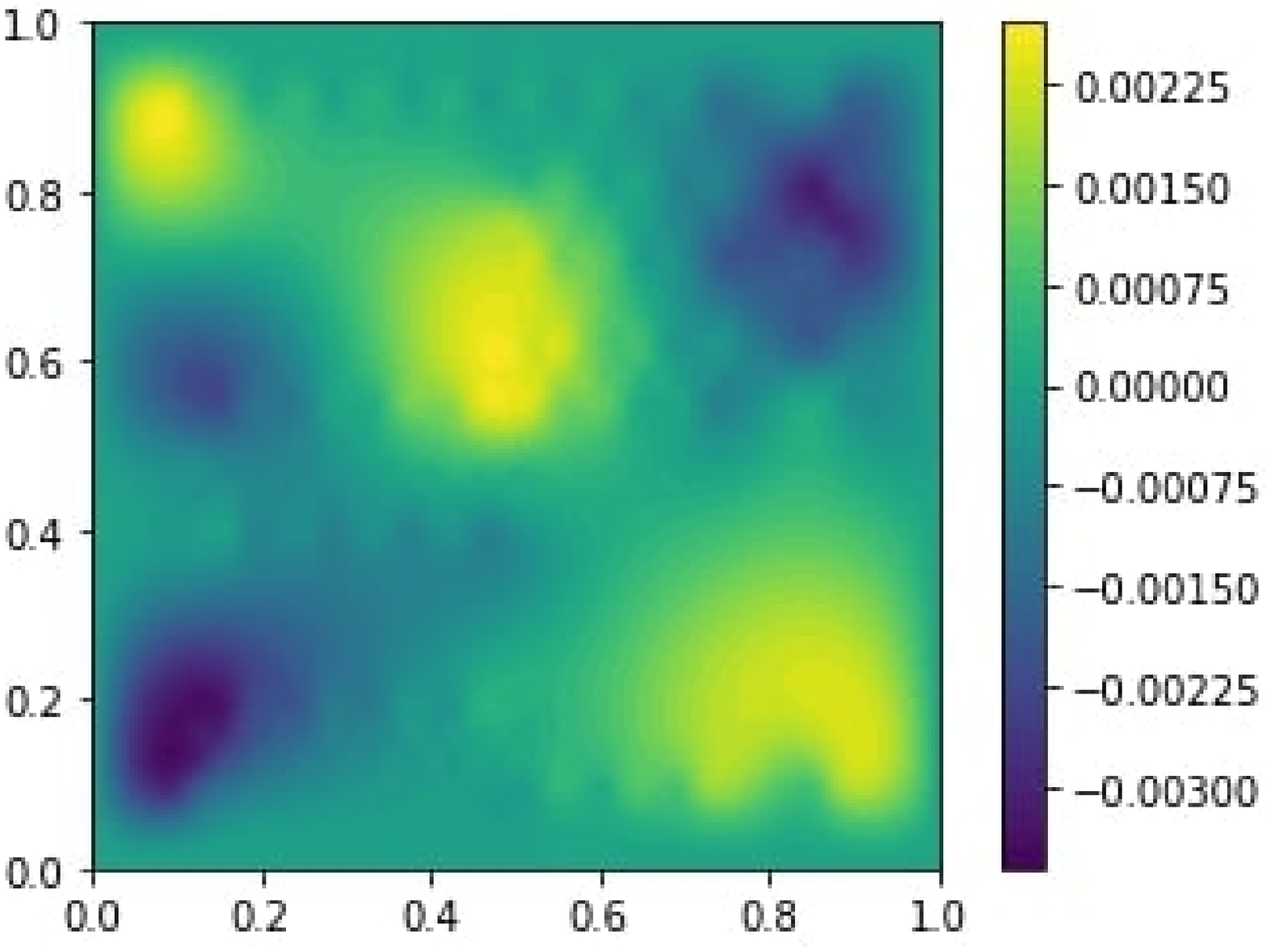}
		}
		\subfigure[$\kappa^{-1}=10^3$ in the yellow region]{
			\label{bar3_velocity2}
			\centering
			\includegraphics[width=2.05in]{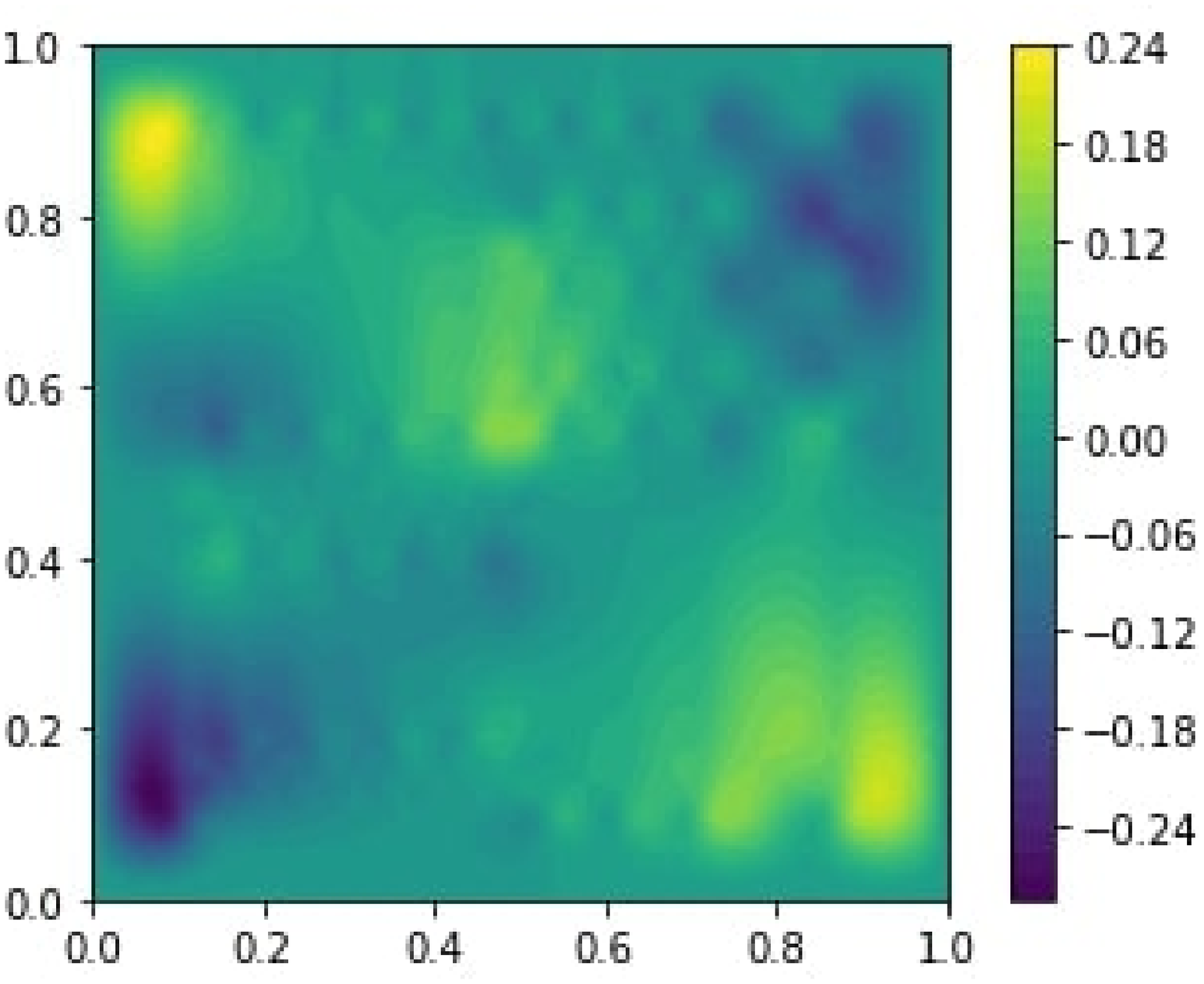}
		}
		\subfigure[$\kappa^{-1}=10^5$ in the yellow region]{
			\label{bar5_velocity2}
			\centering
			\includegraphics[width=1.98in]{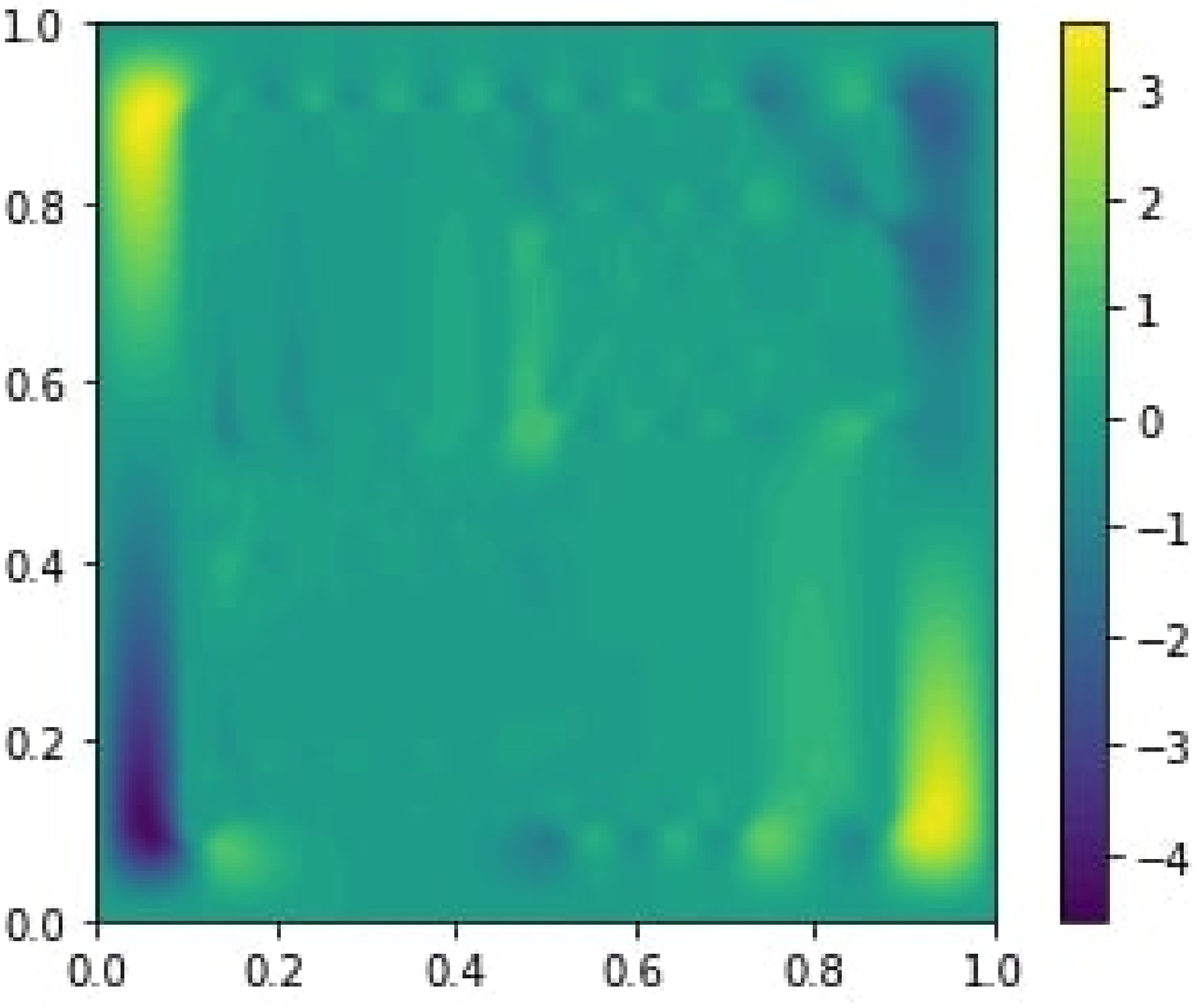}
		}
	\end{center}
	\vspace*{-15pt}
	\caption{Distributions of $u_2$ with $\kappa^{-1}=1$ in purple region in Example \ref{example_3}.} \label{bar_u2}
\end{figure}

\begin{figure}[!ht]
	\begin{center}
		\subfigure[$\kappa^{-1}=10$ in the yellow region]{
			\label{bar1_stress}
			\centering
			\includegraphics[width=2.08in]{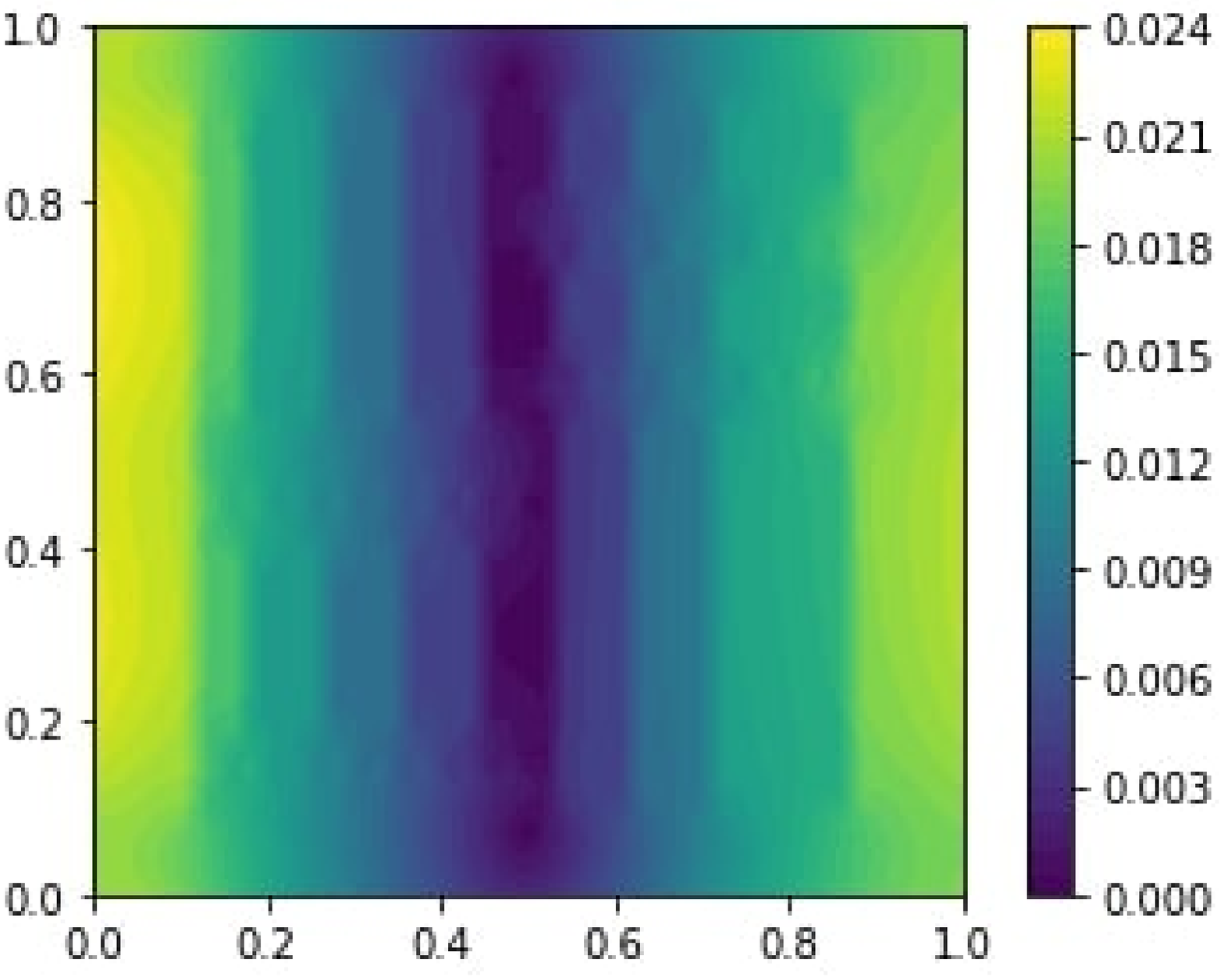}
		}
		\subfigure[$\kappa^{-1}=10^3$ in the yellow region]{
			\label{bar3_stress}
			\centering
			\includegraphics[width=1.99in]{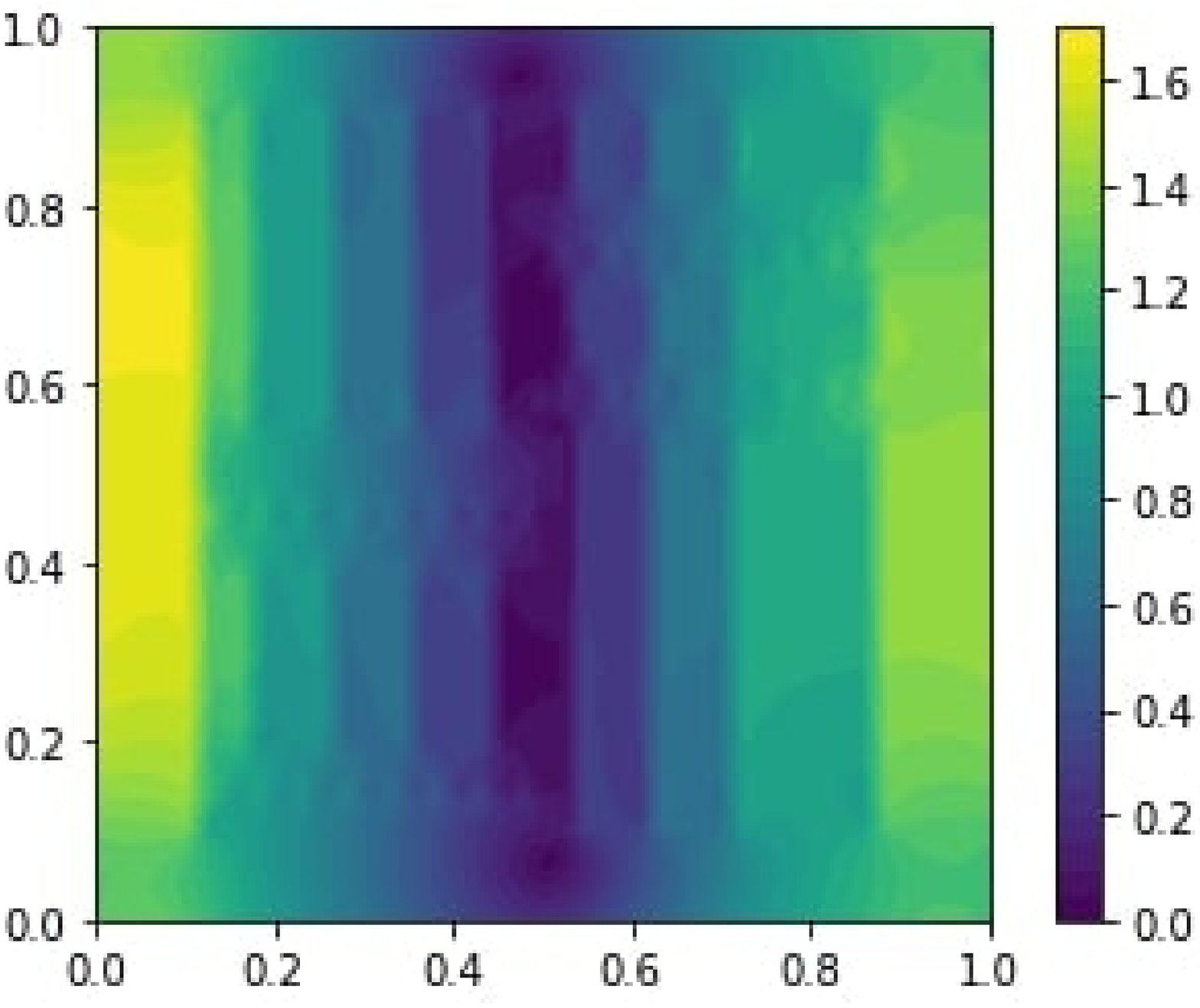}
		}
		\subfigure[$\kappa^{-1}=10^5$ in the yellow region]{
			\label{bar5_stress}
			\centering
			\includegraphics[width=1.97in]{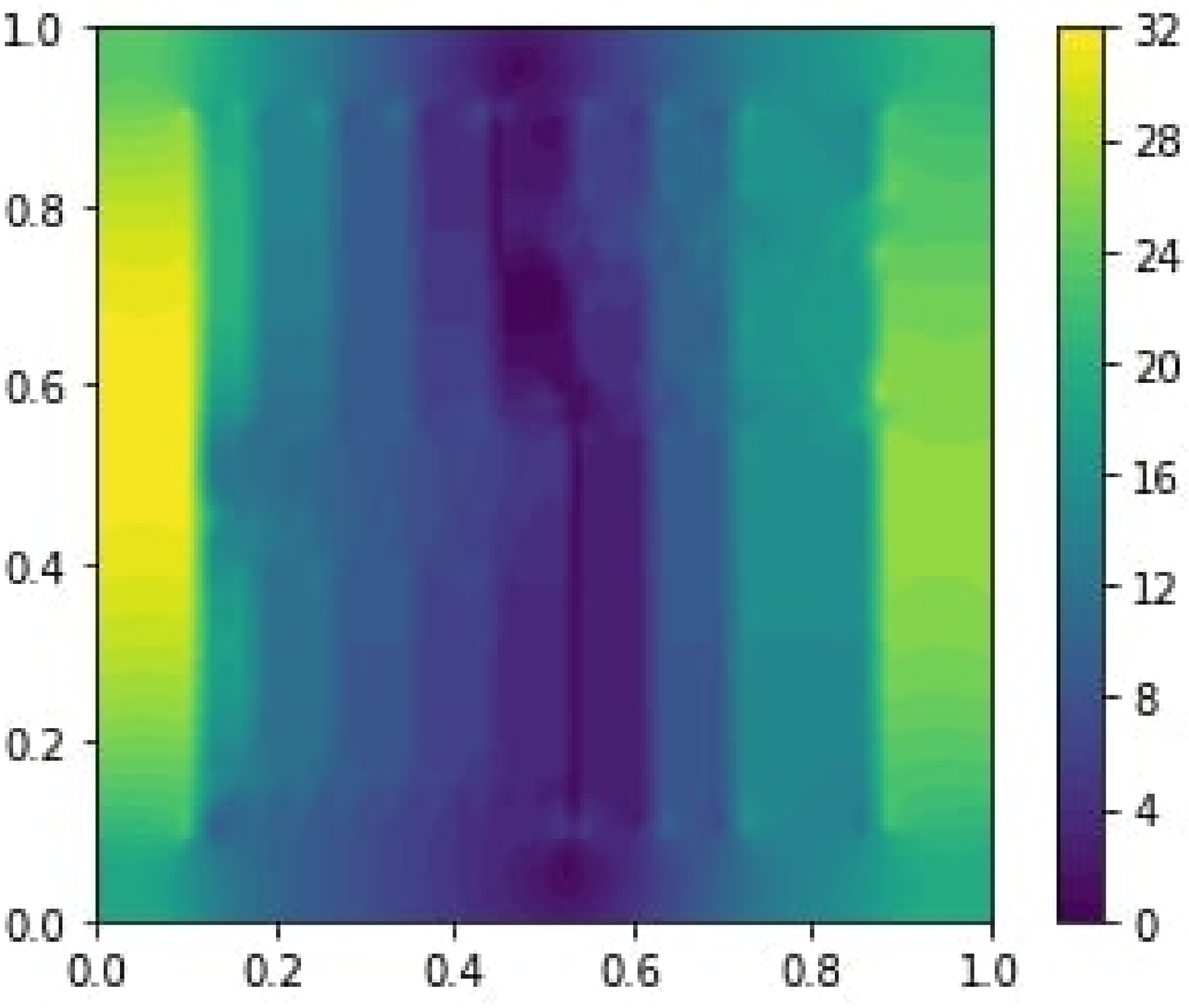}
		}
	\end{center}
	\vspace*{-15pt}
	\caption{Distributions of $\underline{\bm{\sigma}}$ with $\kappa^{-1}=1$ in purple region in Example \ref{example_3}.} \label{bar_s}
\end{figure}

\begin{figure}[!ht]
	\begin{center}
		\subfigure[$\kappa^{-1}=10$ in the yellow region]{
			\label{bar1_pressure}
			\centering
			\includegraphics[width=2.05in]{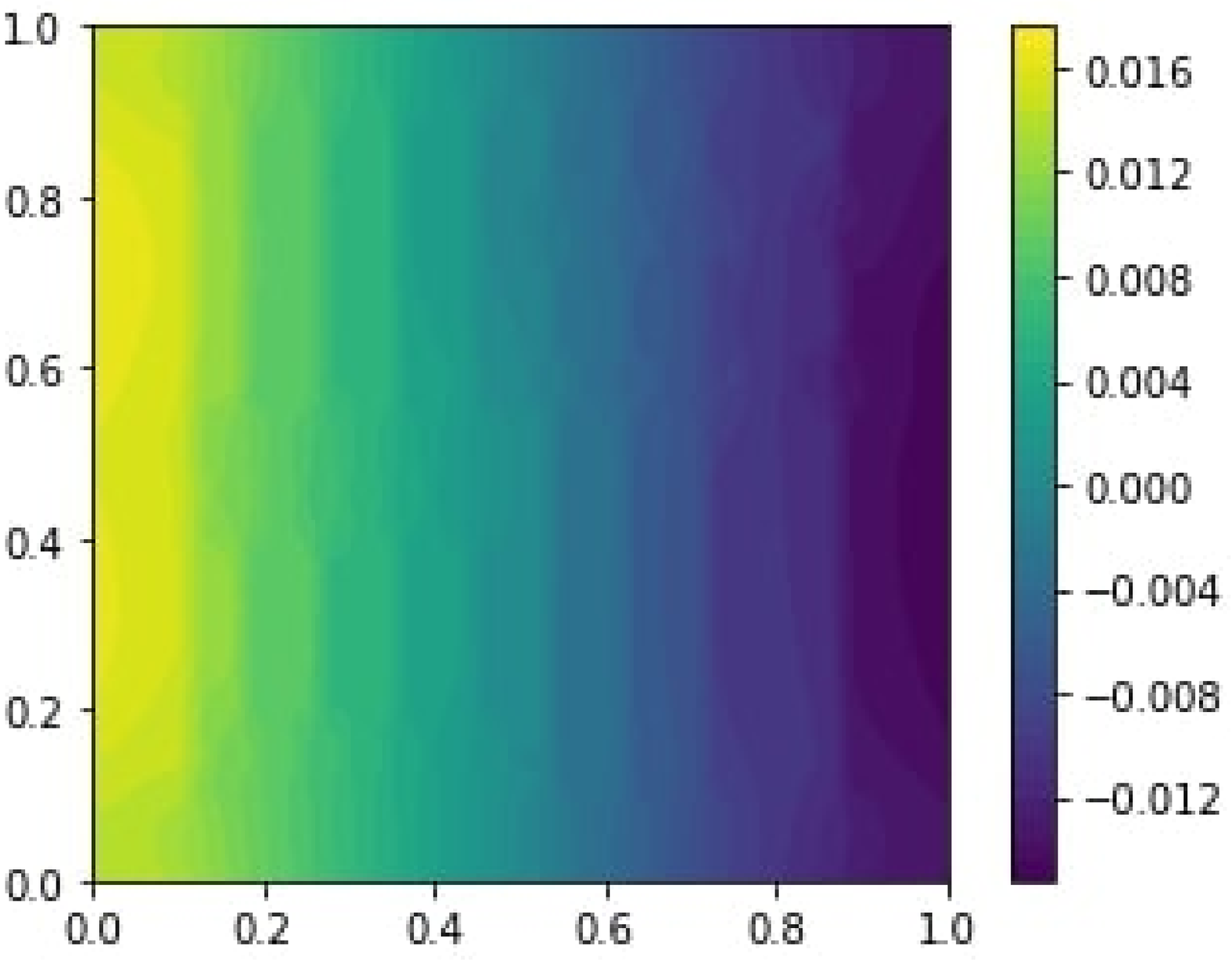}
		}
		\subfigure[$\kappa^{-1}=10^3$ in the yellow region]{
			\label{bar3_pressure}
			\centering
			\includegraphics[width=2.0in]{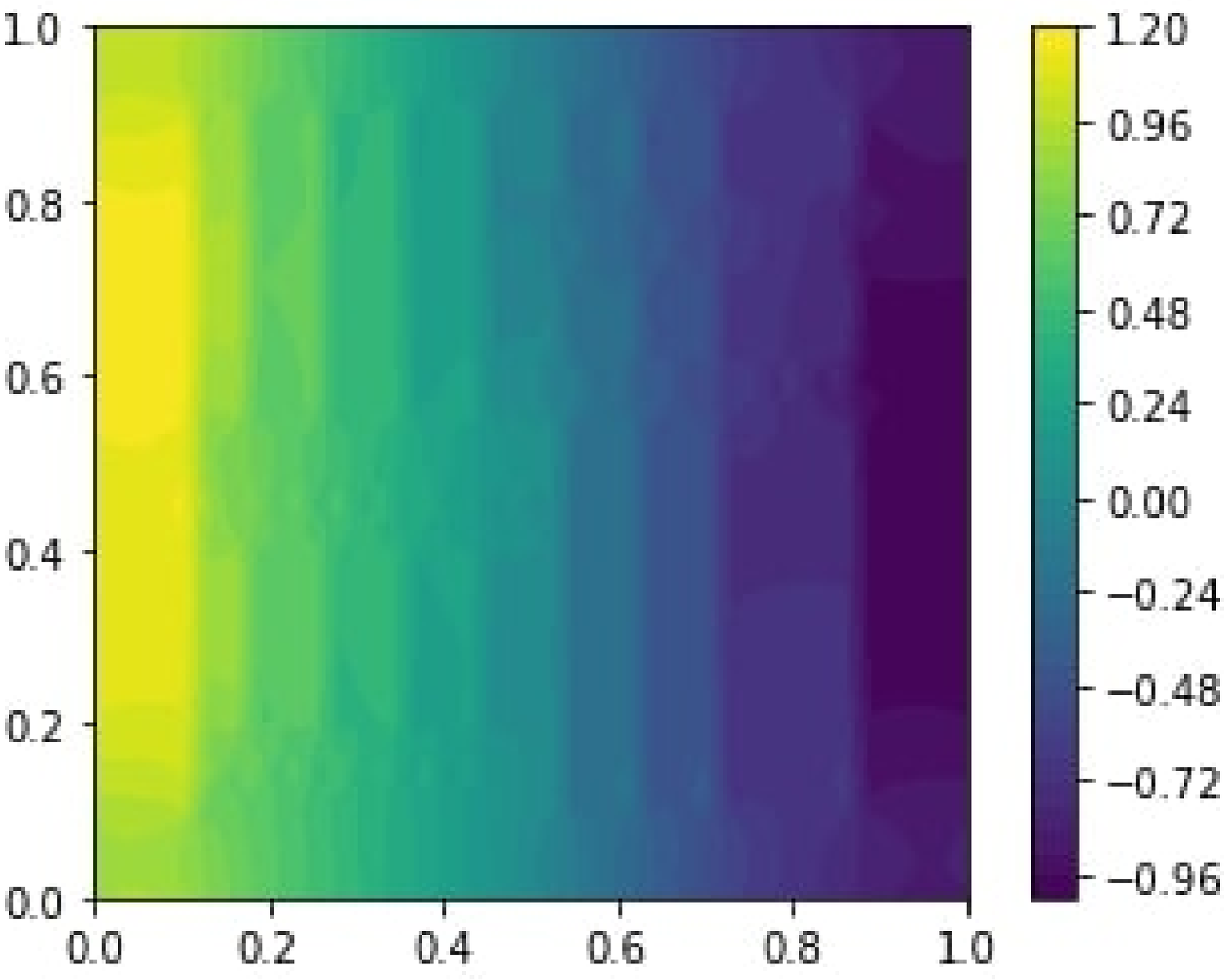}
		}
		\subfigure[$\kappa^{-1}=10^5$ in the yellow region]{
			\label{bar5_pressure}
			\centering
			\includegraphics[width=1.96in]{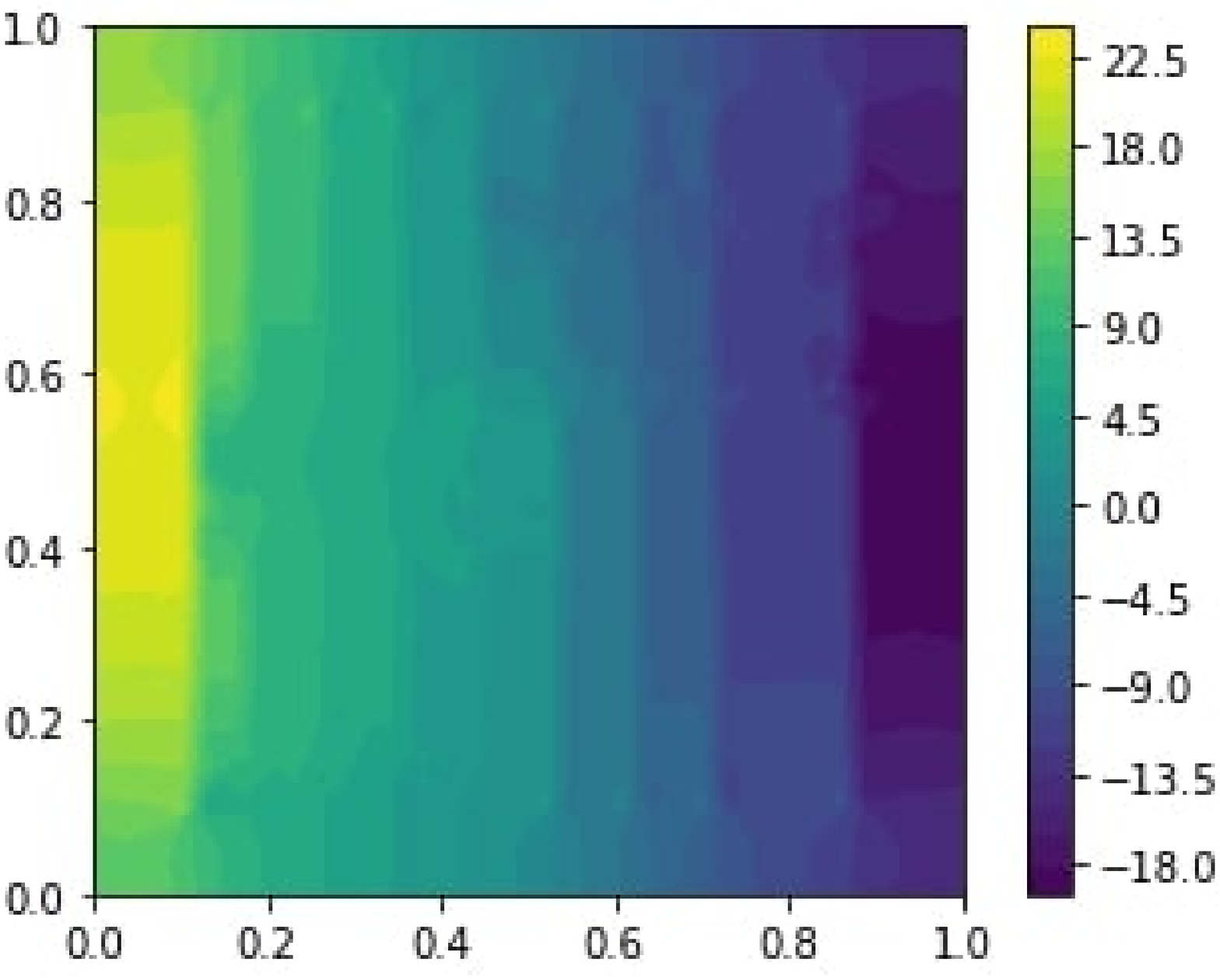}
		}
	\end{center}
	\vspace*{-15pt}
	\caption{Distributions of $p$ with $\kappa^{-1}=1$ in purple region in Example \ref{example_3}.} \label{bar_p}
\end{figure}

\begin{myExam}\label{example_4} In this example, the profile of
$\kappa^{-1}$ is plotted in Figure \ref{vuggy_domain} with $\kappa^{-1}=10,10^3,10^5$ in the
purple region and $\kappa^{-1}=1$ in yellow region. (\cite{Iliev_2011_Variational,Lin2014A}).
\end{myExam}

For $\kappa^{-1}=10,10^3,10^5$ in the
purple region and $\kappa^{-1}=1$ in yellow region of \ref{vuggy_domain}, the first and the second
components of the velocity obtained by MDG method with
$\underline{\bm{\mathcal{P}}}^{\mathbb{S}}_{2}$-$\bm{\mathcal{P}}_{1}$
element are presented in Figure \ref{vuggy_u1} and Figure
\ref{vuggy_u2}, respectively. The stress intensity and pressure
profiles are showed in Figure \ref{vuggy_s} and Figure \ref{vuggy_p}.

From Figure \ref{bar_u1} and Figure \ref{bar_u2}, we can see that the
velocity of fluid in the purple region is faster and in the yellow
region become slower as the contrast permeability increases. While,
from Figure \ref{vuggy_u1} and Figure \ref{vuggy_u2}, we can see that
the velocity of fluid in the yellow region become faster and in the
purple region become slower with the contrast permeability increasing.
Figure \ref{vuggy_s} and Figure \ref{vuggy_p} show that the high
contrast permeability gives rise to the large velocity difference, and
the intensity and pressure increase rapidly.

Both Example \ref{example_3} and Example \ref{example_4} indicate that
the contrast permeability is higher, the change of the velocity,
pressure, and stress is greater. And they show the robustness,
accuracy, and flexibility of the MDG method for the Brinkman problem.

\begin{figure}[!ht]
	\begin{center}
    	\subfigure[$\kappa^{-1}=10$ in the purple region]{
		    \label{vuggy1_velocity1}
		    \centering
		    \includegraphics[width=2.09in]{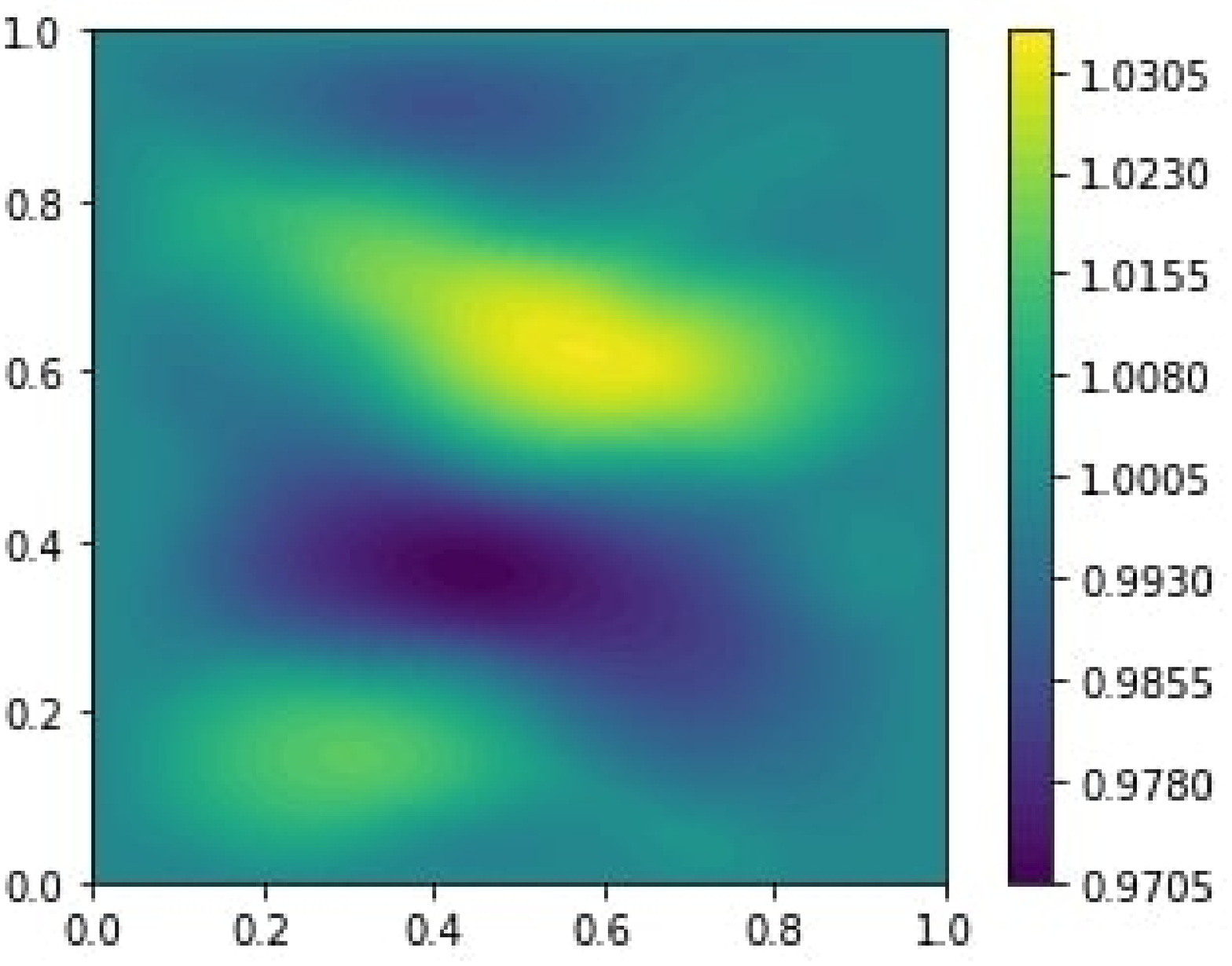}
     	}
    	\subfigure[$\kappa^{-1}=10^3$ in the purple region]{
	    	\label{vuggy3_velocity1}
	    	\centering
	    	\includegraphics[width=1.98in]{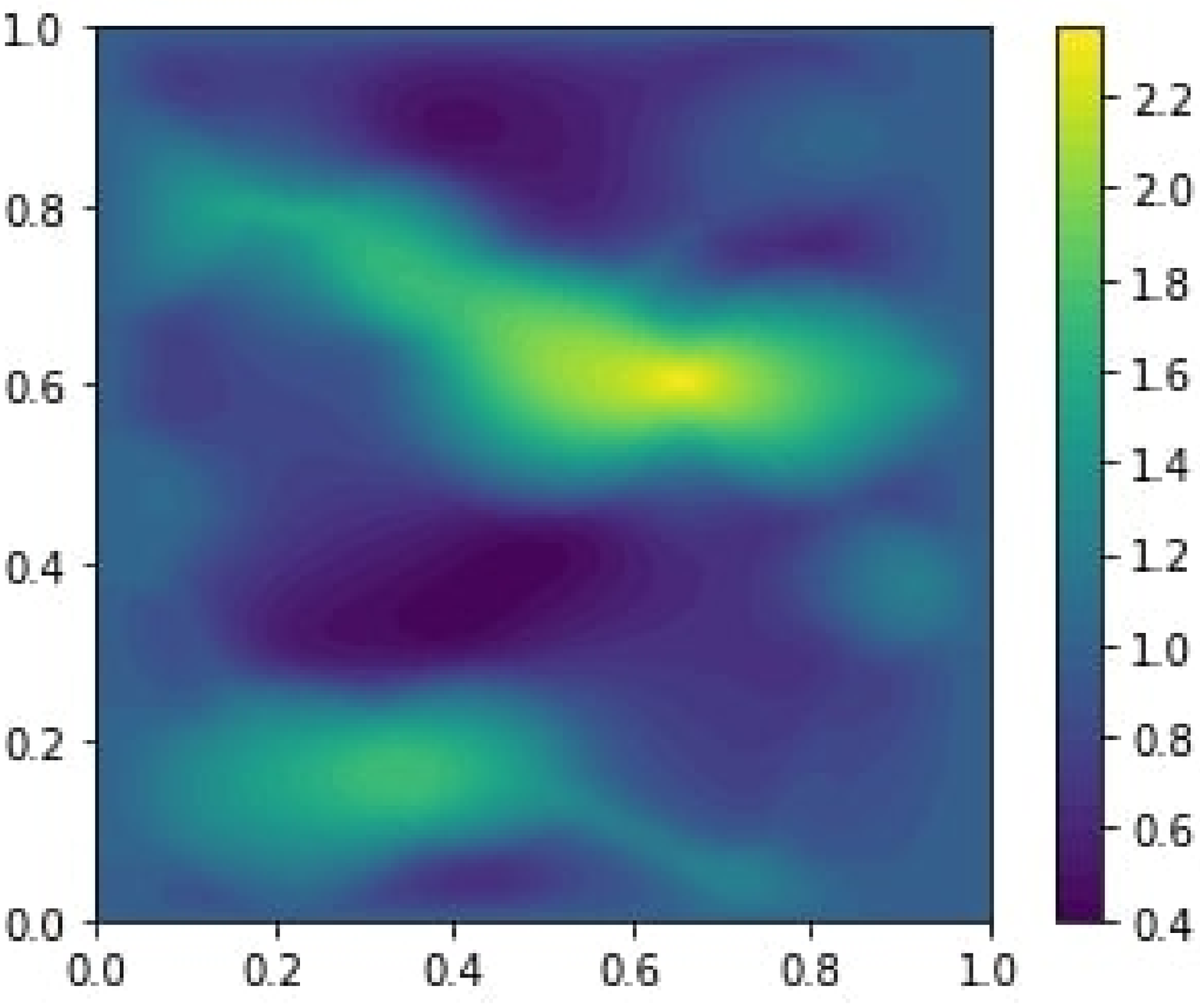}
    	}
		\subfigure[$\kappa^{-1}=10^5$ in the purple region]{
			\label{vuggy6_velocity1}
			\centering
			\includegraphics[width=2.0in]{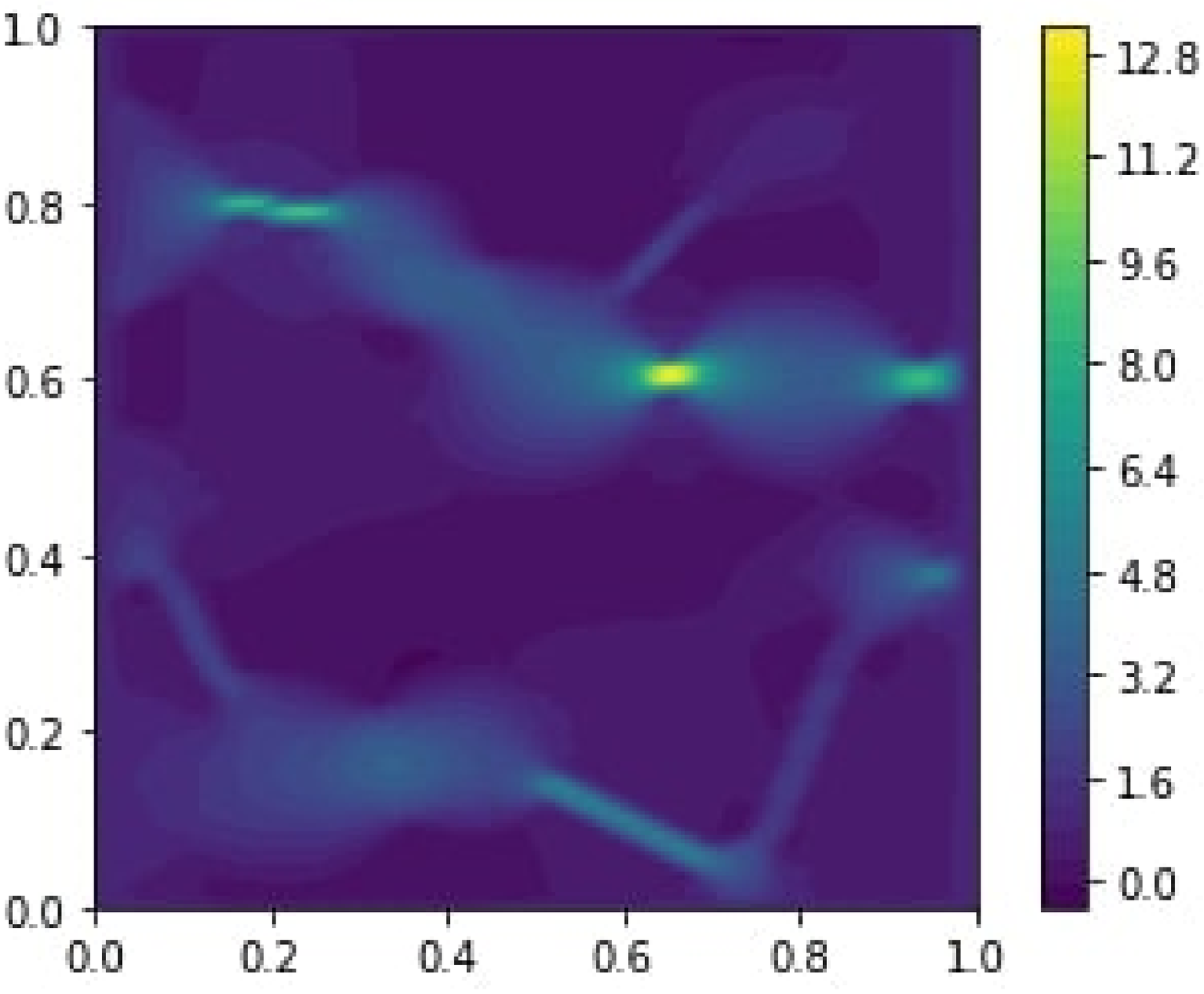}
		}
	\end{center}
	\vspace*{-15pt}
	\caption{Distributions of $u_1$ with $\kappa^{-1}=1$ in yellow region in Example \ref{example_4}.} \label{vuggy_u1}
\end{figure}

\begin{figure}[!ht]
	\begin{center}
    	\subfigure[$\kappa^{-1}=10$ in the purple region]{
		    \label{vuggy1_velocity2}
		    \centering
		    \includegraphics[width=2.11in]{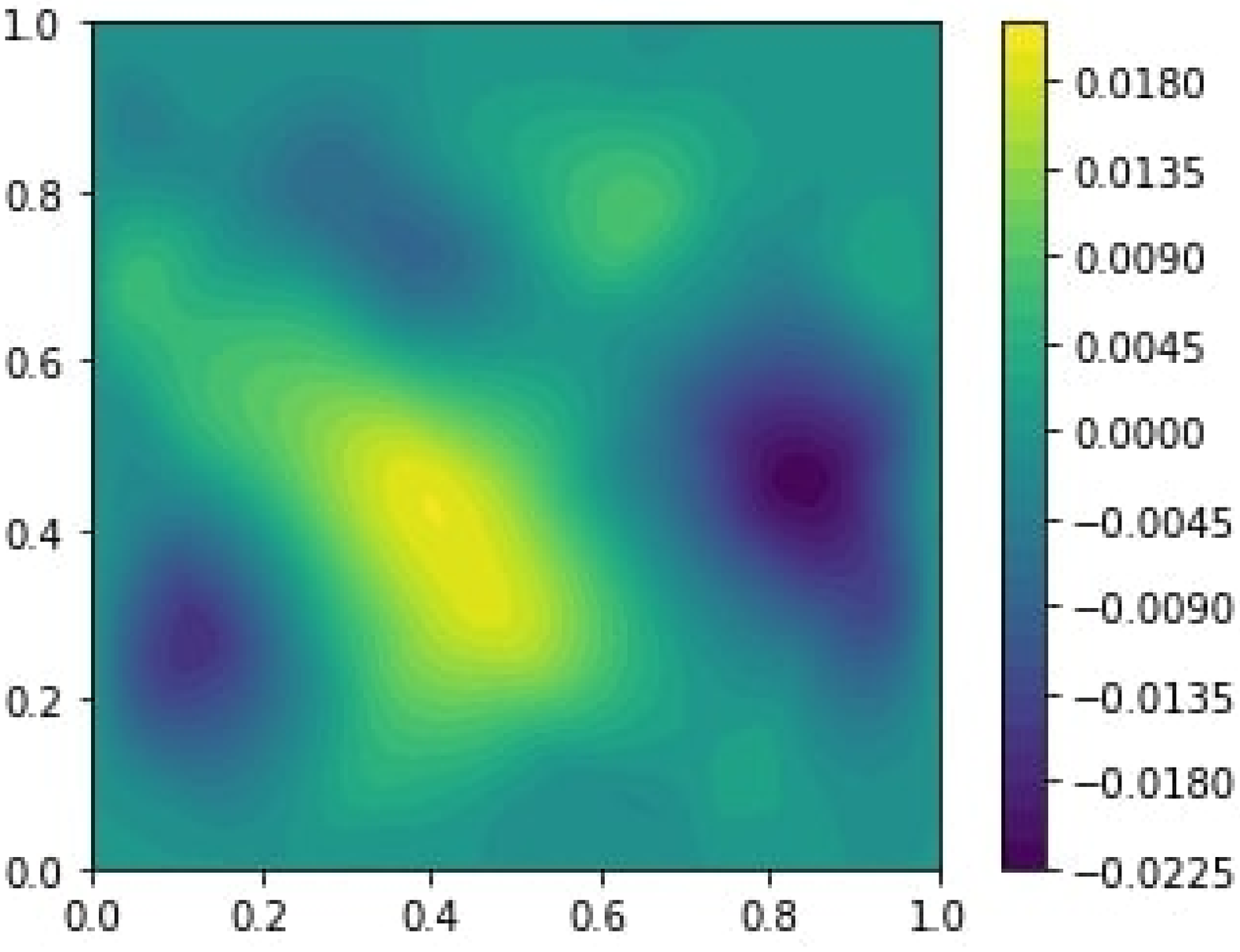}
     	}
    	\subfigure[$\kappa^{-1}=10^3$ in the purple region]{
	    	\label{vuggy3_velocity2}
	    	\centering
	    	\includegraphics[width=2.0in]{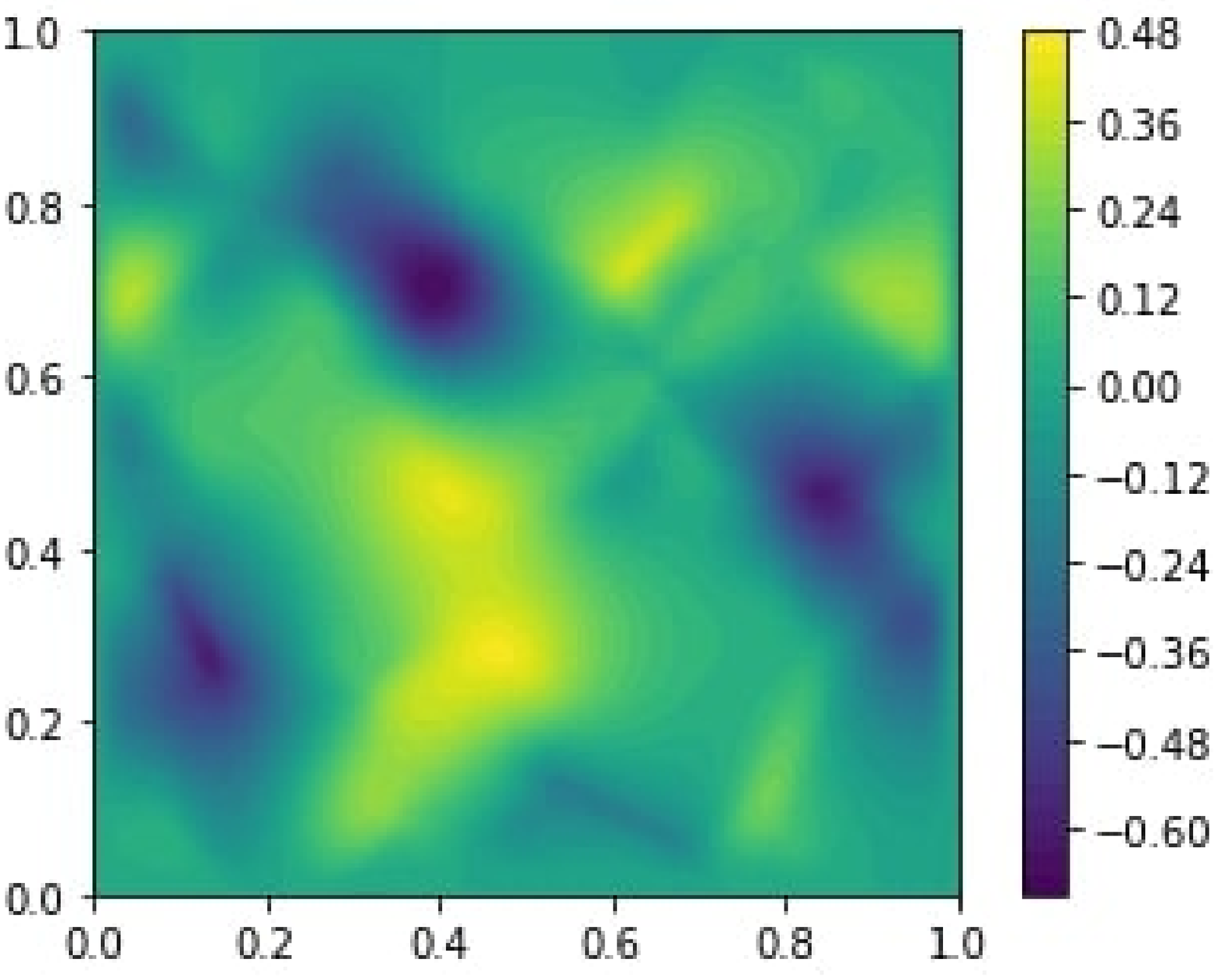}
    	}
		\subfigure[$\kappa^{-1}=10^5$ in the purple region]{
			\label{vuggy6_velocity2}
			\centering
			\includegraphics[width=2.01in]{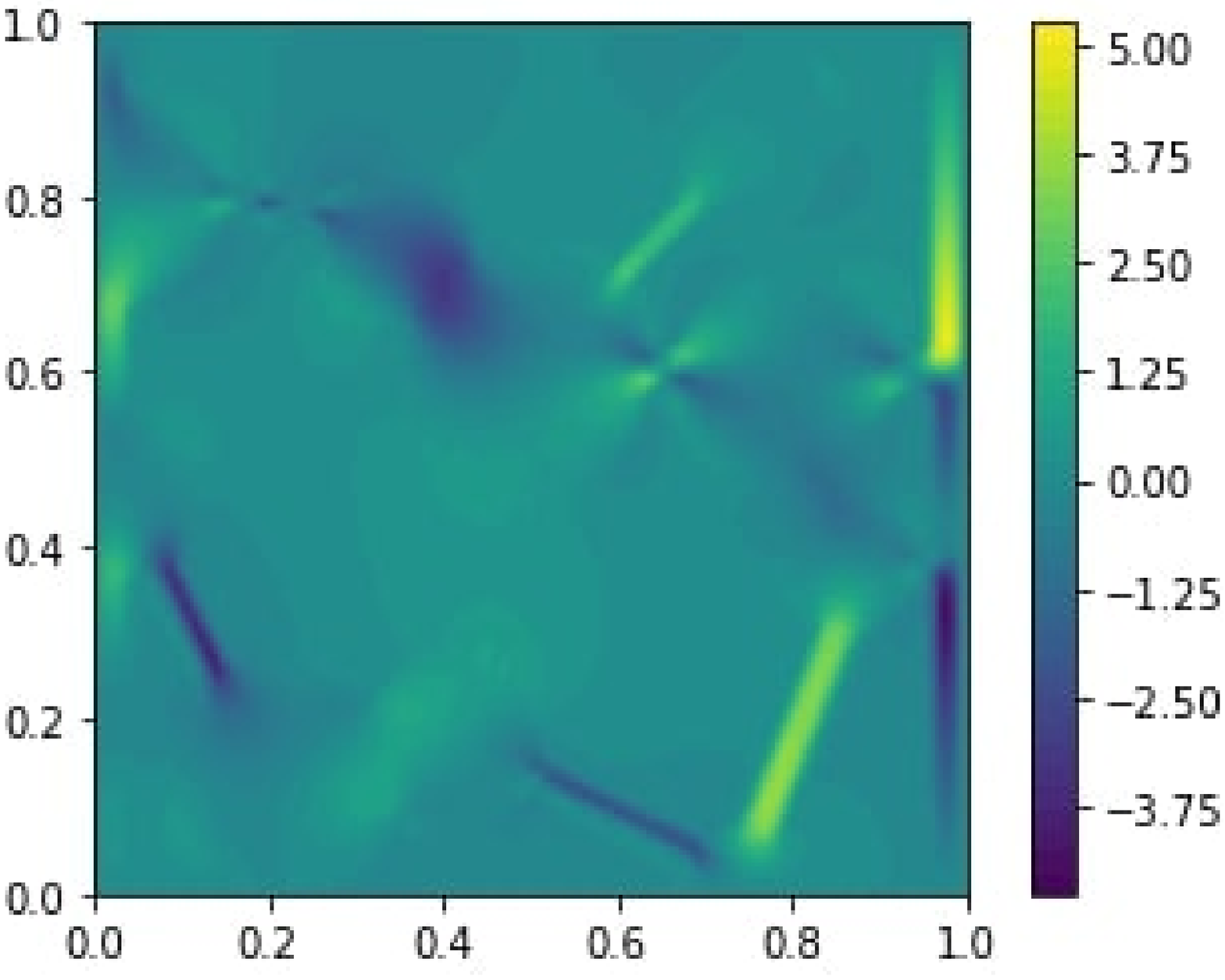}
		}
	\end{center}
	\vspace*{-15pt}
	\caption{Distributions of $u_2$ with $\kappa^{-1}=1$ in yellow region in Example \ref{example_4}.} \label{vuggy_u2}
\end{figure}

\begin{figure}[!ht]
	\begin{center}
    	\subfigure[$\kappa^{-1}=10$ in the purple region]{
		    \label{vuggy1_stress}
		    \centering
		    \includegraphics[width=2.06in]{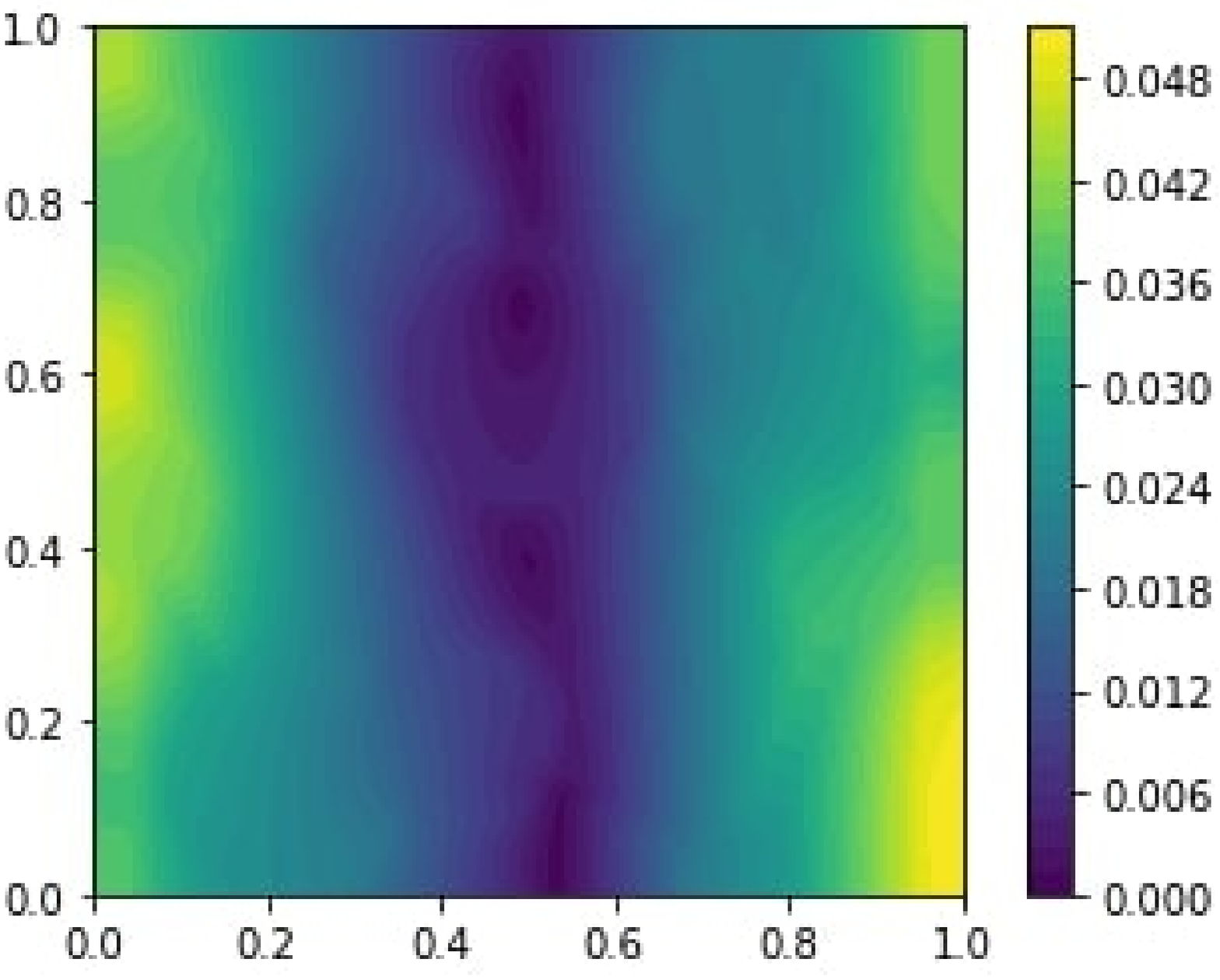}
     	}
    	\subfigure[$\kappa^{-1}=10^3$ in the purple region]{
	    	\label{vuggy3_stress}
	    	\centering
	    	\includegraphics[width=1.98in]{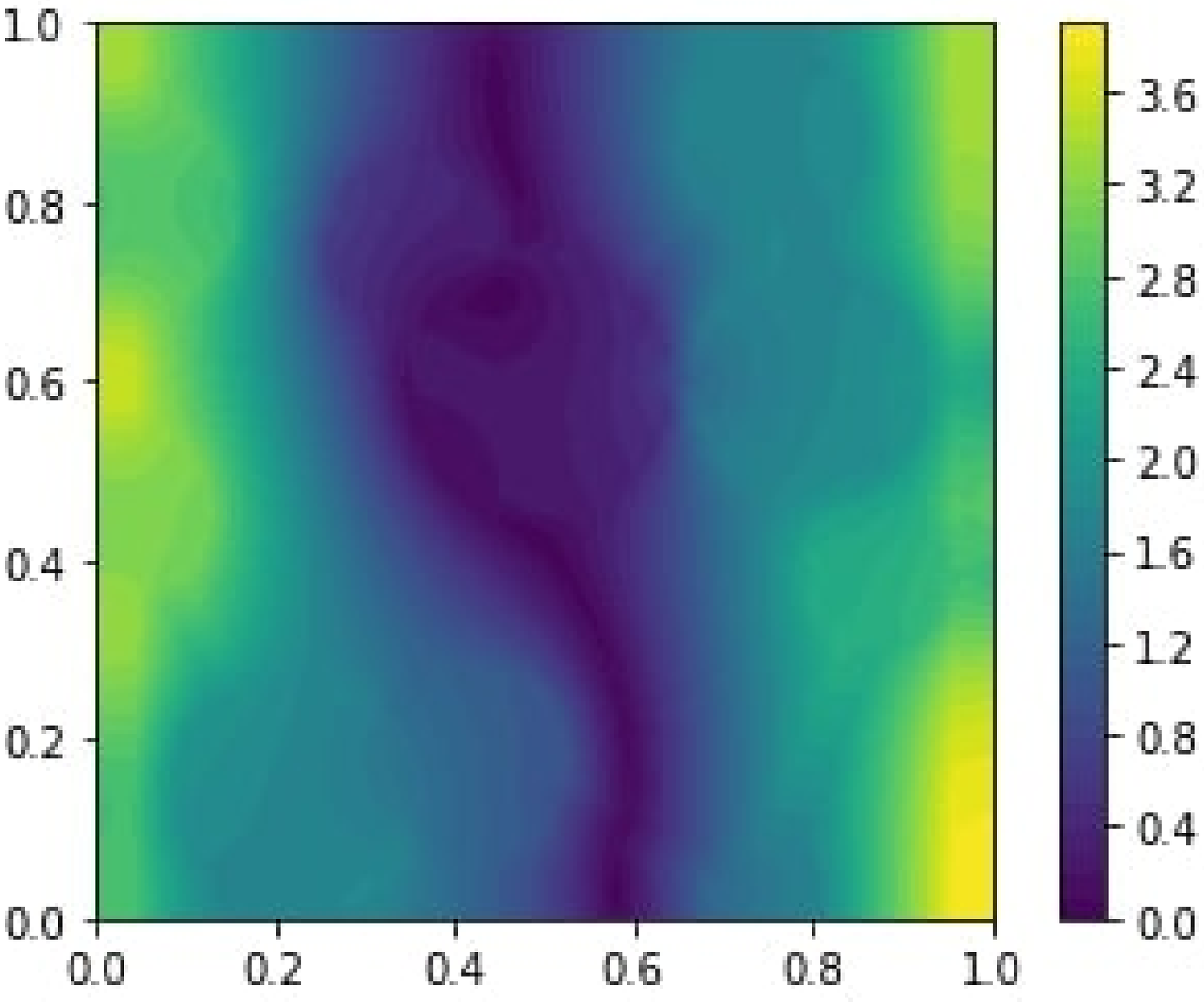}
    	}
		\subfigure[$\kappa^{-1}=10^5$ in the purple region]{
			\label{vuggy6_stress}
			\centering
			\includegraphics[width=2.00in]{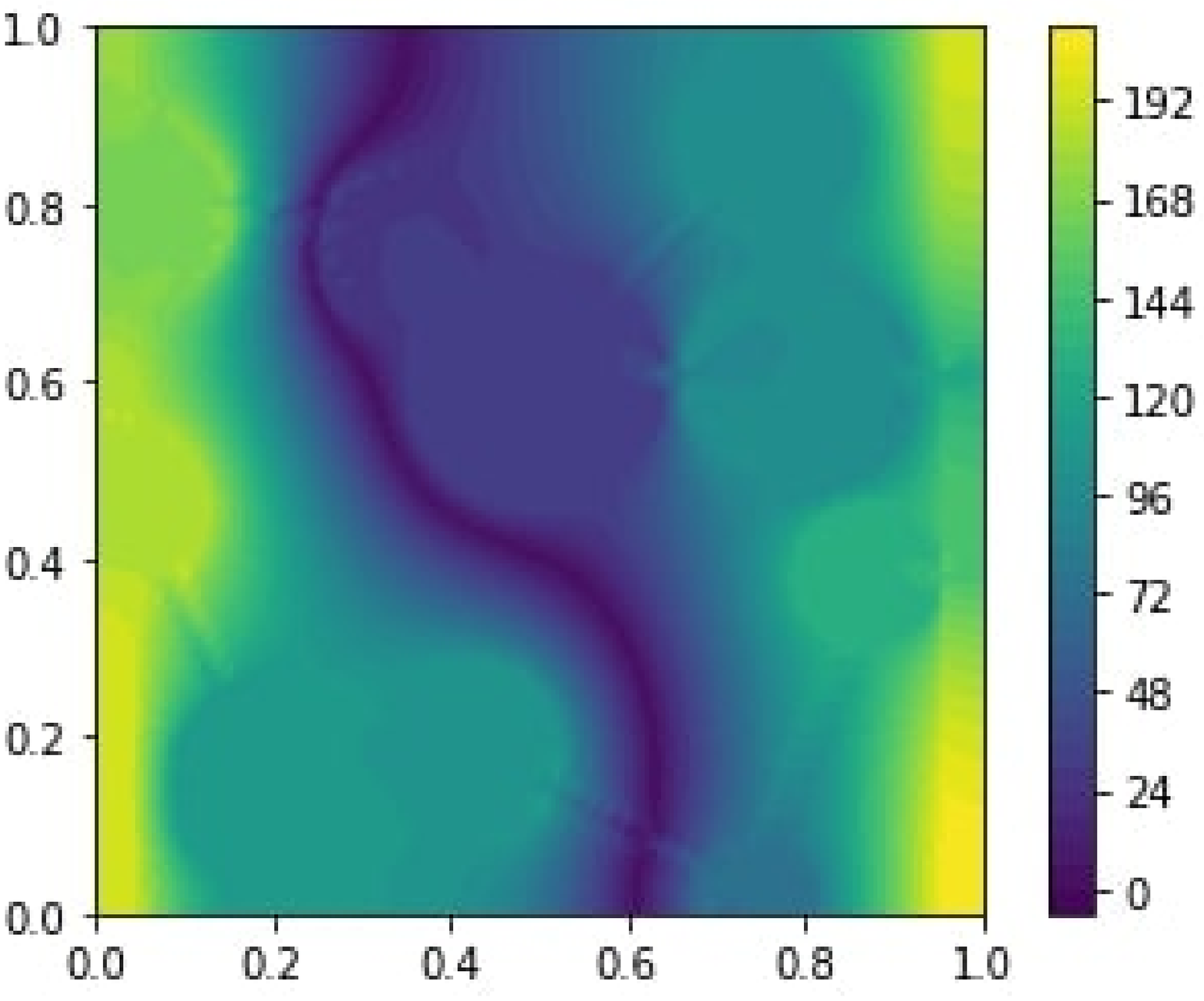}
		}
	\end{center}
	\vspace*{-15pt}
	\caption{Distributions of $\underline{\bm{\sigma}}$ with $\kappa^{-1}=1$ in yellow region in Example \ref{example_4}.} \label{vuggy_s}
\end{figure}

\begin{figure}[!ht]
	\begin{center}
    	\subfigure[$\kappa^{-1}=10$ in the purple region]{
		    \label{vuggy1_pressure}
		    \centering
		    \includegraphics[width=2.10in]{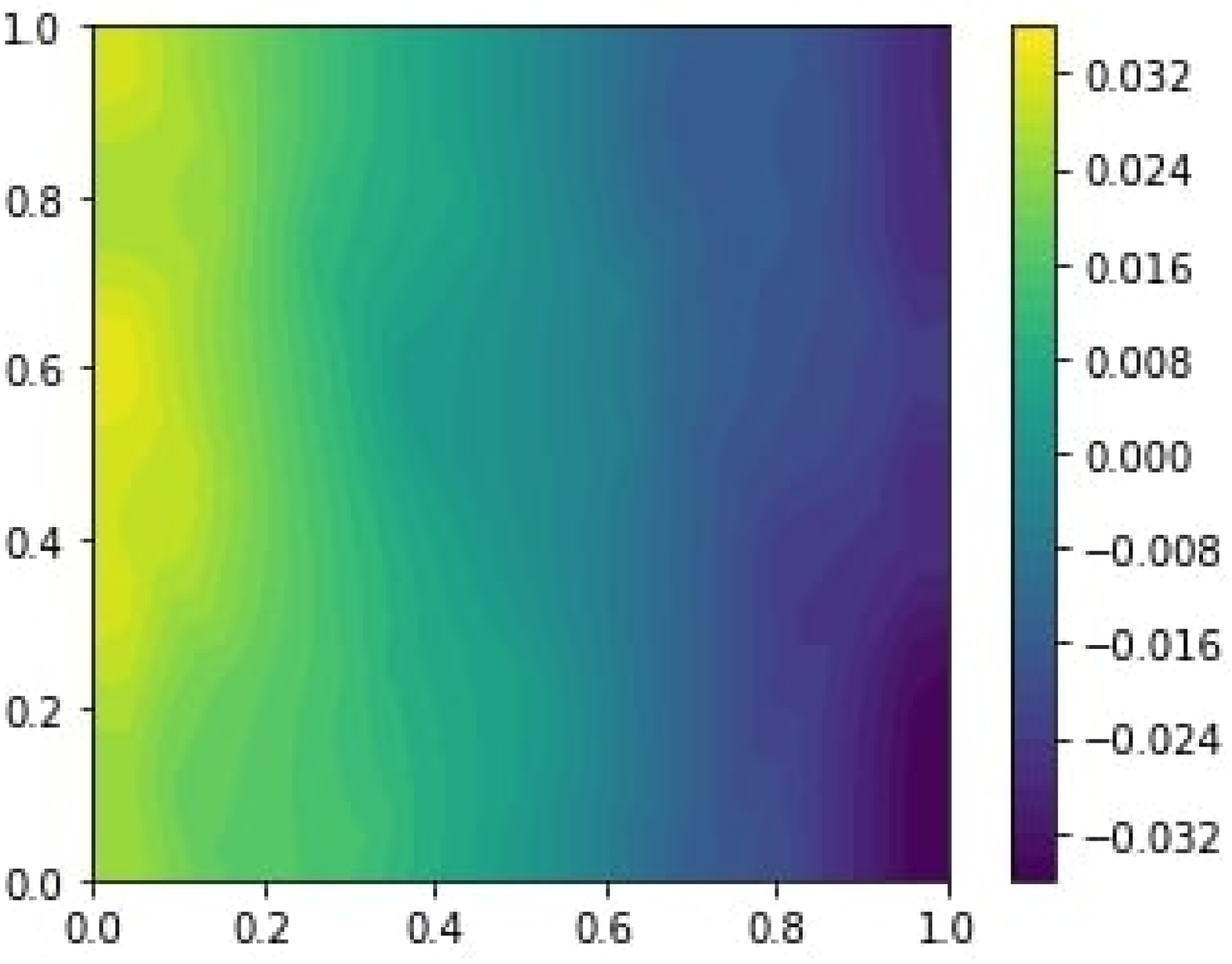}
     	}
    	\subfigure[$\kappa^{-1}=10^3$ in the purple region]{
	    	\label{vuggy3_pressure}
	    	\centering
	    	\includegraphics[width=2.0in]{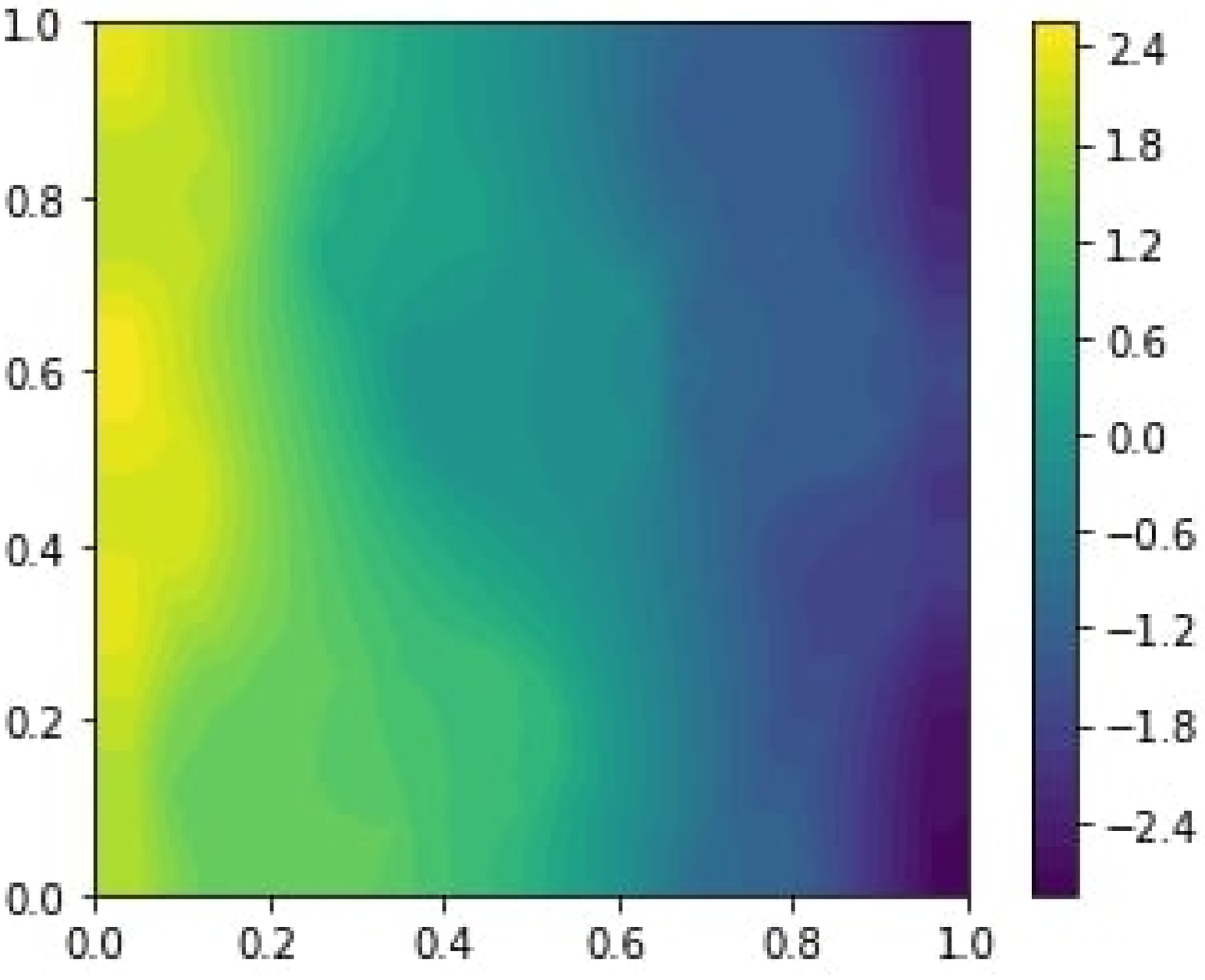}
    	}
		\subfigure[$\kappa^{-1}=10^5$ in the purple region]{
			\label{vuggy6_pressure}
			\centering
			\includegraphics[width=2.02in]{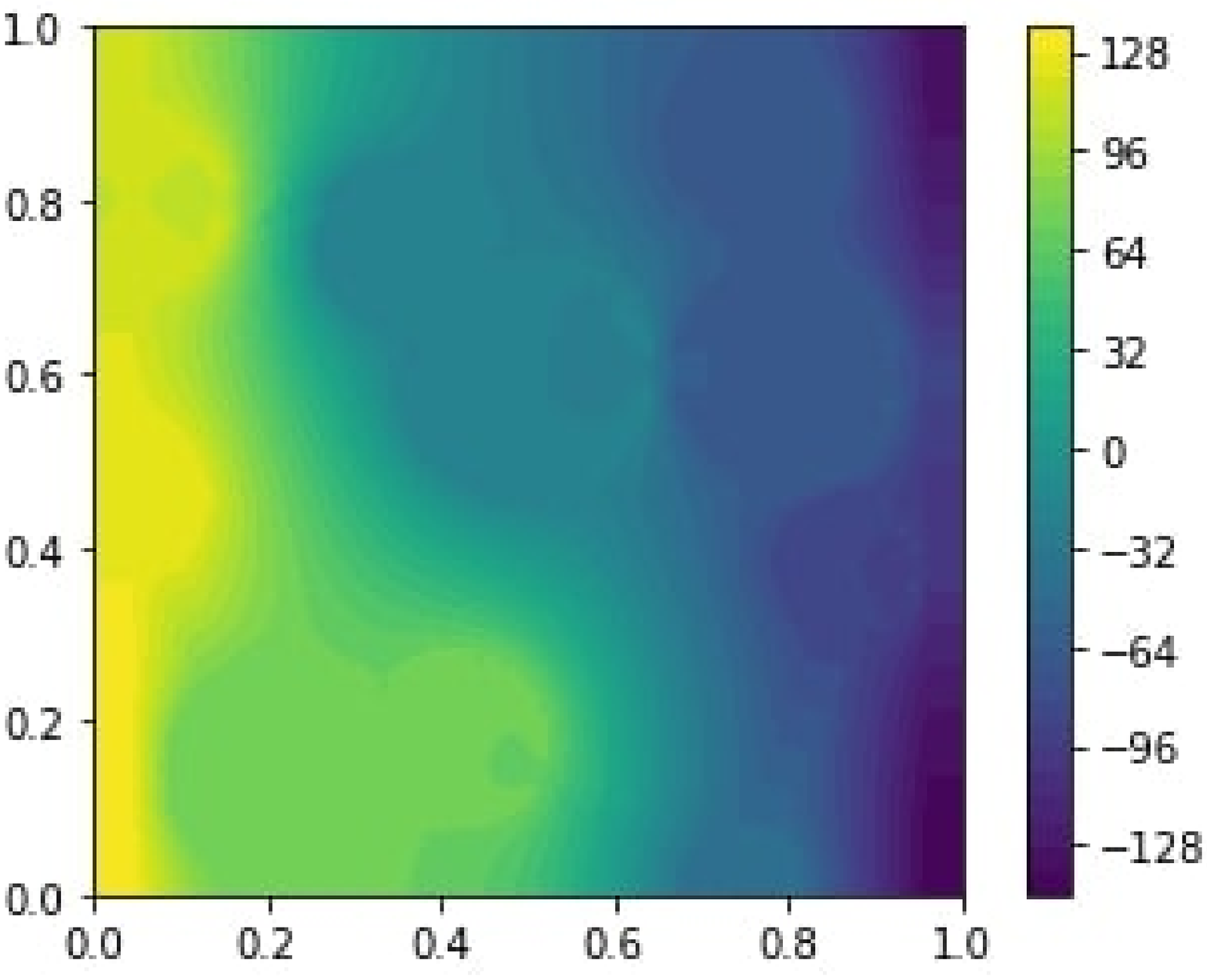}
		}
	\end{center}
	\vspace*{-15pt}
	\caption{Distributions of $p$ with $\kappa^{-1}=1$ in yellow region in Example \ref{example_4}.} \label{vuggy_p}
\end{figure}


\section{Summary}\label{summary}

In this paper, the mixed discontinuous Galerkin method with
$\underline{\bm{\mathcal{P}}}^{\mathbb{S}}_{k+1}$-$\bm{\mathcal{P}}_{k}$
element pair is constructed and studied for solving the Brinkman
equations based on the pseudostress-velocity formulation. The
well-posedness of the MDG scheme is proved by the generalized Brezzi theory,
and a priori error analysis is established. For any $k \geq 0$, we
prove the optimal convergence order for the stress in broken
$\underline{\bm{H}}(\mathbf{div})$ norm and velocity in $\bm{L}^2$
norm. Furthermore, the $\underline{\bm{L}}^2$ error estimate for the
pseudostress is also investigated under certain conditions.  Numerical 
examples confirm the theoretical results. In summary,
the proposed MDG method for Brinkman equation has following main
advantages: (i) it is uniformly stable and efficient from the Darcy
limit to the Stokes limit; (ii) it provides accurate approximation to
both the symmetric stress and the velocity; (iii) it is locally
conservative for the physical quantities.

\bigskip

\noindent{\bf Acknowledgments.}  We thank Professor Suchuan Dong (Purdue 
University) for the helpful discussions.

\bibliographystyle{elsarticle-harv}
\bibliography{yx_ref}

\end{document}